\documentclass[review,onefignum,onetabnum]{siamart_arXiv}

\usepackage{lipsum}
\usepackage{amsfonts}
\usepackage{graphicx}
\usepackage{epstopdf}
\usepackage{algorithmic}
\usepackage{sidecap}
\usepackage{graphics}
\usepackage{adjustbox}
\usepackage{amsmath}
\usepackage{subcaption}
\usepackage{amsfonts}
\ifpdf
  \DeclareGraphicsExtensions{.eps,.pdf,.png,.jpg}
\else
  \DeclareGraphicsExtensions{.eps}
\fi

\newcommand{\ip}{${\rm IP_3}$}
\newcommand{\jipr}{J_{\text{IPR}}}
\newcommand{\jserca}{J_{\text{SERCA}}}

\newcommand{\hinf}{h_{\infty}}
\newcommand{\kipr}{k_{\text{IPR}}}
\newcommand{\kbar}{K}

\newcommand{\tjipr}{\tilde{J}_{\text{IPR}}}
\newcommand{\tjserca}{\tilde{J}_{\text{SERCA}}}

\newcommand{\thinf}{\tilde{h}_{\infty}}
\newcommand{\tkipr}{\tilde{k}_{\text{IPR}}}
\newcommand{\tkbar}{\tilde{K}}

\newsiamremark{remark}{Remark}
\newsiamremark{hypothesis}{Hypothesis}
\crefname{hypothesis}{Hypothesis}{Hypotheses}
\newsiamthm{claim}{Claim}

\newsiamremark{rem}{Remark}
\newsiamremark{note}{Note}

\headers{Process-Oriented GSPT and Calcium Dynamics}{S.~Jelbart, N.~Pages, V.~Kirk, J.~Sneyd and M.~Wechselberger }
\title{Process-Oriented Geometric Singular Perturbation Theory and Calcium Dynamics}
\author{Samuel Jelbart\thanks{Department of Mathematics, University of Munich, Garching. \email{jelbart@ma.tum.de}, partially supported through the SFB/TRR 109 Discretization and Geometry in Dynamics grant and the Australian Research Council grant DP180103022.}
\and Nathan Pages\thanks{Department of Mathematics, University of Auckland, Auckland. \email{npag780@aucklanduni.ac.nz, v.kirk@auckland.ac.nz, sneyd@math.auckland.ac.nz}, partially supported by grant 15-UOA-184 of the Marsden Fund of the Royal Society of New Zealand.}
\and Vivien Kirk\footnotemark[2]
\and James Sneyd\footnotemark[2]
\and Martin Wechselberger\thanks{School of of Mathematics \& Statistics, University of Sydney, Sydney. \email{martin.wechselberger@sydney.edu.au}, partially supported through Australian Research Council grant DP180103022.}}
\usepackage{amsopn}

\ifpdf
\hypersetup{
	pdftitle={Process-Oriented GSPT and Calcium Dynamics},
	pdfauthor={S.~Jelbart, N.~Pages, V.~Kirk, J.~Sneyd and M.~Weschelberger}
}
\fi

\usepackage{graphicx}

\usepackage{fullpage}

\newtheorem{assumption}{Assumption}[section]

\usepackage{comment}
\usepackage{color}
\newcommand\wm[1]{{\color{black}{#1}}}
\newcommand\we[1]{{\color{black}{#1}}}
\newcommand\SJ[1]{{\color{black}{#1}}} % Sam
\usepackage{todonotes}
\newcommand\rev[1]{{\color{black}{#1}}} % Sam
\newcommand\revsj[1]{{\color{black}{#1}}} % Sam
\newcommand\com[1]{{\color{black}{#1}}} % Sam
\nolinenumbers

\begin{document}
	
	\maketitle
	
	\begin{abstract} 
		\rev{Phenomena in chemistry, biology and neuroscience are often modelled using ordinary differential equations (ODEs) in which the right-hand-side is comprised of terms which correspond to individual `processes' or `fluxes'. Frequently, these ODEs are characterised by multiple time-scale phenomena due to order of magnitude differences between contributing processes and the presence of \textit{switching}, i.e.,~dominance or sub-dominance of particular terms as a function of state variables. We outline a heuristic procedure for the identification of small parameters in ODE models of this kind, with a particular emphasis on the identification of small parameters relating to switching behaviours. This procedure is outlined informally in generality, and applied in detail to a model for intracellular calcium dynamics characterised by switching and multiple (more than two) time-scale dynamics. A total of five small parameters are identified, and related to a single perturbation parameter by a polynomial scaling law based on order of magnitude comparisons. The resulting singular perturbation problem has a time-scale separation which depends on the region of state space. W}e prove the existence \rev{and uniqueness} of \rev{stable} relaxation oscillations with three distinct time-scales % in a certain parameter regime, 
		using a coordinate-independent formulation of GSPT in combination with the blow-up method. \rev{We also provide an estimate for the period of the oscillations, and consider a number of possibilities for their onset under parameter variation.}
	\end{abstract}
	
	% REQUIRED
	\begin{keywords} intracellular calcium dynamics, non-standard geometric singular perturbation theory, switching, multiple time-scales, relaxation oscillations, blow-up
		%  example, \LaTeX
	\end{keywords}
	
	% REQUIRED
	\begin{AMS}
		% 68Q25, 68R10, 68U05
		34C15, 34C26, 34E15, 37N25, 92C37
	\end{AMS}

	\section{Introduction} % (fold)
	\label{sec:introduction}
	
%	{\color{red}{VK: Need to make a global decision about whether to refer to `state space' or `phase space'. Both are used and appear to be intrchangeable. I'd prefer phase space, but don't care much.}}
	
	\rev{
	Chemical, biological and physiological phenomena are often modelled by systems of ordinary differential equations (ODEs) which contain multiple time-scales. This is a natural consequence of the fact that the underlying chemical, biological or physiological processes relevant to the dynamics naturally operate on time-scales which can span several orders of magnitude. Singular perturbation theory provides a mathematical toolbox for the analysis of multi-scale problems of this kind. In particular, many authors have demonstrated the suitability of \textit{geometric singular perturbation theory (GSPT)} \cite{fenichel1979geometric,Kuehn2015,Jones1995,wechselberger2018geometric} in combination with a method of geometric desingularization known as \textit{blow-up} \cite{Dumortier1996,Krupa2001a,Krupa2001b} as a rigorous and geometrically informative framework for the analysis of such problems; see \cite{Gucwa2009,kosiuk2011scaling,milik2001} for examples in mathematical chemistry, \cite{deMaesschalck2015,Rubin2002,Mitry2013,Roberts2015,Vo2013} for examples in mathematical neuroscience, and \cite{kosiuk2016geometric,Kristiansen2019d,KuehnSzmolyan2015} for examples in the context of biological and physiological systems.
	
	Despite the success of these analyses, however, for many multi-scale ODEs arising in applications there are still a number of significant modelling and analytical obstacles to be overcome. For example, a necessary pre-condition for the formulation of an ODE as a perturbation problem, and hence for an analysis via GSPT, is the identification of a suitable perturbation parameter $\epsilon \ll 1$. In many systems of interest, there is no explicit perturbation parameter $\epsilon$, and the modeller is confronted with the problem of how to introduce an $\epsilon$ without compromising the validity of the model. The opposite problem is also common, i.e.,~many applications feature multiple candidates for small parameters. In this case, the modeller is often forced to choose which (if any) of the candidate small parameters should be considered as \we{`relevant'} perturbation parameter\we{(s)}. In the case of more than one perturbation parameters, the ordering or relation between these `$\epsilon$'s' becomes significant for the validity and tractability of the resulting perturbation problem. Furthermore, models of this type are often characterised by additional analytical problems stemming from the presence of many time-scales \cite{kosiuk2011scaling,Kruff2019}, non-trivial time-scale separations depending on the region of state and/or parameter space \cite{wechselberger2018geometric,Lizarraga2020,kosiuk2016geometric,Gucwa2009,Kruff2019,Goeke2014}, and `switching', i.e.,~convergence to a non-smooth singular limit as a perturbation parameter tends to zero \cite{kosiuk2016geometric,Kristiansen2019d,Plahte2005,Ironi2011,Machina2013,Glass2018}. 
	
	In recent years, significant progress towards overcoming these analytical obstacles has been made by many authors. In \cite{wechselberger2018geometric}, a coordinate-independent framework for GSPT which is applicable for systems with time-scale separations depending on the region of state space is outlined in detail. Such a framework has also been developed and applied successfully to models for chemical reactions in, for example,~\cite{Goeke2012,Goeke2014}, and previous work by Szmolyan et al \cite{Gucwa2009,kosiuk2016geometric,Kristiansen2019d} demonstrates that GSPT and blow-up techniques are well suited to analyses of problems of this kind. In \cite{kosiuk2011scaling} the authors successfully analysed a model for glycolytic oscillations featuring multiple small parameters and more than two time-scales using GSPT and blow-up techniques. In \cite{Cardin2017} the GSPT framework was extended to account for any number of small parameters and time-scales, similar results in the more general case with no assumed separation of fast/slow variables have been developed and applied in \cite{Kruff2019}, and in \cite{Lizarraga2020b} progress is made towards the identification and analysis of so-called `hidden time-scales' in such systems. In \cite{kosiuk2016geometric} the authors demonstrate that GSPT and blow-up techniques can be adapted to study problems characterised by `switching', a common feature of ODE models in mathematical biology due, for example, to the rapid activation or inactivation of certain processes once a threshold agonist or chemical concentration is reached. 
	Finally, the authors in \cite{Kristiansen2019d} were able to show that blow-up techniques developed originally for the study of regularised (smoothed) piecewise-smooth systems \cite{Buzzi2006,Llibre2007,Llibre2009,Kristiansen2020,kristiansen2018a} can be used in order analyse these systems using singular perturbation theory.
	
	\

	%%%%%%%%%%%%%%%%%%%%%%%%%%%%%%%%%%%%%%%%%%%%%%%%%%
	\begin{figure}[t!]
		\centering
		\begin{subfigure}[b]{0.48\linewidth}
			\includegraphics[width=\linewidth]{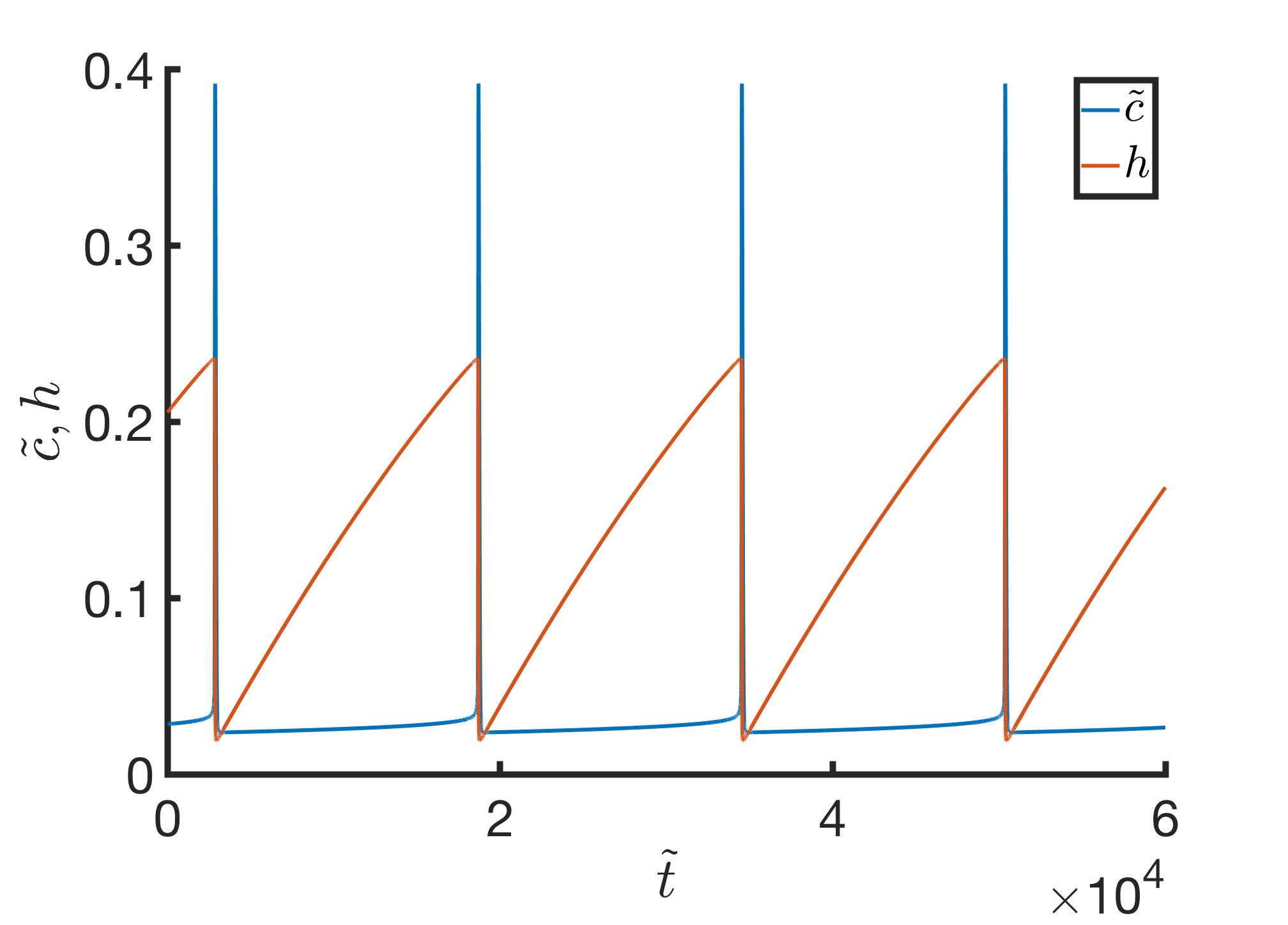}
		\end{subfigure}
		\begin{subfigure}[b]{0.48\linewidth}
			\includegraphics[width=\linewidth]{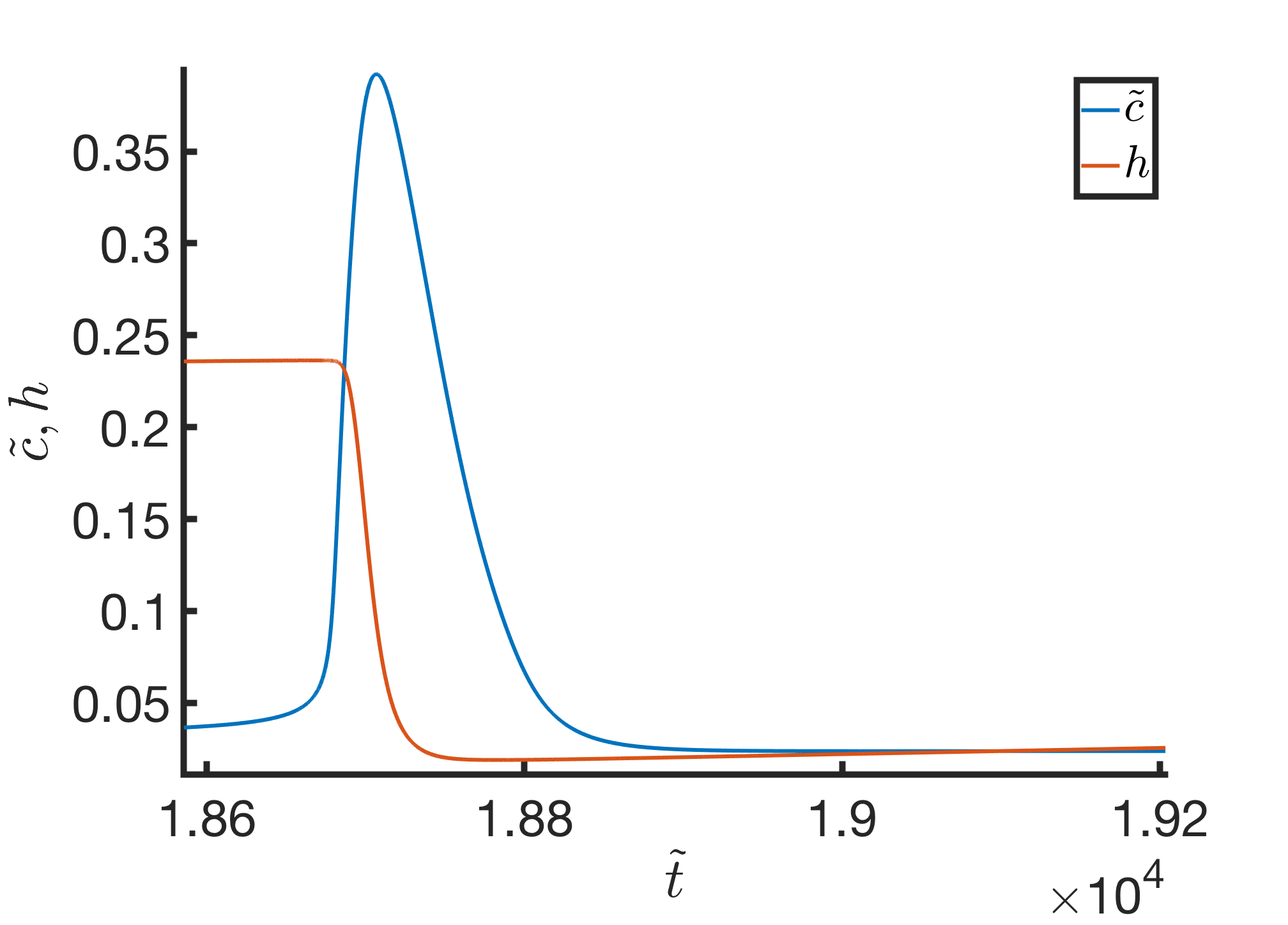}
		\end{subfigure}
		\caption{\rev{Relaxation oscillations in a model for intracellular calcium dynamics defined in Section \ref{sec:motivations_for_the_closed_cell_model}, system \eqref{eq:Non_dim}. Left: time series for the two state variables $\tilde c$ and $h$. Right: enlargement of one spike. Both $\tilde c$ and $h$ exhibit multiple time-scale dynamics, indicative of a non-standard time-scale separation.}}
		\label{fig:non_dim}
	\end{figure}
	%%%%%%%%%%%%%%%%%%%%%%%%%%%%%%%%%%%%%%%%%%%%%%%%%%%

	Our aim in this work is to develop a heuristic procedure for the derivation and analysis of such perturbation problems, focusing in particular on models without \we{{\em a priori} knowledge of relevant} perturbation parameters. \rev{In light of the significant progress toward overcoming the analytical problems associated with singular perturbation analysis outlined in the preceding paragraph, emphasis is placed here on the identification of small parameters and reformulation of a model as a singular perturbation problem.} This procedure is outlined in generality, but developed primarily via application to a model of intracellular calcium oscillations %, \revsj{obtained as 2-dimensional reduction of a higher dimensional model in} 
	\cite{sneyd2017dynamical}. Although the model is only two-dimensional, it features many of the obstacles discussed above and therefore serves as a suitable case study for the development of our methods. \rev{Figure \ref{fig:non_dim} shows} numerical evidence of relaxation oscillations with two or more time-scales in the model, with a time-scale separation which depends on the region of state space. This is evident by the presence of fast and slow components in both state variables. \rev{The model also exhibits} switching, i.e., sharp dynamical transitions caused by the fact that some of the model fluxes change quickly, in a threshold-like fashion, as the calcium concentration changes. \rev{Applying our heuristic procedure to the model, we identify multiple small parameters, which are then linked to a single small parameter via a suitable polynomial scaling based on order of magnitude comparisons. We then prove existence, uniqueness and stability of the observed relaxation oscillations in the singular perturbation problem derived via our methods, provide an estimate for the period of the oscillations, and consider a number of mechanisms for their onset under parameter variation.}

	}
	
	\
	
	The manuscript is structured as follows. \rev{In Section \ref{sec:the_procedure} we outline the general procedure for deriving a suitable singular perturbation problem.} In Section \ref{sec:motivations_for_the_closed_cell_model} we present the model from \cite{sneyd2017dynamical} and its reduction to the \rev{closed-cell} form we will analyze. In Section \rev{\ref{sec:steps_I_IV} we apply our procedure to the \we{(closed-cell)} calcium model, and derive the corresponding singular perturbation problem.} Section \ref{sec:slow-fast_analysis_and_statement_of_the_main_result} is devoted to a GSPT analysis of this system, including the statement of our main results. The detailed blow-up analysis required to prove existence %and uniqueness 
	of the relaxation oscillations, \rev{as well as the bifurcation analysis of the model,} is deferred to the Appendix \rev{for expository reasons}. %\SJ{In Section \ref{sec:effect_of_tau_} we consider possibilities for the onset of the relaxation oscillations under parameter variation. %\wm{[update/change:] the effect of varying a parameter which controls} the frequency of oscillations, and identifies a mechanism for their onset. 
	In Section \ref{sec:discussion_conclusion} we summarise and conclude.

	\section{\rev{Outline of the method}}
	\label{sec:the_procedure}
	
%	\rev{Our aim is to derive a perturbation problem describing the dynamics of the model \eqref{eq:Full_model}. This is non-trivial task due the absence of explicit small parameters. Moreover, due to switch-like terms such as $\tau_h(c)$ (see again Figure \ref{fig:tau_h_fig}), the multiple-scales behaviour is expected to depend in non-trivially on the region of state space. We propose to account for both issues by via a `semi-systematic' approach which we believe to be applicable and adaptable for a large class of low dimensional ODE models in which the right-hand-side is comprised of linear combinations of identifiable process/flux terms as is common in e.g.~mathematical biology and neuroscience. The approach is sketched in the following:
		
%		\smallskip
		
		\rev{In the following we outline our heuristic procedure for deriving singular perturbation problems from ODE models characterised by multi-scale and/or switching behaviour.}
			%\revsj{Here is it worthy to remark that although there is nothing to prevent one from applying the procedure to ODE models in general, it is developed and intended for application to multi-scale problems. For many problems, multi-scale and/or switching dynamics are `obvious' and readily observed in simulations. In Section \ref{sec:motivations_for_the_closed_cell_model} we shall see that this is true for the calcium model which we consider in detail. Such problems are natural candidates for the application of our approach. In cases where for which multi-scale structure or switching behaviour is less obvious, we believe our approach to be helpful for identifying `hidden' multi-scale or switching structure, if there is any to be found.}}
		
		\

		\rev{
			%The method is as follows. 
			Given a model expressed as a system of ODEs comprised of identifiable process/flux terms, perform the following steps:
		
		\smallskip 
		
		\begin{enumerate}
			\item[(I)] \textit{\we{Non-dimensionalise}.}
			%\item[(I)] \textit{Non-dimensionalise and simulate.} %VK: Changed the name of the step to match that used in Section 4
			This is crucial for revealing the relative magnitudes and time-scales associated with particular process/flux terms. %After non-dimensionalisation. % look for signs of multiple scale and/or switching dynamics numerically. % \com{Maybe we should make this step ``Non-dimensionalise and simulate"?}
			\item[\we{(IIa)}] \textit{Associate small parameters to \we{maximal process/flux rates.}} Following step (I), each process/flux term can be written as the product of a constant scaling factor and a normalised process/flux term. % Do so and compare the relative magnitudes of the constant scaling factors. % associated with each process/flux term. %Due to step (I), the largest scaling factor (and thus time-scale) should be of numerical order $1$. 
			Numerically `small' scaling factors constitute candidates for small parameters $\epsilon_i \ll 1$.
			\item[\we{(IIb)}] \textit{Associate small parameters to \we{steep switches in process/flux terms.}} Identify sources of \we{`steep'} switching and associated candidate small parameters $\epsilon_i \ll 1$ such that %. For each such $\epsilon_i$, 
			the limit $\epsilon_i \to 0$ yields an approximation by a non-smooth switch.
			% Identify regions of state space correlated to distinct dynamical behaviours, and compare relative magnitudes of the normalised process/flux terms in these regions. 
		%	The decision as to whether a switch is `sufficiently sharp' for the validity of non-smooth approximations may be informed by a combination of analytical and numerical arguments. % based on the slope of the switch, or numerical observations based on e.g.~heat maps for each normalised process/flux term. 
			%The presence of sharp/abrupt switching behaviour in (III) motivates the identification of candidate small parameters $\epsilon_i \ll 1$, such that the limit $\epsilon_i \to 0$ yields an approximation by a non-smooth switch.
			%	\item[(IV)] Associate candidate small parameters to `sufficiently sharp' switches. The presence of sharp/abrupt switching behaviour in (III) motivates the identification of candidate small parameters $\epsilon_i \ll 1$, such that the limit $\epsilon_i \to 0$ yields an approximation by a non-smooth switch.
			\item[\we{(III)}] \textit{Relate small parameters.} Use numerical order of magnitude comparisons to relate candidate small parameters identified in \we{(IIa) and/or (IIb)}. Choose a polynomial scaling with respect to a single small parameter $\epsilon$ as follows: for $i = 1, \ldots , m$ small parameters $\epsilon_i$ we have
			\begin{equation}
				\label{eq:scaling_polynomial}
				\epsilon = a_1 \epsilon_1^{b_1} = a_2 \epsilon_2^{b_2} = \ldots = a_m \epsilon_m^{b_m} ,
			\end{equation}
			where $\epsilon_i$ are the candidate small parameters from \we{(IIa) and (IIb)}, the $a_i$ are of numerical order $1$ and the exponents $b_i > 0$. There should exist a fixed numerical value of \we{$\epsilon\ll 1$} such that the original (dimensionless) parameter values, i.e.,~the actual parameter values prior to the scaling \eqref{eq:scaling_polynomial}, are returned by setting $\epsilon$ equal to this number.
			%	\item[(V)] Test for loss of smoothness. The existence of candidate small parameters linked to switching behaviour implies a loss of smoothness in the (joint) limit $\epsilon \to 0$. Determine the order at which this loss of smoothness occurs. %This will play a role in determining which analytical methods are available for analysis. % occurs is important for analytical purposes due to small parameters identified in (III)-(IV), the system will lose smoothness at a certain order. 
			%	If the system remains at least $C^{r \geq 1}$ as $\epsilon \to 0$, it can be analysed using perturbation theory, GSPT and blow-up techniques as necessary. If the system is singular in the sense that it loses $C^1$ or even $C^0$ smoothness as $\epsilon \to 0$, the analysis can be supplemented by additional tools from piecewise-smooth theory, in combination with adaptations of the blow-up method for singular perturbations of non-smooth systems.
			\item[\we{(IV)}] \textit{Analyse the system via perturbation theory.} If successful, steps \we{(I)-(III)} yield a well-defined system which can be analysed via the limiting (potentially non-smooth) dynamics.
			% Note that step (V) is important in determining which analytical methods are available for analysis. %If successful, steps (I)-(VI) yield a well-defined system which can be analysed via the limiting (potentially non-smooth) dynamics. Typically, limiting dynamics in distinct regions of state space can be analysed with established tools from (nonstandard) GSPT and PWS theory. Often these regions will be distinguished by different $\epsilon-$rescaled coordinates. The connection/overlap between these regions can often be studied in a rigorous way using blow-up techniques.
		\end{enumerate}
		
		\smallskip
		
		Steps \we{(I)-(III)} provide a semi-systematic \we{pre-processing of any dimensional (ODE) model for deriving a candidate perturbation problem}, and will be carried out in detail for a model for intracellular calcium dynamics in Section \ref{sec:steps_I_IV}. Step \we{(IV)}, the analysis of the resulting perturbation problem, is presented for this case study in Section \ref{sec:slow-fast_analysis_and_statement_of_the_main_result}. % and the Appendices. 
	
		\we{In the following we provide several remarks on  the general procedure outlined above.}
		%We close this section with a number of remarks relating to the general procedure outlined above.	
		\begin{remark}
			\rev{The procedure is developed and intended for application to ODE models exhibiting multi-scale and/or switching dynamics, and is expected to be of most use in this particular setting. In the case that a model has no clear multi-scale or switching behaviour, the procedure should output either a `null result' in which no perturbation parameters are identified, or a regular perturbation problem, in which only small parameters that are not related to multi-scale structure or switching are identified. In the latter case the procedure still provides a simplification, even in the absence of multi-scale dynamics. \we{In case of a singular perturbation problem, small parameters are identified in (IIa) only, in (IIb) only or in (IIa) and (IIb).}}
		\end{remark}
		
		\begin{remark}
			\label{rem:sufficiently_small}
			As always in applied settings, what constitutes a `sufficiently small' numerical value for a parameter to be considered as a perturbation parameter in step \we{(IIa)} is debatable, and left to the discretion of the modeller. Typically, assertions that a particular value is `sufficiently small' for the application of perturbation methods need to be justified in terms of agreements with numerical and/or experimental findings.
		\end{remark}
		
		\begin{remark}
			\label{rem:sufficiently_sharp}
			Similarly to the issues described in Remark \ref{rem:sufficiently_small} above, the question of whether or not a switch is `sufficiently \we{steep}' to justify approximation by a non-smooth switch in step \we{(IIb)} is left to the discretion of the modeller. A quantitative comparison can be seen by noting that in this case, the slope of the normalised switch should be `sufficiently large' for the validity of such approximations. Such approximations should be motivated on numerical, experimental or analytical grounds (or a combination thereof), and subsequently justified in terms of agreement with numerical and/or experimental findings.
		\end{remark}
		
		%\begin{remark}
		%	In some cases the small parameter needed to formulate an approximate by a non-smooth switch will be given, e.g.~Hill functions. In others, it must be introduced. \com{Explain}
		%\end{remark}

		\begin{remark}
			There are many ways in which one could choose to order or relate the identified small parameters. The choice to do so in a polynomial fashion as in step (IV) equation \eqref{eq:scaling_polynomial} is made for simplicity, and in some sense defines our method. It is also important to keep in mind for a given application whether some limits need to remain independent.
		\end{remark}
		
		\begin{remark}
			A singular perturbation problem with small parameters identified in step (IIa) may or may not be given in the so-called `standard form', i.e~it may or may not feature a separation of fast/slow variables. 
		%
		%	The presence of switching behaviour means that the resulting perturbation problem is likely to be singularly perturbed with a multiple-scales structure which can vary throughout the state space. Consequently, the resulting singular perturbation problem will typically \textit{not} be given in a standard form with a %arise in a \textit{non-standard form} which does not admit a 
		%	global separation into fast/slow variables \cite{wechselberger2018geometric}.
		%	
			The identification of one or more perturbation parameters associated to switching in step \we{(IIb)}, implies that the resulting perturbation problem loses smoothness at some order in the singular limit $\epsilon \to 0$. In this case, the system is generically `non-standard', in the sense that it cannot be written in the standard form.
		%	\we{[Comment for Sam: I do not agree with that remark (or I misread it) since it merges two different observations. (i) non-standard and (ii) non-smooth, which are not equivalent ((ii)$\to$(i), but not (i)$\to$(ii)). I assume we mean by  `switch' what is described in (IIb), e.g. a singular non-smooth limit. Case (i) can have a nice smooth rhs and still be non-standard. So, what do you exactly mean by `presence of switching behaviour'?]}
		\end{remark}
	}

	%\section{Physiological background and the model} % (fold)
	\section{\rev{The model}}
	\label{sec:motivations_for_the_closed_cell_model}
	
	\rev{In this section we introduce the model for intracellular calcium dynamics which will serve as a non-trivial case study for our approach.}	
	%Let $c$ denote the calcium concentration in the cytosol, $c_e$ the calcium concentration in the ER, and $h$ a gating variable which represents the capacity of the \ip{} receptors to allow the passage of calcium.
	We consider the following model developed in \cite{sneyd2017dynamical}:
	\begin{equation}
		\begin{split}
			\displaystyle\tau_h(c)\frac{\text{d}h}{\text{d}t}&=\displaystyle\hinf (c) - h,\\
			\displaystyle\frac{\text{d}c}{\text{d}t}&=\displaystyle\jipr \rev{(h,c,c_e)}-\jserca \rev{(c,c_e)} + J_{\rm pm} \rev{(c)},\\
			\displaystyle\frac{\text{d}c_e}{\text{d}t}&=\displaystyle\gamma(\jserca \rev{(c,c_e)}-\jipr \rev{(h,c,c_e)}) ,
			%switched sign of RHS of c_e equation
			\label{eq:Full_model_open}
		\end{split}
	\end{equation}
	\rev{where $c$ denotes the calcium concentration in the cytosol, $c_e$ the calcium concentration in the endoplasmic reticulum (ER), and $h$ is a gating variable describing the capacity of the \ip{} receptors (IPR) to allow the passage of calcium. The particular form of the fluxes $\jipr(h,c,c_e)$, $\jserca(c,c_e)$, $J_{\rm pm}(c)$ and the functions $\hinf(c)$, $\tau_h(c)$ will be specified below.} The parameter $\gamma$ is the ratio between the volume of the cytosol and the volume of the ER. Default parameter values used in the model are specified in Table \ref{tab:params}. 
	
	%%%%%%%%%%%%%%%%%%%%%%%%%%%%%%%%%%%%%
	\begin{table}[t!]
		\center
		\begin{tabular}{|c|l|l||c|l|l|}
			\hline
			Parameter    & Value & Units    & Parameter & Value      & Units           \\
			\hline
			$k_\beta$    & 0.4   & -        & $p$       & 0.05       & $\mu$M          \\
			$K_c$        & 0.2   & $\mu$M   & $K_\tau$  & 0.1        & $\mu$M          \\
			$K_h$        & 0.08  & $\mu$M   & $V_s$     & 0.9        & $\mu$Ms$^{-1}$ \\
			$K_p$        & 0.2   & $\mu$M   & $\kbar$   & 0.00001957 & -               \\
			$\kipr$      & 10    & $s^{-1}$ & $K_s$     & 0.2        & $\mu$M          \\
			$\tau_{\rm max}$ & 1000   & s        & $\gamma$  & 5.5        & -               \\
			$c_t$        & 2     & $\mu$M   & -         & -          & -               \\
			\hline
		\end{tabular}
		\caption{Default (dimensional) parameter values for the model, equations (\ref{eq:Full_model}).}
		\label{tab:params}
	\end{table}
	%%%%%%%%%%%%%%%%%%%%%%%%%%%%%%%%%%
	
	It is useful to define $c_t := c + c_e / \gamma$, which represents the total moles of free calcium in the ER and the cytosol, divided by the cytoplasmic volume. This variable can be used as an alternative to $c_e$ in the model, yielding
	\begin{equation}
		\begin{split}
			\displaystyle\tau_h(c)\frac{\text{d}h}{\text{d}t}&=\displaystyle \hinf (c) - h,\\
			\displaystyle\frac{\text{d}c}{\text{d}t}&=\displaystyle\jipr \rev{(h,c,c_t)}-\jserca \rev{(c,c_t)} + J_{\rm pm} \rev{(c)},\\
			\displaystyle\frac{\text{d}c_t}{\text{d}t}&= J_{\rm pm} \rev{(c)} ,
			\label{eq:Full_model_open_ct}
		\end{split}
	\end{equation}
	\rev{where by replacing the argument $c_e$ by $c_t$ in the flux terms $\jipr\rev{(h,c,c_t)}$ and $\jserca\rev{(c,c_t)}$ we have permitted a slight abuse of notation. The} flux \rev{term $J_{\text{pm}}(c)$} represents the exchange of calcium between the cytosol and the extracellular medium. % \SJ{[Does this sentence change if we change Figure \ref{fig:Calcium_fluxes_diagram}?]}. VK: Think it is okay with either sign of J_pm
	This calcium \we{eflux} is typically slow compared to the rate of other calcium fluxes in the cell and is not necessary on either physiological \cite{bruce2002phosphorylation} or mathematical \cite{pages2019model} grounds in order for there to be oscillations. \we{Blocking the $J_{\text{pm}}(c)$ flux term obtained by setting $J_{\text{pm}}(c) \equiv 0$ gives the corresponding {\em closed-cell model}. The study of this reduced model} 
	% For example, there are cell types which exhibit calcium oscillations in the absence of any calcium flux across the plasma membrane, while there are other cell types for which calcium oscillations are entirely dependent on calcium entry. Some cell types can exhibit both kinds of behaviour, depending on the exact agonist used, and the concentration of the agonist. Thus, as a general rule, it is important to understand the behaviour of models both with and without calcium transport across the plasma membrane. However, 
	%\rev{In} cases where \rev{$J_{\text{pm}}(c)$} is very small relative to the other fluxes, the study of the closed-cell model \rev{obtained by setting $J_{\text{pm}}(c) \equiv 0$} 
	gives enormous insight into the behaviour of the full model, and thus a closed-cell analysis is an important first step for the analysis of almost any model of intracellular calcium. 
%\we{[comment: it's weird that we use a heuristic argument here, while later we go into an in-depth justification].}	
	\rev{The closed-cell reduction leads to} a planar system:
	\begin{align}
		\begin{split}
			\displaystyle\tau_h(c)\frac{\text{d}h}{\text{d}t}&=\displaystyle\hinf (c) - h,\\
			\displaystyle\frac{\text{d}c}{\text{d}t}&=\displaystyle\jipr \rev{(h,c)}-\jserca \rev{(c)} ,
			\label{eq:Full_model}
		\end{split}
	\end{align}
	\rev{where the total calcium $c_t$ is now a parameter. In this work, we focus on understanding the dynamics of the closed-cell model \eqref{eq:Full_model}.} % Note, that understanding the limiting dynamics in the closed-cell system with $J_{\text{pm}} = 0$ is also a necessary first step towards a complete perturbation analysis of the open-cell model \eqref{eq:Full_model_open_ct} in the (also biologically relevant) case $0 < |J_{pm}| \ll 1$.}
	
	\rev{We now specify the functional form of $\tau_h(c)$ and the flux terms. The functions $\tau_h(c)$ and $h_\infty(c)$ are given by Hill functions
	\begin{equation}\label{tau-h-orig}
		\tau_h (c) = \tau_{\rm max} \frac{K_\tau^4}{K_\tau^4+ c^4} , \qquad 
		\hinf(c) = \frac{K_h^4}{K_h^4+ c^4} ,
	\end{equation}
	which are monotonically decreasing switch-like functions; see Figure~\ref{fig:tau_h_fig}. Note that due to the constant factor $\tau_{\rm max}$, the value of $\tau_h(c)$ in particular varies between $0$ s and $1000$s depending on $c$. This induces a sharp, switch-like variation in the speed of evolution of $h$ as a function of calcium concentration $c$.
	}

	\begin{figure}
		\centering
		\includegraphics[scale=0.55]{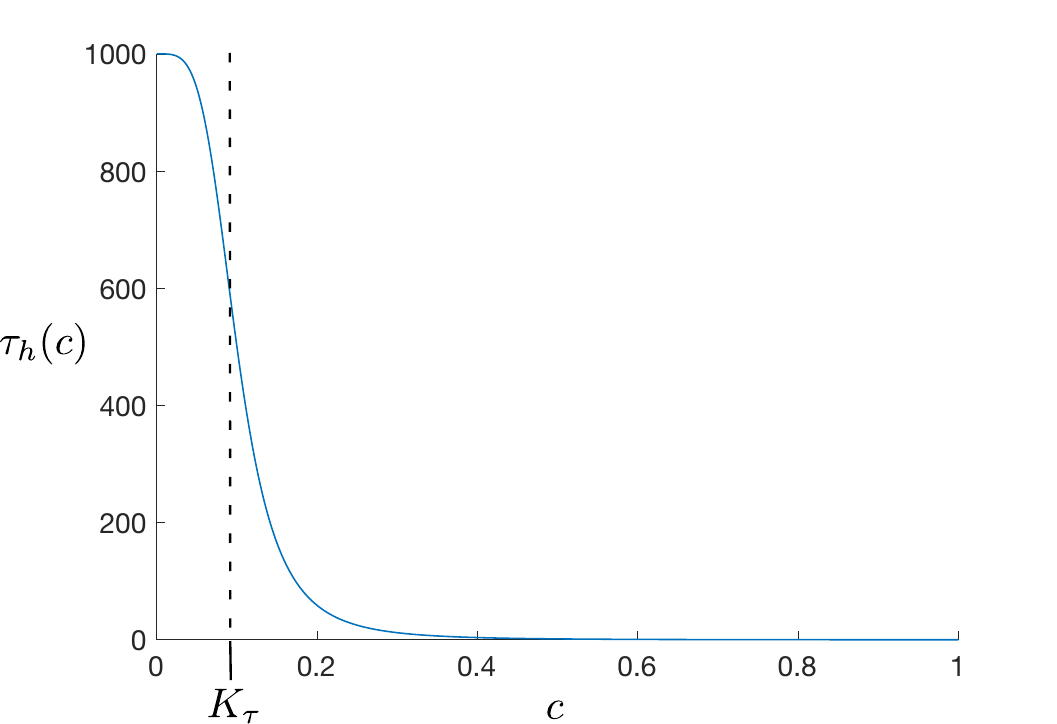}
		\caption{
			The function $\tau_h(c)$ defined in \eqref{tau-h-orig} varies over several orders of magnitude with $c$, which causes the speed of evolution of $h$ to vary significantly through state space.}
		\label{fig:tau_h_fig}
	\end{figure}

	$\jipr$ represents the flux of calcium from the ER to the cytosol through the IPR, and is modelled by 
	\begin{equation}\label{def:PO}
		\jipr\rev{(h,c)} = % \kipr P_O (c_e -c) = 
		\gamma \kipr P_O \rev{(h,c)} \left(c_t - (1 + 1/\gamma) c \right)\,,
	\end{equation}
	where $\kipr$ is the maximum flux through the IPR and $P_O$ is the open probability of the IPR. There are numerous models for $P_O$ in the literature \cite{dupont2016models,sneyd2005models,politi2006models,atri1993single}. \rev{Here we adopt the form proposed in \cite{sneyd2017dynamical}, which is}
	one of the simplest and most recent. \rev{We set}
	\begin{align}
		\label{eq:P_O}
		P_O\rev{(h,c)} = \frac{\beta \rev{(h,c)}}{\beta\rev{(h,c)} + k_\beta \left( \beta\rev{(h,c)} + \alpha\rev{(c)}\right)},
	\end{align}
	where
	\begin{align*}
		\alpha \rev{(c)} = \phi_{pdown}(1 - m_\alpha \rev{(c)} \hinf \rev{(c)}), \qquad
		\beta \rev{(h,c)}  = \phi_p m_\alpha \rev{(c)} h,
	\end{align*}
	with \rev{Hill-type functions}
	\begin{align*}
		%	\begin{split}
		m_\alpha \rev{(c)} = \frac{c^4}{K_c^4 + c^4},  \qquad
%		\hinf        = \frac{K_h^4}{K_h^4+ c^4}, \qquad
		%	\end{split}
		%	\begin{split}
		\phi_{pdown} = \frac{K_p^2}{K_p^2+ p^2}, \qquad
		\phi_p       = \frac{p^2}{K_p^2+p^2}.
		%	\end{split}
	\end{align*}

	\rev{Finally,} $\jserca$ is the flux from the cytosol to the ER via the ATPase pumps. \rev{The} SERCA pump model is bidirectional, and given by
	\rev{
	\begin{equation}
		\label{eq:Jserca}
		\jserca(c) = V_s \frac{c^2-\kbar \gamma^2 (c_t - c)^2}{K_s^2+c^2} = \jserca^+(c) - \jserca^-(c),
	\end{equation}
	where
	\[
		\jserca^+(c) = V_s \left( \frac{{c}^2}{{K}_s^2+c^2} \right),  \qquad
		\jserca^-(c) = K \gamma^2 {V}_s \left( \frac{({c}_t - {c})^2}{{K}_s^2+{c}^2} \right).
	\]
	The decomposition into two separate terms $\jserca^\pm$ reflects the fact that $\jserca$ models two separate fluxes: one positive, resulting from calcium ions being pumped from the cytosol to the ER, and one negative, corresponding to a leak of calcium ions from the ER into the cytosol.}
	%\we{[comment for James, I guess: why call it SERCA, when it is a leak; or is it a leaky pump!?]}
	%It will often be helpful to distinguish these two processes by writing
	%	\begin{equation}
	%		\jserca(c) = \jserca^+(c) - \jserca^-(c),
	%	\end{equation}
	%	where
	%	\begin{equation}
	%	
	%		\jserca^+(c) = V_s \left( \frac{{c}^2}{{K}_s^2+c^2} \right),  \qquad
	%		\jserca^-(c) = K \gamma^2 {V}_s \left( \frac{({c}_t - {c})^2}{{K}_s^2+{c}^2} \right).
	%\end{equation}}
	
	\
	
	\rev{The model \eqref{eq:Full_model} has parameters spanning multiple orders of magnitude (see Table \ref{tab:params}), as well as multiple Hill-type functions which induce a nonlinear switching behaviour as a function of calcium concentration $c$. Both features can potentially lead to a (possibly hidden) multiple time-scale structure. Our aim in the next section is to reveal any multiple time-scale structure by applying the procedure outlined in Section \ref{sec:the_procedure}.}

%	\revsj{Relaxation oscillations can be observed by numerically integrating the system \eqref{eq:Full_model} with the parameter values in Table \ref{tab:params}. The oscillations are represented in $(h,c)-$space in Figure \ref{fig:} \com{\textbf{TODO}}, and a corresponding time trace is shown in \ref{fig:} \com{\textbf{TODO}}. Notice that both time series contain components evolving on different time-scales, which is indicative of a non-standard time-scale separation. The multiple time-scale structure is also reflected in state space, insofar as distinct qualitative behaviours can be observed in three different regions: along the $h-$axis, along the $c-$axis, and away from both $h, c -$axes. %In the following Our aim in the following is to reveal the multiple time-scale structure in \eqref{eq:Full_model} by applying the procedure outlined in Section \ref{sec:the_procedure} to the model \eqref{eq:Full_model}.} \com{\textbf{Should we also include the heat maps here?}}

	\rev{
		\section{Deriving a singular perturbation problem}
		\label{sec:steps_I_IV}

		In this section, we apply steps \we{(I)-(III)} to the calcium model \eqref{eq:Full_model} in order to derive a system amenable for an analysis using perturbation methods.

		%\subsection*{Step (I): Non-dimensionalise and simulate}
		\subsection*{Step (I): \we{Non-dimensionalise}}
		\label{sub:non_dimensionalisation}
		
		We begin by non-dimensionalising} the model \eqref{eq:Full_model}. {\color{black} As well as making the variables unitless, this process also rescales the variables to ensure they are of order one for the regions of state and parameter space of interest.}
	%To enable a reliable identification of the different time-scales in equations \eqref{eq:Full_model}, we non-dimensionalise and rescale the model. 
	Since $h\in [0,1]$ is already dimensionless,
	%so it does not need to be non-dimensionalised and its value is bounded by 0 and 1, so it does not need to be scaled. 
	we only need to non-dimensionalise the variables $c$ and $t$. We define
	%follow a similar procedure to that in \cite{harvey2011multiple} and introduce the dimensionless variable $\tilde{c}$ and dimensionless time-scale $\tilde{t}$,
	\begin{equation}
		c = Q_c \tilde{c}, \qquad
		t = T \tilde{t} ,
	\end{equation}
	where $Q_c$ and $T$ denote a reference concentration and time-scale respectively, to be specified below. 
		\begin{table}[t!]
		\center
		\begin{tabular}{|c|l||c|l|}
			\hline
			Parameter         & Value & Parameter            & Value      \\
			\hline
			$\tilde{k}_\beta$ & 0.4   & $\tilde{K}_\tau$     & 0.05       \\
			$\tilde{K}_c$     & 0.1 & $\tilde{V}_s$        & 0.0081     \\
			$\tilde{K}_s$     & 0.1 & $\tkbar$             & 0.000019 \\
			$\tilde{K}_p$     & 0.1   & $\tilde{K}_h$        & 0.04      \\
			$\tkipr$          & 0.18     & $\gamma$             & 5.5        \\
			$\tilde{p}$       & 0.025  & $\tilde{\tau}_{\rm max}$ & 55000       \\
			$\tilde{c}_t$     & 1  & -                    & -          \\
			\hline
		\end{tabular}
		\caption{Dimensionless parameters in system \eqref{eq:Non_dim}. The tilde notation is dropped in \rev{step (IIa)}.}
		\label{tab:params_non_dim}
	\end{table}
	%
	%so that $\tilde{c}$ and $h$ are of similar magnitude to each other. Similarly, $T$ will be chosen below to be a typical time-scale for the dynamics. 

	We obtain the following dimensionless model:
	\begin{equation}
		\label{eq:Non_dim}
		\begin{split}
			\tilde{\tau}_h(\tilde{c}) \displaystyle\frac{dh}{d\tilde{t}} &= \displaystyle \thinf(\tilde c) - h ,\\
			\displaystyle \frac{d\tilde{c}}{d\tilde{t}} &= \displaystyle \tjipr(h, \tilde c) - \tjserca(\tilde c),\\
		\end{split}
	\end{equation}
	where the new (dimensionless) expressions $\tilde \tau_h(\tilde c), \tilde h_\infty(\tilde c), \tjipr(h,\tilde c)$ and $\tjserca(\tilde c)$ are defined analogously to their dimensional counterparts, except in terms of the new (dimensionless) parameters
	\begin{equation}
		\label{eq:nondim_pars}
		\tilde{K}_i = \frac{K_i}{Q_c}, \ \ \tilde \tau_{\rm max} = \frac{\tau_{\rm max}}{T} , \ \
		\tilde{k}_\text{\text{IPR}} = T \kipr, \ \ \tilde{V}_s = \frac{T}{Q_c}V_s, \ \
		\tilde{c}_t = \frac{c_t}{Q_c}, \ \ \tilde p = \frac{p}{Q_c} ,
	\end{equation}
	where $i = c, h, \tau, s$ or $p$.
	We choose the following reference scales for our model:
	\begin{equation}
		Q_c = c_t = 2\,, 
		\qquad T = \frac{1}{\gamma \kipr} = \frac{1}{55}\,.
	\end{equation}
	Thus, \eqref{eq:Non_dim} shows the dynamics of $c$ relative to the total concentration $c_t$, on the IPR time-scale set by the uniform scaling $\gamma \kipr$; this is the fastest time-scale in the system. Numerical values for the dimensionless model parameters consistent with Table \ref{tab:params} are shown in Table \ref{tab:params_non_dim}. \rev{The relaxation oscillations in Figure \ref{fig:non_dim} were computed using system \eqref{eq:Non_dim} with the (dimensionless) parameters in Table \ref{tab:params_non_dim}. The multiple time-scale structure is also reflected in state space, see Figure \ref{fig:phase_space}. \revsj{The presence of curvature in the `fast part' of the limit cycle bounded away from the nullclines reflects the fact that both $h$ and $\tilde c$ are `fast' in this regime, indicating a non-standard time-scale separation.}} %insofar as distinct qualitative behaviours can be observed in three different regions: along the $h-$axis, along the $c-$axis, and away from both $h, c -$axes.}
	
	%, and Figure \ref{fig:non_dim} shows relaxation oscillations in system \eqref{eq:Non_dim} with the new parameter values in Table \ref{tab:params_non_dim}. Both time series contain components evolving on different time-scales, which is indicative of a \wm{non-standard} time-scale separation. 
%	\com{\textbf{Suggestion: perhaps we should also show the limit cycle here, with nullclines. This should address the reviewers point and illustrate the important regions in phase space.}}  \rev{This can also be observed in state space, where we observe distinct qualitative behaviour in three different regions: along the $h-$axis, along the $c-$axis, and away from both $h, c -$axes.}
	%\[
	%(h,c) \approx (1,1), \qquad (h,c) \approx (1,0), \qquad (h,c) \approx (0,1).
	%\]
	
	%\begin{remark}
	%	For analytical puposes, we recommend that the system be non-dimensionalised such that the leading order contribution on the fastest time-scale is of numerical order $1$.
	%\end{remark}

	\begin{figure}[t!]
		\centering
		\includegraphics[scale=0.5]{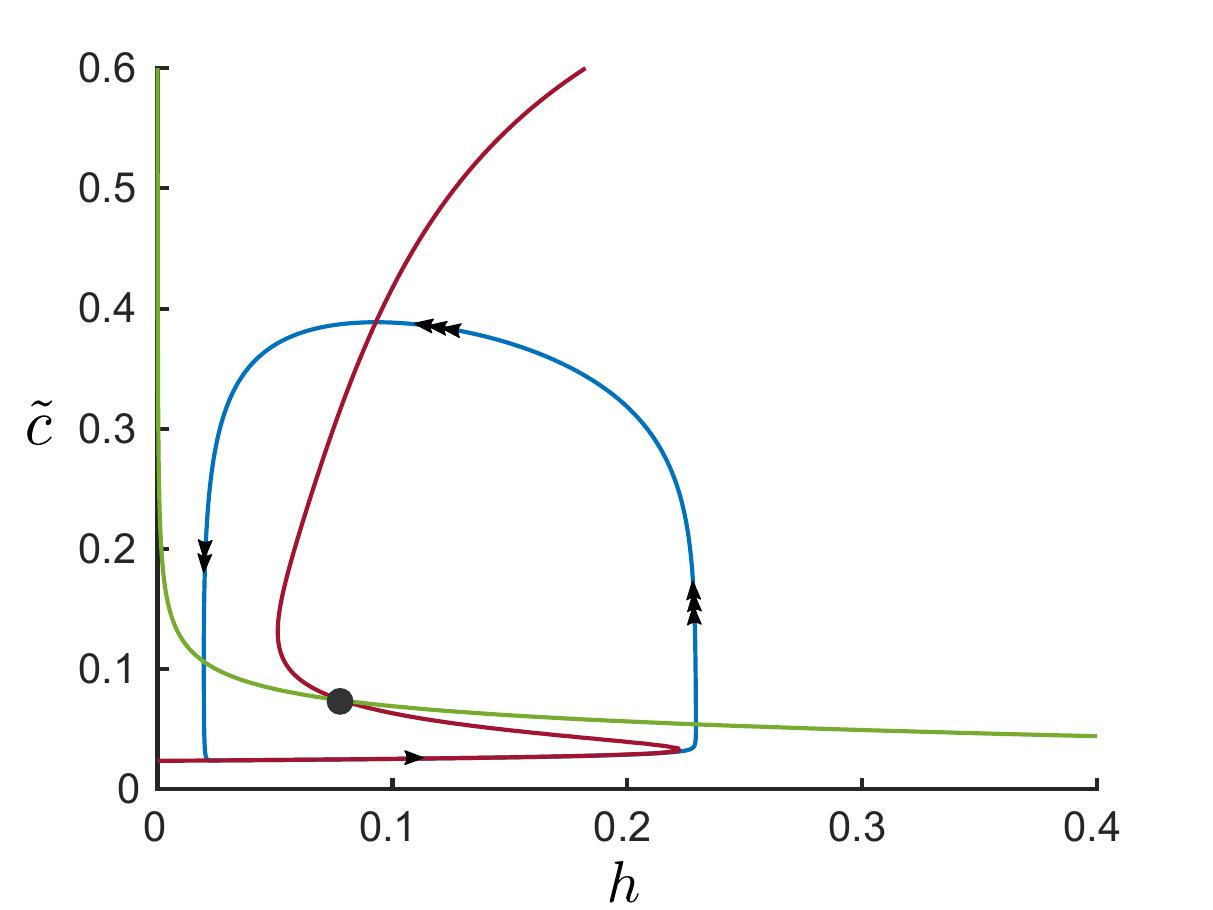}
		\caption{\rev{State space representation of the relaxation oscillations in system \eqref{eq:Non_dim} with parameter values from Table \ref{tab:params_non_dim}, c.f.~the corresponding time series in Figure \ref{fig:non_dim}. The $\tilde c' = 0$ and $h' = 0$ nullclines are shown in red and green respectively, and the unique equilibrium corresponding to their intersection is indicated by a black disk. Notice that the limit cycle closely follows the \revsj{$\tilde c' = 0$} nullcline near the \revsj{$h$ axis, % and to a lesser extent the $h$ axis} $h$ and $\tilde c$ axes, 
		indicative of multiple time-scale dynamics. Curvature in the phase of the limit cycle bounded away from the nullclines is indicative of a non-standard time-scale separation, since both $h$ and $\tilde c$ are fast in this regime. The relative speed of the oscillation about the limit cycle is indicated by single, double and triple arrows. Note that the downward motion (approximately parallel to the $\tilde c-$axis) is faster than the motion along the $\tilde c = 0$ nulllcline, but slower than the `upward' and `leftward' components of the oscillation; see again Figure \ref{fig:non_dim}b.}}}
	%	\we{[comment: a reader might ask why the trajectory doesn't seem to follow as closely the $c$-axis; we need to provide some additional statement here about multiple time-scales (indicate speed through arrows or heat map or refer to Fig 1b....)]}} %{\color{red}VK: The limit cycle does not follow the nullcline closely near the $\tilde c$ axis. There is a bit of a disparity between the size of the fonts for axis labels and tick-mark labels.}}
		\label{fig:phase_space}
	\end{figure}

	% subsection non_dimensionalisation (end)
	
%	\begin{remark}
%		\revsj{The relaxation oscillations in Figure \ref{fig:non_dim} were computed using system \eqref{eq:Non_dim} with the (dimensionless) parameters in Table \ref{tab:params_non_dim}.}
%	\end{remark}
	
	\
	
	In what follows we refer only to the dimensionless model \eqref{eq:Non_dim} and parameters \eqref{eq:nondim_pars}, \wm{dropping the tildes for notational simplicity}.
	
	%\WM{[NZ] I would prefer keeping the tildes in section 4 and only dropping them in section 5. It's otherwise confusing to me, i.e. does it refer to the original (Table 1) or non-dimensional (Table 2). I know you advocated to remove it.} {\color{red}{VK: Having tildes and hats at the same time creates a huge mess, and I think is hard to follow, hence my concern with retaining tildes in Section 4. Is there some way to satisfactorily make it clearer for you (e.g., by using more text) whether we refer to Table 1 or Table 2, and hence still drop tildes in Section 4?}}

	\rev{
		\subsection*{Step \we{(IIa): Associate small parameters to maximal process/flux rates}}
		
		Consider the dimensionless system \eqref{eq:Non_dim}. We need to compare relative magnitudes of the constant scaling factors associated to the (dimensionless) flux terms. These are given by
		\begin{equation}
			\label{eq:eps_numerical}
		\frac{1}{\tau_{\rm max}} = 1.8 \times 10^{-5} , \qquad \gamma k_{\rm IPR} = 1.0 , \qquad 
		V_s = 8.1 \times 10^{-3} , \qquad \frac{V_s K \gamma^2 K_s^2}{c_t^2} = 4.7 \times 10^{-8} ,
		\end{equation}
		using the parameter values in Table \ref{tab:params_non_dim}. Note that the largest scaling factor is $\gamma k_{\rm IPR} = 1.0$ due to the choice of non-dimensionalisation in step (I). Based on these order of magnitude comparisons, we introduce three candidate small parameters:
		\begin{equation}
			\label{eq:eps_123}
			\epsilon_1 = \frac{1}{\tau_{\rm max}} , \qquad \epsilon_2 = V_s, \qquad 
			\epsilon_3 = \frac{V_s K \gamma^2 K_s^2}{c_t^2} .
		\end{equation}
		Note that at this point, the $\epsilon_i$ are considered to be \we{independent}. Their magnitudes relative to each other will be considered in step (III). 
		
		Introducing the small parameters $\epsilon_i$, we obtain the system
		\begin{equation}
			\label{eq:step_II_system}
			\begin{split}
				\left(\frac{\tau_h(c)}{n_h} \right) h' &= \epsilon_1 \left( h_\infty(c) - h \right) , \\
				c' &= J_{\rm IPR}(h,c) - \epsilon_2 \left( \frac{\jserca^+(c)}{n_{\rm SERCA}^+} \right) + \epsilon_3 \left( \frac{\jserca^-(c)}{n_{\rm SERCA}^-} \right) ,
			\end{split}
		\end{equation}
		where $n_h = \tau_{\rm max}$, $n_{\rm SERCA}^+ = V_s$ and $n_{\rm SERCA}^- = V_s K \gamma^2 K_s^2 / c_t^2$ are the normalisation constants associated with $\tau_h(c)$, $J_{\rm SERCA}^+(c)$ and $J_{\rm SERCA}^-(c)$ respectively. \revsj{Notice that by \eqref{eq:eps_123}, we have
			\[
			\epsilon_1 n_h = \frac{\epsilon_2}{n^+_{SERCA}} = \frac{\epsilon_3}{n_{SERCA}^-} = 1 .
			\]
		%In what follows our analysis will depend upon the basic modelling assumption that a 
		In the perturbation analysis which follows, we consider the limiting dynamics as $\epsilon_1, \epsilon_2, \epsilon_3 \to 0$ while keeping $n_h, n^+_{SERCA}, n^-_{SERCA}$ fixed. The basic modelling assumption is that the perturbation analysis for $\epsilon_1, \epsilon_2, \epsilon_3 \ll 1$ is valid and informative for values of $\epsilon_1, \epsilon_2 , \epsilon_3$ up to the numerical values corresponding to equations \eqref{eq:eps_numerical} and \eqref{eq:eps_123}.}

		\subsection*{Step \we{(IIb): Associate small parameters to steep switches in process/flux terms}}
		
%		In step (II) we considered the relative sizes associated to particular flux terms in terms of their constant scaling factors. While such considerations are important and necessary for the identification of multiple scale structure, they do not take into account the fact that the relative size of flux contributions also depends on the region in state space. Such considerations are particularly relevant if the model features switching behaviour. 
		% in which the magnitude (or more importantly, the relative size) of process/flux terms vary significantly as a function of state variables.
		%As discussed already in Section \ref{sec:motivations_for_the_closed_cell_model}, 
		
		As observed in Section \ref{sec:motivations_for_the_closed_cell_model}, the presence of the Hill function $\tau_h(c)$ in the left-hand-side of system \eqref{eq:Full_model} is expected to lead to significant time-scale variation in the $h$ dynamics as a function of the calcium concentration $c$. However, this has not been accounted for in the introduction of small parameters $\epsilon_i$ in step \we{(IIa)}. This can be seen by considering $\epsilon_i \to 0$ for each $i = 1,2,3$ in system \eqref{eq:step_II_system}. The system obtained in this (singular) limit fails to capture the structure of the oscillations identified in Figures \ref{fig:non_dim} and \ref{fig:phase_space}. In particular, the resulting limiting problem has $h'=0$, % cannot describe the phase of the oscillations characterised by a downward motion along the $c-$axis, and the trivial fast dynamics 
		so its fast dynamics (vertical in $c$) cannot accurately represent the observed (fast, nonlinear) solution segment that occurs away from the $c-$ and $h-$axes.}
	%connection between the phases of the oscillation that pass close to the $h-$axis and the $c-$axis. %Such considerations motivate a closer look into switching mechanisms in the model.
	 
	%What is perhaps less clear at this stage, is the effect of the other switch-like terms (Hill functions) in the model on the relative magnitudes and hence multiple-scales structure of the dynamics.

	\begin{figure}
		\centering
		\includegraphics[scale=0.1]{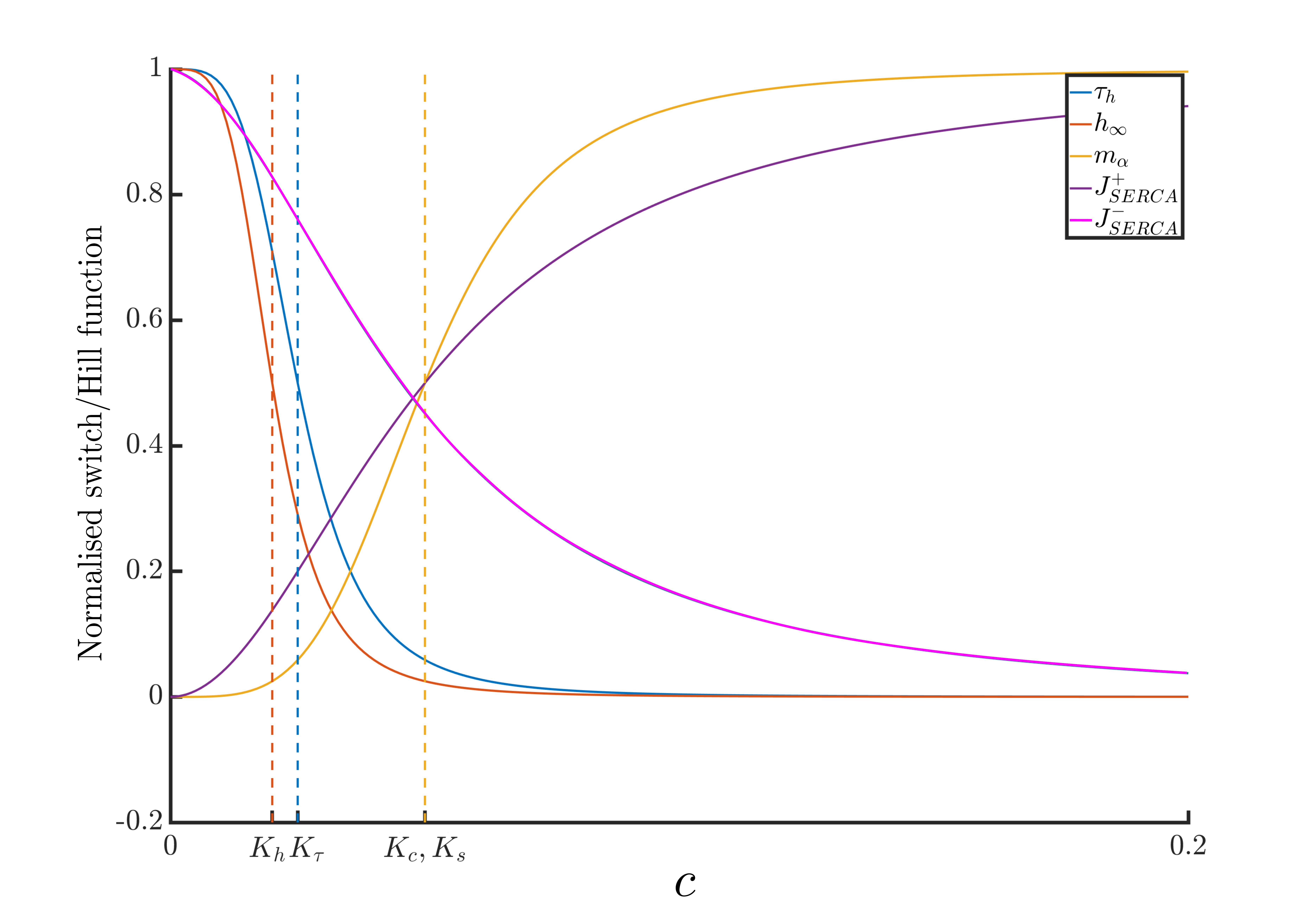}
		\caption{Normalised switch-like functions in system \eqref{eq:Non_dim}, which play a role in determining the dominant processes in \eqref{eq:Non_dim} for a given value of $c$. Corresponding half-values $K_i$ for $i=h,\tau,c,s$ are also shown. The $J^-_{\rm SERCA}$ profile is less than $0.5$ at the half-value $K_s$ due to extra $c_t-$dependence in the numerator in \eqref{eq:Jserca}. \rev{In order to derive a suitable perturbation problem from \eqref{eq:Full_model} we approximate the steepest of these, i.e.~$\tau_h(c)$ and $h_\infty(c)$, by one-sided non-smooth switches via the limits in \eqref{eq:switch_approx}.}}
		\label{fig:processes_fig}
	\end{figure}
	
	\
	
	\rev{In order to understand the role of switching in the model \eqref{eq:Full_model}, we consider the effect of the five Hill-type functions
		\begin{equation}
			\label{eq:Hill_functions}
			\tau_h(c), \ h_\infty(c), \ m_\alpha(c), \ J_{\rm SERCA}^\pm(c) .
		\end{equation}
		%in the model \eqref{eq:Full_model}. 
		Each (normalised) Hill function is plotted in Figure \ref{fig:processes_fig}, and each one is a candidate for approximation by a non-smooth switch. There are many ways to make such approximations, and the `best' choice may depend heavily on the problem at hand. For our particular model \eqref{eq:Full_model}, we shall appeal to the fact that general Hill-type functions
		\begin{equation}
			\label{eq:Hill}
			{\rm Hill}_{K_d}^+(c) = \frac{c^m}{K_d^m + c^m} \qquad \text{or} \qquad
			{\rm Hill}_{K_d}^-(c) = \frac{K_d^m}{K_d^m + c^m} ,
		\end{equation}
		with half-value $K_d$ and exponent $m \in \mathbb N_+$ can be well approximated by a suitable one-sided switch, given a sufficiently small half-value $K_d \ll 1$. This is achieved directly via the limit $K_d \to 0$ since
		\begin{equation}
			\label{eq:switch_approx}
			\lim_{K_d \to 0} {\rm Hill}_{K_d}^+(c) =
			\begin{cases}
				1 & \text{if } c > 0 , \\
				0 & \text{if } c = 0 , 
			\end{cases}
			\qquad \text{and} \qquad 
			\lim_{K_d \to 0} {\rm Hill}_{K_d}^-(c) =
			\begin{cases}
				0 & \text{if } c > 0 , \\
				1 & \text{if } c = 0 . 
			\end{cases}
		\end{equation}
		Such an approach has been used in \cite{kosiuk2016geometric} in the case of Michaelis-Menten terms (Hill functions \eqref{eq:Hill} with $m=1$), allowing for the detailed analysis of a non-standard relaxation oscillation in a minimal model of the embryonic cell cycle.} 
	\begin{comment}
	\
	
	\
	
	More precisely, we are interested in the approximation of general Hill functions
	\[
	{\rm Hill}^+(c) = \frac{c^m}{K_d^m + c^m} \qquad \text{or} \qquad
	{\rm Hill}^-(c) = \frac{K_d^m}{K_d^m + c^m} ,
	\]
	by parameter dependent functions $\Phi^{\pm}(c,\epsilon)$ such that
	\[
	\Phi^{\pm}(c,\delta) = {\rm Hill}^{\pm}(c) 
	\]
	for some numerically small $\delta > 0$, and
	\[
	\lim_{\epsilon \to 0} \Phi^+(c,\epsilon) = 
	\begin{cases}
	1  & \text{if } c > K_d , \\
	0  & \text{if } c < K_d ,
	\end{cases}
	\]
	or
	\[
	\lim_{\epsilon \to 0} \Phi^+(c,\epsilon) = 
	\begin{cases}
	1  & \text{if } c < K_d , \\
	0  & \text{if } c > K_d .
	\end{cases}
	\]
	\end{comment}
	\rev{In the following we consider which of the functions \eqref{eq:Hill_functions}, if any, it is reasonable to approximate in this way. 
	\we{
	\begin{remark}
	Alternatively, one could obtain another non-smooth switch out of Hill-type functions \eqref{eq:Hill} by fixing $K_d \in (0,1)$ and taking the limit $m\to \infty$ where the discontinuity is located away from the boundary at $c=K_d$. \rev{See \cite{Plahte2005,Ironi2011,Machina2013,Glass2018} for examples of such an approach in the context of gene regulatory networks.}
	\end{remark}
	}
		
		Making a limiting approximation of a smooth switch by a step function as in \eqref{eq:switch_approx} is valid only for `sufficiently \we{steep}' functions, i.e.,~the switch must exhibit a significant (order of magnitude or greater) variation over a `sufficiently narrow' region in state space. As with the question ``how small is small enough?", the question ``how \we{steep} is \we{steep} enough" is left to the discretion of the modeller, but can be informed by numerical and/or analytical considerations. %{\color{red} VK: I found the following statements (down to `Based on this principle') confusing because I read it as saying that a `sufficiently sharp' switch is one with slope greater than one. Thus it read like a proposal of a ``reasonable" approach  immediately followed by a choice to do something different. A clearer explanation might be to say we should evaluate the slopes at the midpoints and identify the steepest switches of those with slope greater than about 1 (since anything with slope less than one in non-dimensionalised coordinates could not be considered steep). Then try approximating just the steepest of those with non-smooth limiting switches. If that doesn't work, try adding the next steepest etc. Some of this could be covered off just be referring to the remark.} 
		A reasonable analytic approach is to adopt the notion that a switch may be considered \we{(sufficiently) `steep'} if the magnitude of the derivative (i.e.,~the slope) evaluated at its midpoint is \revsj{greater than $1$ by a numerical order of magnitude, i.e.,~by at least a factor of $10$.} %numerical order of magnitude greater than $1$. 
		Calculating slopes for the switches in Figure \ref{fig:processes_fig}, we find
		\begin{equation}
			\begin{gathered}
				\text{slope} (\tau_h) |_{K_\tau} = %1 / K_{\tau} 
				20 , \qquad
				\text{slope} (m_\alpha) |_{K_c} = %1 / K_c 
				10 , \qquad
				\text{slope} (h_\infty) |_{K_h} = %1 / K_h 
				25 , \\
				\text{slope} (J_{\rm SERCA}^+) |_{K_s} = %1 / 2 K_s 
				2.5 , \qquad
				\text{slope} (J_{\rm SERCA}^-) |_{K_s \approx K_{mid}} = 5.0  . 
			\end{gathered}
		\end{equation}
		The large slopes associated to $\tau_h(c)$ and $h_\infty(c)$ suggest that they are suitable candidates for approximation by a non-smooth switch. Conversely, the slopes associated with the $J^\pm_{\rm SERCA}(c)$ terms are still of numerical order $1$, which suggests that these terms should not be analysed via non-smooth approximations. The question remains as to whether or not the slope associated with $m_\alpha(c)$ %$\revsj{, which lies precisely `between' numerical orders of $1$ and $10$,} 
		should be considered as sharp enough for the validity of non-smooth approximations. In the following we \revsj{shall} adopt the principle that for borderline cases such as these additional approximations should be avoided\revsj{; see however Remark \ref{rem:sufficiently_sharp2} below}. \revsj{In accordance with this principle, we} %Based on this principle, we shall 
		restrict ourselves to the non-smooth approximation of only the steepest Hill functions $\tau_h(c)$ and $h_\infty(c)$. We therefore introduce two more small parameters
		\begin{equation}
			\label{eq:eps_45}
			\epsilon_4 = K_\tau , \qquad \epsilon_5 = K_h . 
		\end{equation}
		The resulting perturbation problem will be given in step \we{(III)}, once all five small parameters $\epsilon_i$ have been related to a single small parameter by a common scaling.

		\begin{remark}
		\label{rem:sufficiently_sharp2}
			In general, additional (non-smooth) approximations can be expected to lead to simpler limiting systems, but the validity of the perturbation problem derived via these approximations as a means for studying the dynamics of the model \eqref{eq:Full_model} becomes harder to justify with each approximation. The aim is to balance tractability and accuracy, keeping in mind that one often comes at the expense of the other. In practice, a posteriori arguments will often also play an important role in striking such a balance. We have found for the model \eqref{eq:Full_model}, that a non-smooth approximation of $\tau_h(c)$ and $h_\infty(c)$ alone is sufficient to capture the relevant dynamics in an analytically tractable way. However, one can continue to push for `more tractable' perturbation problems describing the model \eqref{eq:Full_model} by introducing additional approximations. The next step, is to approximate $m_\alpha(c)$ by a non-smooth switch, followed by $J^-_{\rm SERCA}(c)$, and so on. In general, this approximation procedure can be applied sequentially to each switch, starting with the steepest, and terminating at a point deemed by the modeller to strike the desired balance between accuracy and tractability. % is deemed in decreasing steepness for each identified switch,   % and terminated at the desired point.
		\end{remark}
	}

	\rev{
		\subsection*{Step \we{(III)}: Relate small parameters}
		
		Steps \we{(I)-(IIb)} above yielded no fewer than five independent small parameters $\epsilon_i, i = 1,\ldots,5$. In order to obtain a tractable perturbation problem, we introduce a common scaling for all five $\epsilon_i$ in terms of a single small parameter $\epsilon$. This can be achieved by defining $\epsilon$ in terms of one of the five $\epsilon_i$ and then determining sensible choices for the $a_i$ and $b_i$ in \eqref{eq:scaling_polynomial} based on order of magnitude comparisons for a fixed value of $\epsilon$ which returns the original parameter values in Table \ref{tab:params_non_dim}. We define a common small parameter $\epsilon$ via the choice
		\[
		\sqrt{\epsilon} := \epsilon_4 ,
		\]
		%which has numerical value $\sqrt{\epsilon} := K_{\tau}=0.05$, VK: I don't think it makes sense to include this statement, since we are thinking of eps varying. It is covered by the statements below in any case.
		and the polynomial scaling}
	\begin{equation}
		\label{eq:common_scaling}
		\epsilon := \epsilon_{1}^{1/2}{ \hat \tau_{\rm max}}^{1/2} = \epsilon_2 \hat V_s^{-1} = \epsilon_3 \hat K^{-1} \gamma^{-2} = \epsilon_4^2 = \epsilon_5^{2} \hat K_h^{-2} .
		%	\epsilon := \epsilon_\tau^2 = \epsilon_{\tau_{\rm max}}^{1/2}{ \hat \tau_{\rm max}}^{1/2} = \epsilon_h^{2} \hat K_h^{-2} = \epsilon_s \hat V_s^{-1} = \epsilon_K \hat K^{-1} \gamma^{-2},
	\end{equation}
	\rev{Equation \eqref{eq:common_scaling} amounts to a choice of $a_i$ and $b_i$ in \eqref{eq:scaling_polynomial}. Specifically, we have chosen exponents
		\[
		b_1 = 1/2, \quad 
		b_2 = 1, \quad
		b_3 = 1, \quad 
		b_4 = 2, \quad
		b_5 = 2 ,
		\]
		and coefficients $a_i$ (which are defined in terms of the hatted parameters in \eqref{eq:common_scaling}) which can be expressed in terms of the original (dimensionless) system parameters as follows:
	\begin{equation}
		a_1 = {\hat \tau_{\rm max}}=K_\tau^4 \tau_{\rm max}, 
		 \ \ a_2 = {\hat V_s}=V_s K_{\tau}^{-2}, 
		  \ \ a_3 = {\hat K}=K K_{\tau}^{-2} ,
		 \ \ a_4 = 1,
		 \ \ a_5 = {\hat K_h}=K_h K_\tau^{-1} .
		\label{eq:hatpars}
	\end{equation}
	Notice that the $a_i$ parameters $\hat \tau_{\rm max}, \hat K_h, \hat V_s$ and $\hat K \gamma^2$ are of numerical order $1$ as required; see Table~\ref{tab:params_post_scaling}. Moreover, by setting $\epsilon = K_\tau^2 = 2.5 \times 10^{-3}$ and \we{substituting} the new rescaled parameter values from Table~\ref{tab:params_post_scaling}, the values coincide with the original parameter values in Table~\ref{tab:params_non_dim}. Hence, for our particular problem, the assumption that the $\epsilon_i$ scale with a single $\epsilon$ according to \eqref{eq:common_scaling} amounts to the assumption that $\epsilon = 2.5 \times 10^{-3}$ is `sufficiently small' for the validity of perturbation arguments.}
	
	\begin{table}[t!]
		\center
		\begin{tabular}{|c|l||c|l|}
			\hline
			Parameter         & Value & Parameter            & Value      \\
			\hline
			$k_\beta$ & 0.4   & $\epsilon$     & 0.0025      \\
			$K_c$     & 0.1 & $\hat{V}_s$        & 3.24    \\
			$K_s$     & 0.1 & $\hat K$             & 0.0076 \\
			$K_p$     & 0.1   & $\hat K_h$        & 0.8      \\
			$\kipr$          & 0.18     & $\gamma$             & 5.5        \\
			$p$       & 0.025  & $\hat{\tau}_{\rm max}$ & 0.34       \\
			$c_t$     & 1  &  -    &   -    \\
			\hline
		\end{tabular}
		\caption{Numerical parameter values for  \eqref{eq:main2} and \eqref{eq:main1} obtained by setting $\epsilon=K_\tau^2=0.0025$ and applying \eqref{eq:hatpars}  to the values of $\tau_{\rm max}, V_s, K$ in Table \ref{tab:params_non_dim}. All other parameter values are unchanged. The hat notation will be dropped in Section \ref{sec:slow-fast_analysis_and_statement_of_the_main_result}.}
		\label{tab:params_post_scaling}
	\end{table}
	
	\rev{
		\begin{remark} %{\color{red}{VK: I found this remark (after the first sentence) confusing. I suggest either omitting all except the first sentence or explaining more clearly what ``remedied by redefining $\epsilon$'' means.}}
			The choice to fix $\sqrt{\epsilon} = \epsilon_4$ was made a posteriori so that the leading order perturbation in the perturbation problem obtained is of order $\epsilon$. %In practice, a good approach is to choose $$\epsilon := \max\{\epsilon_i : i = 1, \ldots , m \}$$ so that the appropriate exponents $a_i$ satisfy $a_i \geq 1$. If the leading order perturbation in the resulting system is of order $\epsilon^d$ for some $d \neq 1$, \revsj{one can always rewrite the system as a perturbation problem with respect to $\varepsilon := \epsilon^d$.}
		\end{remark}
	}

	After applying the common scaling \eqref{eq:common_scaling}, we finally arrive at a version of the model in a form that is amenable to \rev{perturbation} analysis, i.e., \eqref{eq:main2} defined in the following \rev{proposition}.
	
	\begin{proposition}
		\label{prop:sing_pert_form}
		System \eqref{eq:Non_dim} can be written as
		\begin{equation}\label{eq:main2}
			\begin{split}
				h' & = \hat \tau_{\rm max}^{-1} \left(\mathfrak h_{\infty}({c},\epsilon) - h\right) \left({c}^4 + \epsilon^2 \right), \\
				{c}' &= \mathfrak J_{\text{IPR}}(h,{c},\epsilon) - \epsilon \mathfrak{J}^+_{\text{\rm SERCA}}({c},\epsilon) + \epsilon^2 \mathfrak{J}^-_{\text{\rm SERCA}}({c},\epsilon),
			\end{split}
		\end{equation}
		where
		\begin{equation}
			\label{eq:mathfrak_h}
			\mathfrak h_\infty(c,\epsilon) = \frac{\epsilon^2 \hat K_h^4}{\epsilon^2 \hat K_h^4 + c^4} ,
		\end{equation}
		$\mathfrak J_{\text{IPR}}(h,c,\epsilon)$ is given by \eqref{def:PO} with explicit $\epsilon-$dependence due to $K_\tau = \sqrt{\epsilon}$, 
		and
		\begin{align}\label{eq:serca_terms}
			\mathfrak{J}_{\text{\rm SERCA}}^+(c) = \hat V_s \left( \frac{c^2}{K_s^2+c^2} \right), \quad 
			\mathfrak{J}_{\text{\rm SERCA}}^-(c) = \hat K \gamma^2 \hat V_s\left( \frac{(c_t - c)^2}{ K_s^2+c^2} \right) .
		\end{align}

		By considering $\epsilon \ll 1$, system \eqref{eq:main2} can be written as the series expansion
		\begin{align}
			\label{eq:main1}
			\begin{split}
				\begin{pmatrix}
					h' \\
					{c}'
				\end{pmatrix}
				=&
				\begin{pmatrix}
					- \hat \tau_{\rm max}^{-1} \\
					\mathfrak{J}^{(0)}_{\text{\text{IPR}}}(h,{c})
				\end{pmatrix}
				{c}^4h + \epsilon
				\begin{pmatrix}
					0 \\
					- \mathfrak J^+_{\text{\rm SERCA}}({c})
				\end{pmatrix}
				\\
				&+ \epsilon^2
				\begin{pmatrix}
					\hat \tau_{\rm max}^{-1}(\hat K_h^4 - h) \\
					\mathfrak{J}^{(1)}_{\text{\text{IPR}}}(h,{c}){c}^8h + \mathfrak J^-_{\text{\rm SERCA}}({c})
				\end{pmatrix}
				+ \epsilon^4
				\begin{pmatrix}
					R_{h_\infty}({c},\epsilon) \\
					R_{\text{\text{IPR}}}(h,{c},\epsilon)
				\end{pmatrix} ,
			\end{split}
		\end{align}
		with IPR terms
		\[
		\mathfrak{J}^{(k)}_{\textrm{\text{IPR}}}(h,c) := \left( \frac{1}{c^{4(k+1)} h}\right) k_{\text{IPR}} P_O^{(k)}(h,c,0) ( \gamma c_t - (1+ \gamma) c), \qquad k = 0, 1,
		\]
		where
		\begin{equation}
			\begin{split}
				P_O^{(0)}(h,c) &= \frac{p^2 c^4 h}{p^2 (1 + k_\beta) c^4 h + k_\beta K_p^2(K_c^4 + c^4)}, \\
				P_O^{(1)}(h,c) &= \frac{k_\beta K_p^2 p^2 c^4 h}{\left(c^4 h(k_\beta+1) p^2+k_\beta K_p^2 \left(c^4+K_c^4\right)\right)^2} .
			\end{split}
		\end{equation}
		The remainder terms $R_{h_\infty}({c},\epsilon)$ and $R_{\text{\text{IPR}}}(h,{c},\epsilon)$ are $O(1)$ as $\epsilon \to 0$.
	\end{proposition}
	\begin{proof}
		System \eqref{eq:main2} is obtained from system \eqref{eq:Non_dim} after making the substitutions in \rev{\eqref{eq:eps_123} and \eqref{eq:eps_45}}, 
		%\eqref{eq:h_scale_tau}, \eqref{eq:Kh_scaling}, \eqref{eq:h_scale_tau_\rm max} and \eqref{eq:c_scale}
		applying the common scaling \eqref{eq:common_scaling}, and dividing both sides of the equation for $h'$ by $\tau_h(c)$.
		
		The expansion in \eqref{eq:main1} follows after significant but standard algebraic manipulations and Taylor expansion in $\epsilon$.
	\end{proof}
	
	%VK: I did not check in detail the expansion for J_IPR
	
	Systems \eqref{eq:main2} and \eqref{eq:main1} will be considered in detail in \rev{the following section}, % \ref{sec:slow-fast_analysis_and_statement_of_the_main_result}, 
	where we present a GSPT analysis of the dynamics for $0 < \epsilon \ll 1$. Since only these systems will be considered in the remainder of the manuscript, we will drop the hat notation on rescaled parameters, for the sake of readability. Final system parameters are given in Table \ref{tab:params_post_scaling}.

	\section{Multiple time-scale analysis} %and statement of the main result} VK: We've shifted the emphasis a bit to the algorithm, so I am not sure it makes sense to call this "the main result"
	\label{sec:slow-fast_analysis_and_statement_of_the_main_result}

	\rev{Having derived a singular perturbation problem via \we{`pre-processing'} steps \we{(I)-(III)} in Section \ref{sec:steps_I_IV}, it remains to carry out step \we{(IV)}, i.e.,~a} singular perturbation analysis of system% \eqref{eq:main2} resp.~
	~\eqref{eq:main1}. \rev{This will allow us to %state our main result, namely, 
	{\color{black}{prove}} existence and uniqueness of stable three-time-scale relaxation oscillations in the singularly perturbed formulation %of \eqref{eq:main2} and 
	\eqref{eq:main1}.}

	 \revsj{It is worth noting (see also Section \ref{ssec:statement_of_the_main_result} and Remark \ref{rem:sing_analysis} below), that a simple existence result for the limit cycle can also be derived using standard phase plane arguments and the Poincar\'e-Bendixson theorem. While such an approach has the advantage of simplicity, singular perturbation analyses are typically better suited to uncovering detailed information about the geometric and multiple time-scale structure of the oscillations. Our analysis allows for a clear identification of the dominant physiological processes controlling each phase of the limit cycle. There are also methodological motivations for the use of singular perturbation arguments. In particular, the methods developed and applied herein can be readily adapted and applied to multiple time-scale problems in higher dimensions.}

%	{\color{red}{I think this is a good place to briefly address the referee's comments about it being possibly easier to do all this with Poincare-Bendixson or other standard stuff.  Maybe say something like "As will be discussed further in Section 5.3 below, there are other, apparently simpler, ways to prove the existence of ... However, the advantages of our approach include the ability to clearly identify the dominant physiological processes in each phase of the limit cycle and the ability to generalise the method to higher-dimensional systems; see xxx (a reference to a Remark number if the commentary on this stays as a Remark).}}
	%%%
	%\begin{comment}
	%As we have emphasised throughout, the system \eqref{eq:main1} is \emph{not} in the standard form for singular perturbation problems
	%\begin{align}\label{eq:stnd_form}
	%\begin{array}{lcl}
	%x'=\epsilon \tilde g(x,y,\epsilon), \\
	%y'=\tilde f(x,y,\epsilon),
	%\end{array}
	%\end{align}
	%i.e. the multiple time-scale structure is \wm{non-uniform} in the sense that it cannot be (globally) represented in terms of a splitting into slow and fast variables. In fact system \eqref{eq:main1} respectively \eqref{eq:main2}
	%\end{comment}
	%%%
	%which may be considered as problems in the general form \eqref{eq:non_standard_GSPT}. Explicitly, 
	
	\
	
	\rev{System \eqref{eq:main1} is not in the standard form for slow-fast systems, but} can be written in the general form %\eqref{eq:non_standard_GSPT} as
	\begin{equation}\label{eq:general_form}
		\begin{pmatrix}
			h' \\
			c'
		\end{pmatrix}
		=g(h,c)+\sum_{j\geq 1} \epsilon^j W_j(z) = N(h,c)f(h,c)+\epsilon G(h,c,\epsilon), \qquad 0 < \epsilon \ll1,
	\end{equation}
	where $g(h,c) := N(h,c)f(h,c)$, with
	\[
	N(h,c) =
	\begin{pmatrix}
		- \tau_{\rm max}^{-1} \\
		\mathfrak{J}^{(0)}_{\text{IPR}}(h,c)
	\end{pmatrix},
	\qquad
	f(h,c)=c^4h,
	\]
	and $G(h,c,\epsilon) := \sum_{j\geq 1} \epsilon^j W_j(z)$ is given by
	\[
	G(h,c,\epsilon) =
	\begin{pmatrix}
		0 \\
		- \mathfrak J^+_{\text{\rm SERCA}}(c)
	\end{pmatrix}
	+ \epsilon
	\begin{pmatrix}
		\tau_{\rm max}^{-1}(K_h^4 - h) \\
		\mathfrak{J}^{(1)}_{\text{IPR}}(h,c)c^8h + \mathfrak J^-_{\text{\rm SERCA}}(c)
	\end{pmatrix}
	+ \epsilon^3
	\begin{pmatrix}
		R_{h_\infty}(c,\epsilon) \\
		R_{\text{IPR}}(h,c,\epsilon)
	\end{pmatrix}.
	\]
	Singular perturbation problems in the general \rev{(non-standard)} form \eqref{eq:general_form} frequently arise in applications, and have been studied using a combination of GSPT and blow-up techniques in, e.g.,~\cite{Gucwa2009,Huber2005,kosiuk2016geometric,KuehnSzmolyan2015,Lizarraga2020,Lizarraga2020b}. For a formal introduction to the use of GSPT in analysing problems in the general form \eqref{eq:general_form}, we refer the reader to \cite{jelbart2019two} for the planar case, and to \cite{Goeke2014,wechselberger2018geometric} for the general (dimension $n\geq 2$) case. See also \cite{Kruff2019,Lizarraga2020b,Lizarraga2020c} for applications to systems with $m>2$ time-scales.

	\rev{It is important to realise that the presence of switching in system \eqref{eq:general_form} leads to distinct limiting systems depending on calcium concentration $c$. Specifically, for {\color{black}{small $c$, if we write $c = \sqrt{\epsilon}C$ with $C= O(1)$, we have
	\[
	\lim_{\epsilon \to 0} \mathfrak h_\infty(\sqrt{\epsilon}C,\epsilon) = \frac{K_h^4}{K_h^4 + C^4} ,
	\]
	while for $c = O(1)$,}} 
	the limit $\lim_{\epsilon \to 0}\mathfrak h_\infty(c,\epsilon)$ is given by the second expression in \eqref{eq:switch_approx}. Such observations lead to the identification of two non-overlapping regimes:
	
	\smallskip
	
	\begin{enumerate}
		\item[(R1)] $c \in \mathcal I_1=[a,b]$ for some fixed $a$ and $b$ with $0<a<b$;
		
		\smallskip
		
		\item[(R2)] $c \in \mathcal I_2=[0,\sqrt{\epsilon}\nu]$ for some fixed $\nu>0$. 
	\end{enumerate}
	
	\medskip
	
	\noindent For sufficiently small $\epsilon \ll 1$, $c = O(1)$ in (R1) and $c = O(\sqrt{\epsilon})$ in (R2).} Singular limit analyses in regimes (R1) and (R2) will provide sufficient information for the statement of our main result in Section \ref{ssec:statement_of_the_main_result}. % i.e.~the existence of an attracting relaxation cycle in \eqref{eq:main1}. 
	The observed dynamics and corresponding analysis is qualitatively similar, but not identical, to the autocatalator model considered in \cite{Gucwa2009}. \SJ{Similar systems have also been studied in the (R2) regime only in \cite{Kuehn2014}.}

	\rev{
		\begin{remark}
			%Due to the switching properties of system \eqref{eq:general_form}, t
			The right-hand-side in \eqref{eq:general_form} is $C^\infty$ for each fixed $0 < \epsilon \ll 1$, but only $C^{r}$ {\color{black}{(for some $r \geq 0$)}} in the limit $\epsilon \to 0$ due to a jump discontinuity in the switch terms. In general, the order $r$ is important since it imposes analytical and methodological constraints, particularly if $r=0$. In the case of system \eqref{eq:main1}, we have $r=2$ because of the term $\mathfrak h_{\infty}(c,\epsilon) \epsilon^2$, since $\mathfrak h_{\infty}(c,\epsilon)$ is discontinuous at $c=\epsilon=0$. The case $r = 0$ can arise in applications \cite{Kristiansen2019d}, in which case, the resulting perturbation problem can be understood as a smooth perturbation of a piecewise-smooth dynamical system.
			% It is worthy to note that even in such cases, certain adaptations of the blow-up method can often be used in order to study the dynamics of such systems from a geometric viewpoint \cite{Llibre2007,Llibre2009,kristiansen2018a,Kristiansen2020,Kristiansen2019d}.
		\end{remark}
	
%	\begin{remark}
%		The fact that $r \geq 1$ is significant analytically, since it means that the system can be studied using perturbation methods including GSPT. It is worthy to note, however, that the $r = 0$ can also arises in applications \cite{Kristiansen2019d}. In this case, the resulting perturbation problem can be viewed as a (smooth) perturbation of a piecewise-smooth dynamical system. It is worthy to note that even in such cases, certain adaptations of the blow-up method can often be used in order to study the dynamics of such systems from a geometric viewpoint \cite{Llibre2007,Llibre2009,kristiansen2018a,Kristiansen2020,Kristiansen2019d}.
%	\end{remark}
	}

	\subsection{\wm{Multiple time-scale} analysis in regime (R1)}
	\label{ssec:singular_limit_analysis_R1}

	We first consider the dynamics in regime (R1). Setting $\epsilon=0$ in \eqref{eq:general_form} gives the \emph{layer problem},
	\begin{equation}\label{eq:layer_R1}
		\begin{pmatrix}
			h' \\
			c'
		\end{pmatrix}
		=
		N(h,c) f(h, c)
		=
		\begin{pmatrix}
			- \tau_{\rm max}^{-1} \\
			\mathfrak{J}^{(0)}_{\text{IPR}}(h,c)
		\end{pmatrix}c^4h,
	\end{equation}
	which has two lines of equilibria, or \emph{critical manifolds},
	\begin{equation}
		\label{eq:manifolds_R1}
		S_c = \left\{(0, c) : c \geq 0 \right\}, \qquad S_h = \left\{(h, 0) : h \geq 0 \right\} .
	\end{equation}
	
	\begin{comment}
	\begin{rem}
	The presence of critical manifolds $S_c$ and $S_h$ \textit{define} system \eqref{eq:manifolds_R1} as a singular (as opposed to regular) perturbation problem in the GSPT sense, originally defined in Fenichel's seminal work \cite{fenichel1979geometric}.
	\end{rem}
	\end{comment}
	
	\begin{rem}
		In contrast to standard form slow-fast systems (see,  e.g.,~\cite{Kuehn2015}),
		%	\begin{align}\label{eq:stnd_form}
		%	\begin{array}{lcl}
		%	x'=\epsilon \tilde g(x,y,\epsilon), \\
		%	y'=\tilde f(x,y,\epsilon),
		%	\end{array}
		%	\end{align}
		there is no globally defined slow variable in system \eqref{eq:main1}. Hence the layer problem \eqref{eq:layer_R1} is \emph{not} a bifurcation problem in either $h$ or $c$.
	\end{rem}
	
	\begin{comment}
	\begin{rem}
	\label{rem:auxilliary_system}
	Orbits of the layer problem \eqref{eq:layer_R1} are described by the auxiliary system $(h',c')^T = N(h,c)$, which is orbitally equivalent to system \eqref{eq:layer_R1} on $\{(h,c):h>0,c>0\}$ via a transformation of time $d\hat t = f(h,c) dt$. % We will often use the fact that these systems are orbitally equivalent away from the set $S = \{(h,c) : f(h,c) = 0\}$.
	\end{rem}           
	\end{comment}              
	
	%In considering stability properties of the manifolds $S_c$ ($S_h$), 
	Evaluating the Jacobian at any given $p\in S_c$ or $S_h$ yields a trivial eigenvalue $\lambda_0 \equiv 0$, whose corresponding eigenvector spans the corresponding tangent space at $p$. Direct calculations 
	%show that $S_c$ (resp.~$S_h$) also has an 
	\wm{provide the} associated non-trivial eigenvalue $\lambda_c(c)$ (resp.~$\lambda_h(h)$) \wm{of $S_c$ (resp.~$S_h$)}, given by the Lie derivative expressions
	\begin{equation}
		\label{eq:evs_R1}
		%\begin{split}
		\lambda_c(c) = \langle \nabla f,N\rangle \big |_{S_c} = - \tau_{\rm max}^{-1} c^4 \leq 0 ,  \qquad
		\lambda_h(h) = \langle \nabla f,N\rangle \big |_{S_h} \equiv0,
		%\end{split}
	\end{equation}
	see also \cite[eqn.~(3.7)]{jelbart2019two}.
	%
	\begin{comment}
	\begin{lemma} (\cite{jelbart2019two,wechselberger2018geometric}). 
	Given a planar singular perturbation problem in the general form \eqref{eq:general_form} with critical manifold $S = \{(x,y) : f(x, y) = 0 \}$. Then the \wm{non-uniform} eigenvalue at a point $p \in S$ is given by the inner product
	\begin{equation}\label{eq:nontrivial_ev}
	\lambda(p) = \left\langle \nabla f(p), N(p) \right\rangle .
	\end{equation}
	The corresponding eigenvector lies within $\mathcal N_p = span (N(p))$.
	\end{lemma}
	
	\begin{proof}
	Linearisation; see also \cite[eqn.~(3.7)]{jelbart2019two}.
	\end{proof}
	\end{comment}
	%
	%In agreement with common nomanclature, points $p \in S$ for which $\lambda(p) \neq 0$ are called \textit{normally hyperbolic}. Subsets $S_n \subseteq S$ consisting entirely of such points are also referred to as normally hyperbolic. 
	%
	\begin{comment}
	\begin{rem}
	\label{rem_contact}
	In the case that $p\in S$, $\nabla f(p) \neq 0$ and $N(p) \neq 0$, equation \eqref{eq:nontrivial_ev} implies that a loss of normal hyperbolicity occurs at a tangency between the fast dynamics and the critical manifold $S$ (c.f. the loss of normal hyperbolicity at a fold point in standard form problems \eqref{eq:stnd_form}).
	\end{rem}
	\end{comment}
	%
	%In our case, we obtain
	%\begin{equation}
	%\label{eq:evs_R1}
	%\begin{split}
	%\lambda_c(h) &= \langle \nabla f(h,0),N(h,0)\rangle\equiv0, \\
	%\lambda_h(c) &= \langle \nabla f(c,0),N(c,0)\rangle= - c^4 \leq 0 .
	%\end{split}
	%\end{equation}
	Hence, the manifold $S_c$ is normally hyperbolic and attracting for all $c > 0$, and degenerate at $(0,0)$. The manifold $S_h$ is degenerate.
	
	Fast fibers constitute non-trivial heteroclinic connections between $S_h$ and $S_c$ in the layer problem \eqref{eq:layer_R1}. In particular, they obey the equation
	\begin{equation}\label{eq:fibers_R1}
		\frac{\text{d}c}{\text{d}h} = - \tau_{\rm max} \mathfrak{J}^{(0)}_{\text{IPR}}(h,c)
		\begin{cases}
			< 0 \qquad & \text{if} \qquad c < \left(\frac{\gamma}{1+\gamma} \right)c_t, \\
			= 0 \qquad & \text{if} \qquad c = \left(\frac{\gamma}{1+\gamma} \right)c_t, \\
			> 0 \qquad & \text{if} \qquad c > \left(\frac{\gamma}{1+\gamma} \right)c_t,
		\end{cases}
	\end{equation}
	%i.e.~their orbits are determined by the auxilliary problem $(h', c')^T = N(h,c)$; see Remark \ref{rem:auxilliary_system} and
	see Figure \ref{SingLimFig}.
	
	\begin{figure}[t!]
		\centering
		%		\begin{subfigure}[b]{0.48\linewidth}
		\includegraphics[scale=0.23]{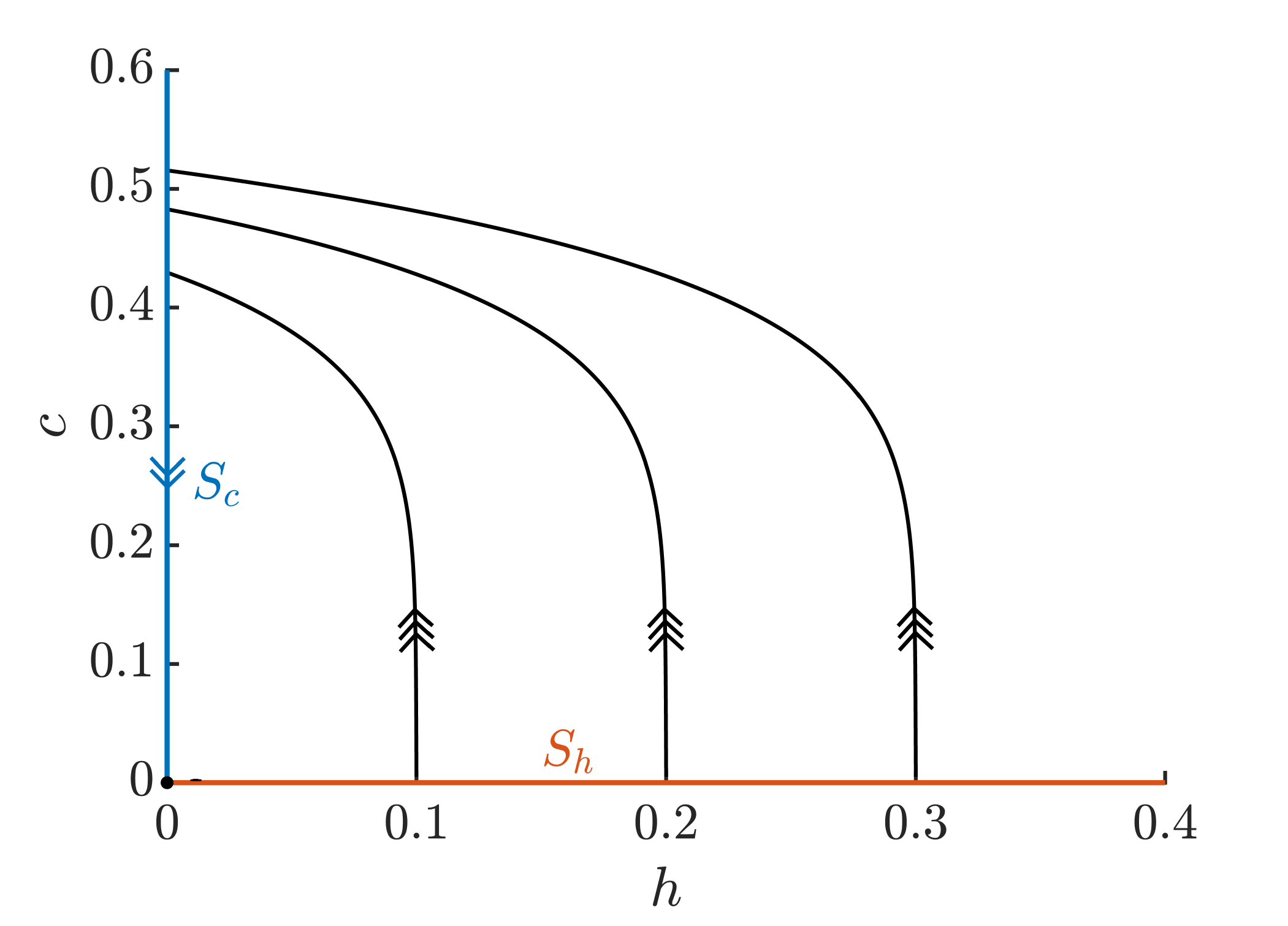} 
		\includegraphics[scale=0.23]{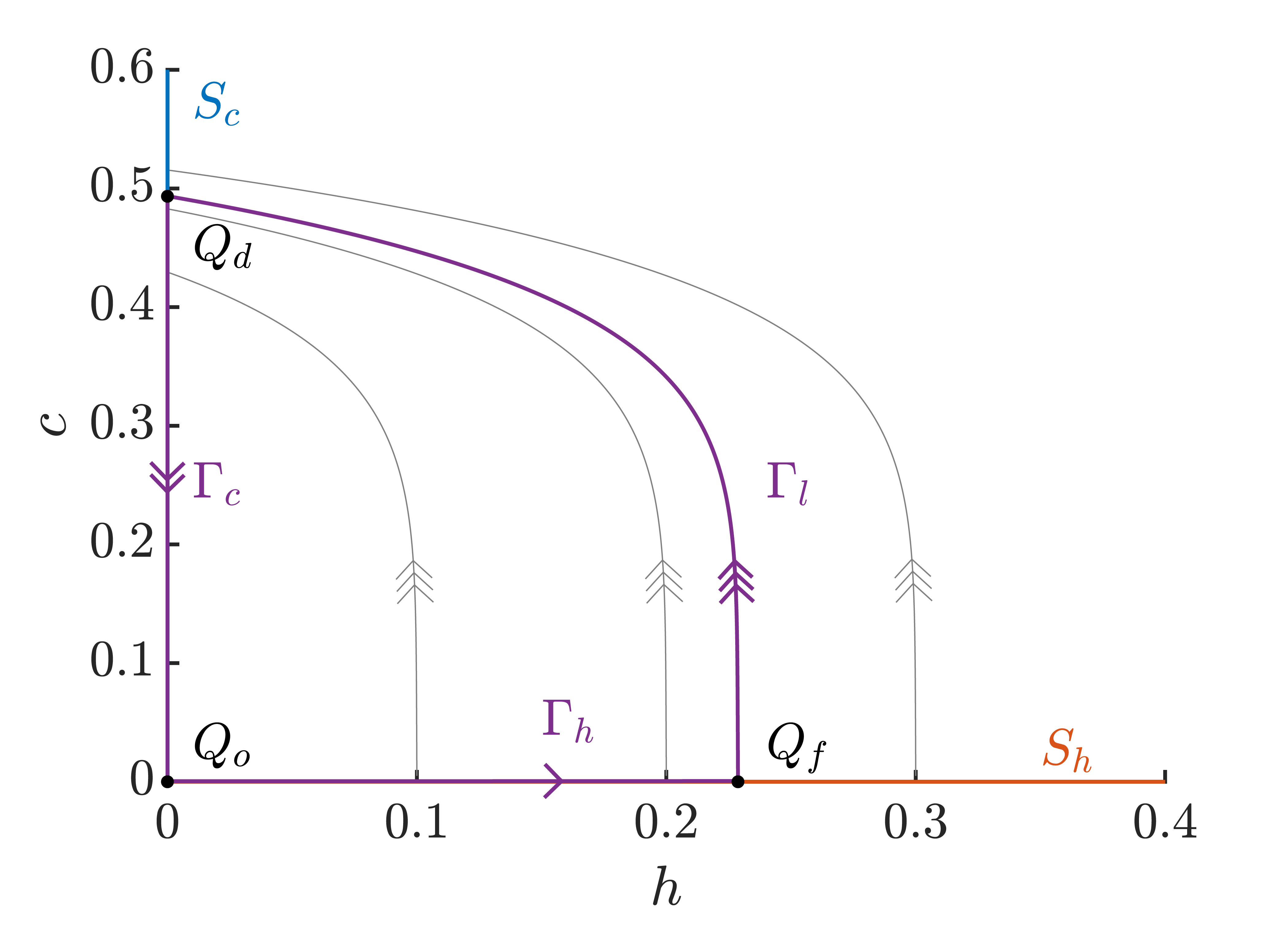}
		%		\caption{\textbf{[Old one]}}
		%		\end{subfigure}
		%		\begin{subfigure}[b]{0.48\linewidth}
		%		\includegraphics[width=\linewidth]{Images/CalciumSingFig_New.png}
		%		\caption{\textbf{[New one]}}
		%		\end{subfigure}
		\caption{Singular limit dynamics. (Left) The (nonlinear) fast connection from $S_h$ to $S_c$ is determined by the layer problem \eqref{eq:layer_R1} and the flow on $S_c$ is given by the reduced problem \eqref{eq:reduced_R1}. The manifold $S_h$ is degenerate. \wm{(Right) The singular orbit $\Gamma=\Gamma_h \cup \Gamma_c \cup \Gamma_l$. Note, the flow indicated on $S_h$ is not formally defined, and included for illustrative purposes only; it represents the rightward motion along $\mathcal S_{a,\delta}$ in Figure \ref{fig:manifold_R2}, and helps to illustrate the three-time-scale structure of the oscillations.}}\label{SingLimFig}
	\end{figure}
	
	\begin{rem}
		\label{rem:return}
		The expressions in \eqref{eq:fibers_R1} are not on their own sufficient to prove that orbits of the layer problem connect points on $S_h$ and $S_c$. A proof of this property %is deferred 
		\rev{follows from Lemma \ref{lem:map12} in Appendix \ref{app:proof_theorem}, and is} omitted here for expository reasons.
	\end{rem}
	
	\wm{Next, we consider \eqref{eq:general_form} on the slow time-scale $\tau = \epsilon t$, i.e.,
		\begin{equation}\label{eq:general_form_slow}
			\epsilon
			\begin{pmatrix}
				\dot{h} \\
				\dot{c}
			\end{pmatrix}
			= N(h,c)f(h,c)+\epsilon G(h,c,\epsilon), \qquad 0 < \epsilon \ll1.
		\end{equation}
		The following result (see \cite{jelbart2019two}) considers the limit $\epsilon \to 0$ of \eqref{eq:general_form_slow}, called the \textit{reduced problem}, which describes the leading order slow flow on $S_c$:}
	%The aim is to obtain a description of the flow on
	%normally hyperbolic submanifolds of $S_c$. 
	%This is provided by the following result from \cite{jelbart2019two}; see also \cite{fenichel1979geometric,goeke2014constructive,wechselberger2018geometric} for similar methods applicable in the dimension $n > 2$ case.
	
	\begin{lemma}
		\label{lem:reduced}
		Given a planar singular perturbation problem in the general form \eqref{eq:general_form}, with critical manifold $S$, and normally hyperbolic submanifold $S_n \subset S$. Then the reduced problem on $S_n$ is given by
		\begin{equation}\label{eq:reduced_eqn}
			\begin{pmatrix}
				\dot h \\
				\dot c
			\end{pmatrix}
			=
			\left[\frac{\det (N|G)}{\left\langle \nabla f, N \right\rangle}
			\begin{pmatrix}
				-D_cf \\
				D_hf
			\end{pmatrix}\right]
			\bigg|_{S_n} ,
		\end{equation}
		where the $(\dot{\ })$ notation denotes differentiation with respect to slow time $\tau$, and the determinant $\det (N|G)$ is taken with respect to the matrix with columns $N(h,c)$ and $G(h,c,0)$.
	\end{lemma}
	
	\begin{proof}
		The expression in \eqref{eq:reduced_eqn} is obtained by projecting the leading order perturbation vector $G(h,c,0)$ at each point $p \in S_c \setminus \{(0,0)\}$ onto its component in the corresponding tangent space $T_pS_c$; see \cite[Proposition 3.4]{jelbart2019two}. \wm{See also \cite{fenichel1979geometric,Goeke2014,wechselberger2018geometric} for similar methods applicable in the} case that the dimension $n$ is  greater than two.
	\end{proof}
	
	\begin{comment}
	\begin{rem}
	Lemma \ref{lem:reduced} applies only for normally hyperbolic submanifolds of the critical manifold $S$. There is no well-defined reduced problem on the degenerate line $S_h$, which by expression \eqref{eq:evs_R1} is everywhere nonhyperbolic.
	\end{rem}
	\end{comment}
	
	%To obtain the reduced problem, we project the perturbation $G(h,c,0)$ at each $p\in S_c$ onto its component in the tangent space $T_pS_c$. This can be done directly using the formula (3.19) from \cite{jelbart2019two}, and by doing so we obtain
	%%%
	\begin{comment}
	\begin{equation}\label{eq:reduced_R11}
	\begin{pmatrix}
	\dot h \\
	\dot c
	\end{pmatrix}
	=
	\left[\frac{1}{\lambda_h(c)}\det (N|G)
	\begin{pmatrix}
	-D_cf \\
	D_hf
	\end{pmatrix}\right]
	\bigg|_{h=\epsilon=0}
	=
	\begin{pmatrix}
	0 \\
	- \mathfrak J^+_{\rm SERCA}(c)
	\end{pmatrix},
	\end{equation}
	\
	where the overdot notation denotes differentiation with respect to slow time $\tau$, and the determinant $\det (N|G)$ is taken with respect to the matrix with columns $N(h,c)$ and $G(h,c,0)$.
	\end{comment}
	%%%
	We obtain an expression for the reduced problem on $S_c \setminus \{0,0\}$ directly via equation \eqref{eq:reduced_eqn}, which gives
	\begin{equation}\label{eq:reduced_R1}
		\begin{pmatrix}
			\dot h \\
			\dot c
		\end{pmatrix}
		=
		\begin{pmatrix}
			0 \\
			- \mathfrak J^+_{\text{\rm SERCA}}(c)
		\end{pmatrix} .
	\end{equation}
	In particular, $- \mathfrak J^+_{\text{\rm SERCA}}(c) < 0$ for all $c >0$ and $- \mathfrak J^+_{\text{SERCA}}(0) = 0$; the reduced flow on $S_c$ is toward the origin. The origin itself can be considered as a non-hyperbolic equilibrium for the reduced problem extended to all of $S_c$; see Figure \ref{SingLimFig}.

	\
	
	The preceding analysis for both layer and reduced problems implies the following result %concerning the dynamics 
	for $0 < \epsilon \ll 1$, which follows by classical results due to Fenichel \cite{fenichel1979geometric}.
	
	\begin{lemma}\label{lem:slow_manifolds_R1}
		There exists an $\epsilon_0>0$ such that for all $\epsilon\in(0,\epsilon_0]$, compact submanifolds of $S_c$ %$\{(0,c):c\in[c_-,c_+]\}\subset S_c$ 
		perturb to $O(\epsilon)$-close locally invariant \wm{one-}dimensional slow manifolds% with graph representation
		\[
		S_{c,\epsilon}=\left\{(\tilde{\nu}(c,\epsilon),c):c\in[c_-,c_+]\right\},
		\]
		where $\tilde \nu(c,\epsilon)=O(\epsilon)$, and $c_\pm$ are positive constants satisfying $c_+ > c_- > 0$. The leading order flow on each $S_{c,\epsilon}$ is governed by the reduced problem \eqref{eq:reduced_R1}.
	\end{lemma}
	
	%\begin{proof}
	%	Fenichel theory \cite{fenichel1979geometric}.
	%\end{proof}
	
	It remains to understand the dynamics near the degenerate line $S_h$. For this we must look in regime (R2). %Looking at a rescaling of system \eqref{eq:main2} in R2 allows to understand the dynamics near $S_h$, but is not enough to prove the existence of relaxation oscillation linking regime R1 and R2. To do so a blow up analysis is carried in the appendix.

	% section singular_limit_analysis_1 (end)
	
	\subsection{\wm{Multiple time-scale} analysis in regime (R2)}
	\label{ssec:singular_limit_analysis_R2}

	We now consider the dynamics in regime (R2), where $c = O(\sqrt{\epsilon})$. \revsj{We work in the rescaled coordinates $(h,c) = (h,\sqrt{\epsilon} C)$ and define
	%In order to avoid fractional powers, 
	%We set
	\[
	\delta := \sqrt{\epsilon} 
	\]
	for simplicity, i.e.,~in order to avoid fractional exponents.
%	\wm{
%		Defining a rescaled variable $C=O(1)$ via {\color{red}{This looks like the coordinate change above when (R1) and (R2) were defined - surely we can do this once only. Also, would it be easier to understand if we defined $c=\sqrt{\epsilon} C$ and then say that we write $\delta=\sqrt{\epsilon}$ for simplicity? Otherwise the introduction of $\delta$ just looks arbitrary.}}
%		$$
%		c=\delta C\,,\qquad\mbox{where }\delta := \sqrt{\epsilon}
%		$$
%		and 
	Rewriting system \eqref{eq:general_form} in the new rescaled coordinates on an} `intermediate-slow' time-scale $t_1 = \epsilon^{3/2} t = \delta^3 t$ \revsj{yields} % in system  \eqref{eq:general_form},}
	%applying the time rescaling $dt=\delta^{-3}d\tilde t$, 
%	we obtain 
	%the following system from \eqref{eq:main2}, 
	%which describes the dynamics in regime (R2),
	\begin{equation}
		\label{eq:main_R2}
		\begin{pmatrix}
			h' \\
			C'
		\end{pmatrix}
		=
		\tilde N(h,C) \tilde f_0(h,C) + \delta \tilde G(h,C,\delta), \qquad 0 < \delta \ll 1 ,
	\end{equation}
	%\begin{equation}\label{eq:main_R2}
	%	\begin{split}
	%		h' &= \delta \tilde g(h,C) = - \frac{\delta}{\tau_h(C)} \left(h - h_\infty (C) \right) , \\
	%		C' &= \tilde f(h,C,\delta) = \tilde f_0(h,C) + \delta \tilde f_{\rm rem}(C,h,\delta) ,
	%	\end{split}
	%\end{equation}
	where
	\begin{equation}
		\tilde N(h,C) =
		\begin{pmatrix}
			0 \\
			1
		\end{pmatrix},
		\qquad
		\tilde f_0(h,C) = A_{\text{\text{IPR}}} \gamma c_t C^4 h - A_{\text{SERCA}} \left(C^2 - K \gamma^2 c_t^2 \right) ,
	\end{equation}
	%{\color{red}{VK: I could not reproduce the SERCA term here. I needed $A_{\text{SERCA}}=\frac{K \gamma^2V_s}{K_s^2}$ to make it work. Did I miss a simplification?}}
	and
	\begin{equation}
		\label{eq:K2_rhs}
		\tilde G(h,C,\delta) =
		\begin{pmatrix}
			\tilde g(h,C) \\
			\tilde f_{\rm rem}(C,h,\delta)
		\end{pmatrix}
		=
		\begin{pmatrix} 
			- \tau_h(C)^{-1} \left(h - h_\infty (C) \right)  \\
			- \left(A_{\text{\text{IPR}}} (1 + \gamma) C^4 h +2 A_{\text{SERCA}} c_t \right) C +  O\left(\delta\right) 
		\end{pmatrix} .
	\end{equation}
	In the above we defined
	\[
	A_{\text{\text{IPR}}} = \left(\frac{k_{\text{IPR}} p^2}{k_\beta K_c^4 K_p^2}\right) , \qquad 
	A_{\text{SERCA}}=\frac{V_s}{K_s^2},
	%A_{\text{SERCA}} = \frac{1}{K_s^2} , 
	%\qquad
	%h_\infty(C) = \frac{\hat K_h^4}{\hat K_h^4 + C^4} ,
	\]
	%with
	%\begin{equation}
	%\label{eq:K2_rhs}
	%\begin{split}
	%\tilde g(h,C) &= - \frac{1}{\tau_h(C)} \left(h - h_\infty (C) \right) , \\
	%\left(\frac{C^4 p^2 c_t \gamma \kipr}{k_\beta K_c^4 K_p^2}\right) \left[h-\left(\frac{k_\beta V_s K_c^4 K_p^2 }{C^4 p^2 c_t \gamma \kipr K_s^2}\right) \left(C^2 - \gamma ^2 c_t^2 K \right)\right],
	%\tilde f_{\rm rem}(h,C,\delta) &= - \left(A_{\text{\text{IPR}}} (1 + \gamma) C^4 h +2 A_{\text{SERCA}} c_t \right) C + O\left(\delta %, \delta C^2
	%\right) .
	%\end{split}
	%\end{equation}
	and permitted a slight abuse of notation by letting the prime ($'$) denote differentiation with respect to the new \wm{intermediate-slow  time $t_1$}. In writing \eqref{eq:K2_rhs}, we also appeal to the fact that $\mathfrak h_\infty(\delta C, \delta^2) = h_\infty(C)$. Finally, note that $\tilde f_0(h,C)$ is just the leading order term in an expansion in $\delta$, and that in this regime, % this expression has contributions from \text{IPR} as well as SERCA processes, i.e.~
	IPR and SERCA processes compete at leading order. % in this regime.
	
	\begin{rem}
		System \eqref{eq:main_R2} is in the so-called standard form for slow-fast systems. We present it here in the general form \eqref{eq:general_form}, and 
		%We present the system \eqref{eq:main_R2} in the general form \eqref{eq:general_form}. Note that this system is also in the standard form \eqref{eq:stnd_form}
		%despite being in the standard form \eqref{eq:stnd_form}
		%for slow-fast systems in the analysis to follow, we
		proceed via the same (more general) approach
		%In the following, however, we consider the system \eqref{} as being in the general form
		%\[
		%\begin{pmatrix}
		%h' \\
		%c'
		%\end{pmatrix}
		%=
		%\tilde N(h,c) \tilde f_0(h,c) + \delta \tilde G(h,c,\delta),
		%\]
		%with
		%\[
		%\tilde N(h,c) =
		%\begin{pmatrix}
		%0 \\
		%1
		%\end{pmatrix},
		%\qquad
		%\tilde G(h,c,\delta) =
		%\begin{pmatrix}
		%\tilde g(h,C) \\
		%\tilde f_{\rm rem}(C,h,\delta)
		%\end{pmatrix},
		%\]
		%and present an analysis in the (more general) approach
		adopted in Section \ref{ssec:singular_limit_analysis_R1}. Our reasons for doing so are three-fold: (i) consistency with earlier sections is maintained; (ii) there are no additional technical difficulties, and (iii) this approach illustrates the relationship between the theory developed for standard form problems and its non-standard form generalisation.
	\end{rem}
	%and can be rewritten the following more convenient form:
	%\begin{equation}\label{eq:main_R2}
	%	\begin{split}
	%		h' &= -\delta \varphi_1(C) \left( h - \varphi_2(C)\right) , \\
	%		C' &= \varphi_3(C)\left(h-\zeta(C)\right)+ \delta \tilde f_{\rm rem}(h, C, \delta),
	%	\end{split}
	%\end{equation}
	%where $\varphi_i(C)$, $i=1,2,3$ are positive functions
	%\[
	%	\varphi_1(C) = \frac{C^4+K_\tau^4}{K_\tau^4 \tau_{\rm max}}, \qquad
	%	\varphi_2(C)= \frac{K_h^4}{C^4+K_h^4}, \qquad
	%	\varphi_3(C) = \frac{C^4 p^2 c_t \gamma \kipr}{k_\beta K_c^4 K_p^2},
	%\]
	%\textcolor{red}{\textbf{[the first two are simply related to $\tau_h(C)$ and $h_\infty(C)$; write it in terms of those]}} and
	%\begin{equation}\label{eq:manifold_R2}
	%	\zeta(C)=\frac{A}{C^4}(C^2-\gamma^2c_t^2 K), \qquad \text{where} \qquad A= \frac{k_\beta V_s K_c^4 K_p^2 }{p^2 c_t \gamma \kipr K_s^2}.
	%\end{equation}
	
	Setting $\delta = 0$ in \eqref{eq:main_R2} yields the layer problem
	\begin{equation}
		\label{eq:layer_R2}
		\begin{pmatrix}
			h' \\
			C'
		\end{pmatrix}
		=
		\tilde N \tilde f_0(h,C) 
		=
		\begin{pmatrix}
			0 \\
			1
		\end{pmatrix}
		\left(A_{\text{\text{IPR}}} \gamma c_t C^4 h - A_{\text{SERCA}} \left(C^2 -  K \gamma^2 c_t^2 \right)\right) ,
	\end{equation}
	%\begin{align}\label{eq:layer_R2}
	%	\begin{array}{lcl}
	%		h'=0, \\
	%		C'= A_{ipr} \gamma c_t C^4 h - A_{serca} \left(C^2 - \gamma ^2 c_t^2 K \right) ,
	%	\end{array}
	%\end{align}
	%obtained by setting $\epsilon = 0$ in \eqref{}.
	which has critical manifold 
	\begin{equation}
		\label{eq:S_R2}
		\mathcal S = \left\{\left(\zeta(C), C \right) : C \geq 0 \right\} , \qquad 
		\zeta(C) = \frac{A_{\text{SERCA}}}{A_{\text{IPR}} \gamma c_t C^4}(C^2-K \gamma^2c_t^2) ,
	\end{equation}
	%shown in Figure \ref{fig:manifold_R2}, 
	%where %we have defined
	%\begin{equation}
	%\label{eq:manifold_R2}
	%\zeta(C) = \frac{A_{SERCA}}{A_{\text{IPR}} \gamma c_t C^4}(C^2-\gamma^2c_t^2 K) ,
	%\end{equation}
	see Figure \ref{fig:manifold_R2}. The non-trivial eigenvalue along $\mathcal S$ is given by the Lie derivative expression
	\begin{equation}
		\label{EV_R2}
		\lambda_{\mathcal S}(C) = \left\langle \nabla f, N \right\rangle \big|_{\mathcal S} =  D_C\tilde f_0(\zeta(C), C) = \frac{2 A_{\text{SERCA}}}{C}\left(C^2-2 K\gamma^2 c_t^2\right) .
	\end{equation}
	The manifold $\mathcal S$ has a fold point at
	\begin{equation}
		\label{eq:Fold}
		F = (\zeta(C_F),C_F) = \left(\frac{A_{\text{SERCA}}}{4 K \gamma^3 c_t^3 A_{\text{IPR}}}, \sqrt{2 K}\gamma c_t  \right),
	\end{equation}
	where $\lambda_{\mathcal S}(C_F)=0$ and the following nondegeneracy conditions are satisfied:
	\begin{equation}\label{eq:fold_nondegeneracy}
		%\begin{split}
		D_C^2\tilde f_0(\zeta(C_F), C_F) = 4 A_{\text{SERCA}} > 0, \ \ \ 
		D_h\tilde f_0(\zeta(C_F), C_F) = 4 A_{\text{IPR}} K^2 \gamma^5 c_t^5 >0.
		%\end{split}
	\end{equation}
	The first inequality in \eqref{eq:fold_nondegeneracy} implies that $\mathcal S=\mathcal S_a\cup\{F\}\cup \mathcal S_r$, where
	\[
	\mathcal S_a=\left\{(\zeta(C),C):0<C<C_F\right\}, \qquad \mathcal S_r=\left\{(\zeta(C),C):C>C_F\right\},
	\]
	and $\mathcal S_a$ (resp.~$\mathcal S_r$) is normally hyperbolic and attracting (resp.~repelling), as sketched in Figure \ref{fig:manifold_R2}.

	\begin{comment}
	\begin{rem}
	Recall by Remark \ref{rem_contact} that normal hyperbolicity is lost at points of tangency between the layer flow and the critical manifold. The loss of normal hyperbolicity at a regular fold point $F \in \mathcal S$ in particular is due to a \textit{quadratic} tangency with between the layer flow and the critical manifold $\mathcal S$. For problems given in the general form \eqref{eq:general_form}, such points are more generally referred to as \textit{order one contact points}. These are identified as points satisfying $\lambda_{\mathcal S}(F) = 0$, together with nondegeneracy conditions which reduce to those in \eqref{eq:fold_nondegeneracy} in the case of standard form problems \eqref{eq:stnd_form} (i.e. in the case of a regular fold); see \cite{jelbart2019two} for the dimension $n=2$ case, and \cite{wechselberger2018geometric} for the general case.
	\end{rem}
	\end{comment}

	\begin{figure}[t!]
		\centering
		%	\begin{subfigure}[b]{0.48\linewidth}
		\includegraphics[scale=0.35]{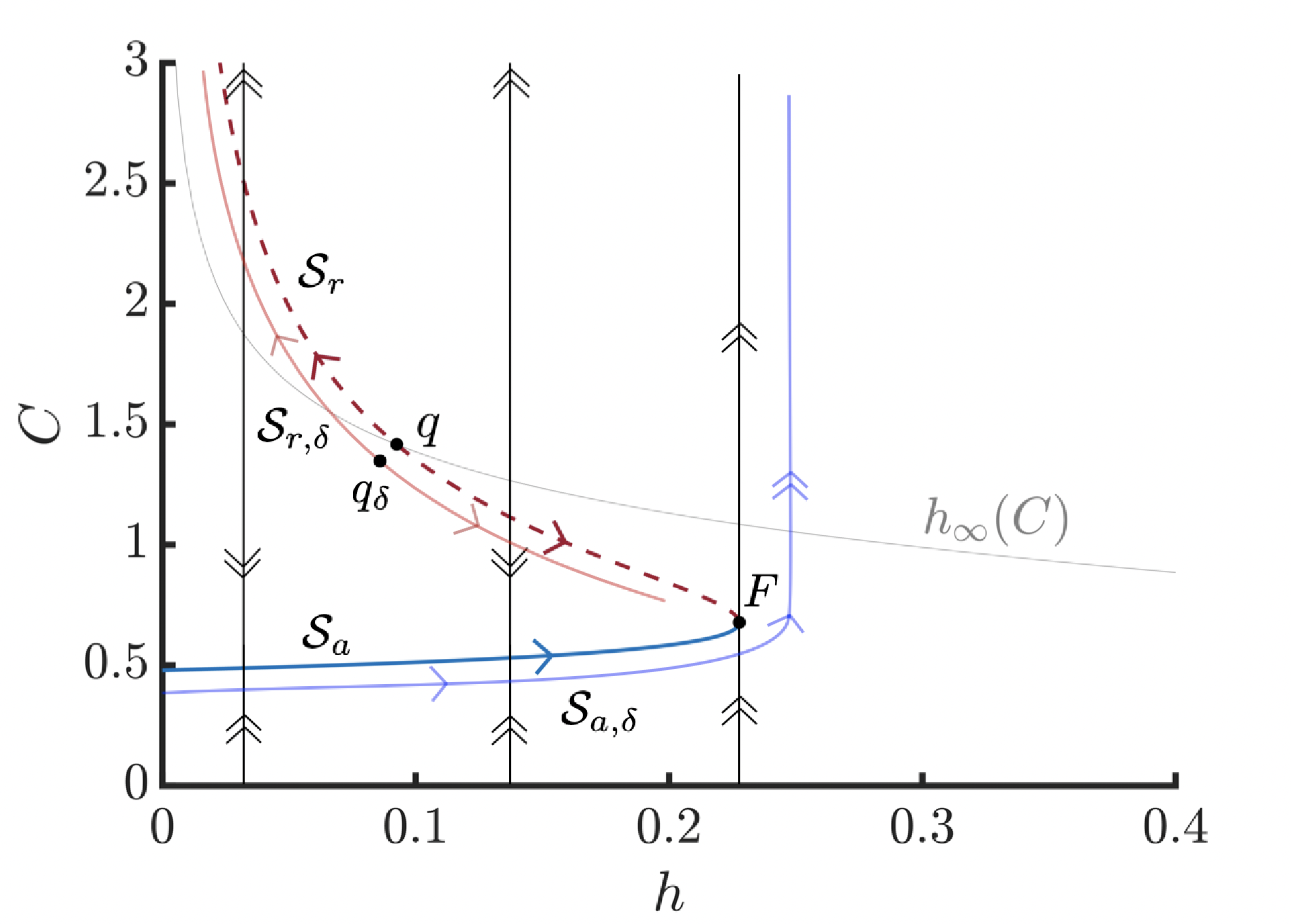}
		%		\caption{\textbf{[Old one]}}
		%		\end{subfigure}
		%		\begin{subfigure}[b]{0.48\linewidth}
		%		\includegraphics[width=\linewidth]{Images/chart_R2_New.png}
		%		\caption{\textbf{[New one]}}
		%		\end{subfigure}
		\caption{Sketch of the dynamics in regime (R2). The critical manifold $\mathcal S=\mathcal S_a\cup\{F\}\cup \mathcal S_r$ is shown with $\mathcal S_a$ in blue and $\mathcal S_r$ in dashed red. The fold point $F$, equilibrium $q$ and $h-$nullcline (in grey) are also indicated. Perturbed slow manifolds $\mathcal S_{a,\delta}$ and $\mathcal S_{r,\delta}$ are shown in shaded blue and red respectively, with the attracting slow manifold $\mathcal S_{a,\delta}$ extending through the neighbourhood of the fold point $F$ as described by Lemma \ref{lem:slow_manifolds_R2}.}
		\label{fig:manifold_R2}
	\end{figure}

	Now consider the reduced problem on the manifold $\mathcal S$, \wm{which describes the leading order dynamics on the infra-slow time-scale $\tau_1 = \delta t_1 = \epsilon^2 t$}. As in Section \ref{ssec:singular_limit_analysis_R1}, we use the formula \eqref{eq:reduced_eqn} to \wm{derive} the reduced problem,
	\begin{equation}\label{eq:reduced_R2}
		\begin{pmatrix}
			\dot h \\
			\dot C
		\end{pmatrix}
		=
		\left[\frac{\det (\tilde N| \tilde G)}{\langle \nabla \tilde f_0, \tilde N \rangle}
		\begin{pmatrix}
			-D_c \tilde f_0 \\
			D_h \tilde f_0
		\end{pmatrix}\right]
		\bigg|_{\mathcal S} 
		=
		\begin{pmatrix}
			1 \\
			-\frac{D_h\tilde f_0(\zeta(C),C)}{\lambda_{\mathcal S}(C)}
		\end{pmatrix}
		\tilde g(\zeta(C),C),
	\end{equation}
	where by another (slight) abuse of notation, $(\dot{\ })$ refers to differentiation with respect to %a new \wm{infra-slow time-scale VK: removed this since it seems redundant given definition of infra-slow just above.
	the infra-slow time-scale $\tau_1$. Explicitly, we obtain the system
	%\begin{equation}\label{eq:reduced_R2}
	%\begin{split}
	%\dot h &= - \frac{1}{\tau_h(C)} \left(\zeta(C) - h_\infty (C) \right) , \\
	%\dot C &= \frac{A_{ipr} \gamma c_t  C^5}{2 A_{serca} \tau_h(C) (C^2 - 2 K \gamma^2 c_t^2)} \left( \zeta(C) - h_\infty(C)\right). 
	%\end{split}
	%\end{equation}
	\begin{equation}\label{eq:reduced_R2_explicit}
		\begin{pmatrix}
			\dot h \\
			\dot C
		\end{pmatrix}
		=
		\begin{pmatrix}
			- \tau_h(C)^{-1} \\
			\frac{A_{\text{IPR}} \gamma c_t  C^5}{2 A_{\text{SERCA}} \tau_h(C) (C^2 - 2 K \gamma^2 c_t^2)}
		\end{pmatrix}
		\left(\zeta(C) - h_\infty (C) \right) .
	\end{equation}
	
	\begin{rem}
		\label{rem:time-scales}
		In total we have identified \emph{four} \wm{distinct} time-scales in system \eqref{eq:general_form} that are involved in our analysis \wm{(fast, slow, intermediate-slow, infra-slow)}:
		\[
		t ,%= O(1), 
		\qquad \tau = \epsilon t, % = O (\epsilon), 
		\wm{
			\qquad t_1 = \delta^3 t = \epsilon^{3/2} t , %O (\epsilon^{3/2}), 
			\qquad \tau_1 = \delta t_1 = \epsilon^2 t . %O(\epsilon^2).
		}
		\]
		%We prefer to interpret the \wm{intermediate-slow time-scale $t_1$} as a necessary intermediary between time-scales, rather than identify it with a particular physiological process.
		%	\WM{not sure if I understand that statement? Is it necessary here?}
	\end{rem}
	
	\begin{comment}
	\begin{rem}
	Equation \eqref{eq:reduced_R2} shows that the `usual' expression for the reduced problem associated with standard form problems \eqref{eq:stnd_form} (see, e.g. \cite{Kuehn2015}) is obtained directly from the more general expression \eqref{eq:reduced_eqn}.
	\end{rem}
	\end{comment}
	
	System \eqref{eq:reduced_R2_explicit} can have up to three equilibria $(\varphi(C_{\ast}), C_{\ast}) \in \mathcal S$ in the physiological domain $h, C \geq 0$. Their locations can be determined by solving the equation
	\[
	\zeta(C) = h_\infty (C) ,
	\]
	which reduces to the problem of identifying positive, real-valued roots of the cubic equation
	\begin{equation}\label{eqs:reduced_equilibria_R2b}
		P(m) = \alpha_3 m^3 + \alpha_2 m^2 + \alpha_1 m + \alpha_0 = 0,
	\end{equation}
	obtained after setting $m=C^2$. 
	\wm{
		The coefficients $\alpha_i$, $i=0,1,2,3$ are given by}
	\begin{equation}
		\begin{aligned}
			\alpha_3 &= \frac{A_{\text{SERCA}}}{\gamma c_t A_{\text{\text{IPR}}}}, && \alpha_2 = - K_h^4 - \frac{A_{\text{SERCA}}}{ A_{\text{\text{IPR}}}} K \gamma c_t , \\
			\alpha_1 &= \frac{K_h^4 A_{\text{SERCA}}}{\gamma c_t A_{\text{\text{IPR}}}} , && \alpha_0 = - \frac{K K_h^4 \gamma c_t A_{\text{SERCA}}}{A_{\text{\text{IPR}}}} .
		\end{aligned}
	\end{equation}
	%see Appendix \ref{app:equilibria} for expressions for $\alpha_i$, $i = 0,1,2,3$. 
	%Since the expression for $\zeta(C)$ in \eqref{eq:S_R2} depends on the system parameters $(p,c_t)$, 
	Note that the number of real-valued roots for equation \eqref{eqs:reduced_equilibria_R2b} depends on the system parameters $p$ and $c_t$; \wm{see Figure~\ref{fig:cusp}, which identifies a cusp structure in $(p,c_t)$-space.}
	%\WM{Sam: if available, present cusp figure; see also caption Figure 8}
	%\WM{[NZ] In Fig 7, which panel do you prefer: left or right? VK: Left panel looks best to me.}
	
	\begin{figure}[t!]
		\centering
		\includegraphics[scale=0.8]{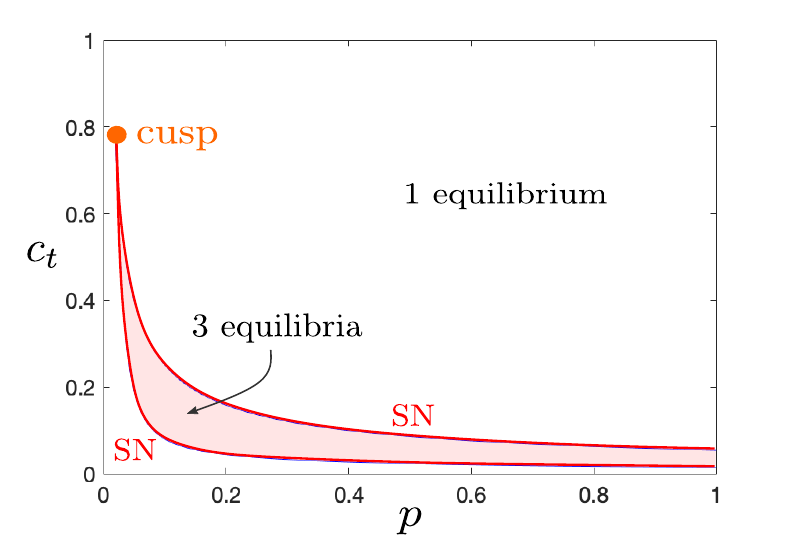}%\quad
		\caption{\wm{Number of equilibria in $(p,c_t)$-parameter space. The region with three equilibria  is bounded by two codimension-one saddle-node branches emanating from a codimension-two cusp point at $(p,c_t)\sim (0.02,0.8)$; outside this region there exists a unique equilibrium. Parameter values are as in Table \ref{tab:params_post_scaling}}.}
		\label{fig:cusp}
	\end{figure}

	Of course, the location and stability properties of the corresponding equilibria for system \eqref{eq:reduced_R2} will also depend on these parameters. We are interested %in the present work 
	{here in the case in which} system \eqref{eq:reduced_R2} has a unique, unstable equilibrium $q \in \mathcal S_r$. % This is made explicit in the following.

	\begin{assumption}
		\label{ass:1}
		The parameters $p$, $c_t$ are chosen so that the cubic polynomial $P(m)$ in \eqref{eqs:reduced_equilibria_R2b} has a unique positive real-valued root, such that system \eqref{eq:reduced_R2} has a unique equilibrium $q = (\varphi(C_{\ast}), C_{\ast}) \in \mathcal S$. Moreover, we require that 
		\begin{equation}\label{eq:stability_bounds}
			\frac{2 A_{\text{IPR}} K_h^4 \gamma c_t}{A_{\text{SERCA}}} > \frac{1}{C_\ast^8}\left(K_h^4 + C_\ast^4 \right) \left(C_\ast^2 - C_F^2 \right) > 0 ,
		\end{equation}
		implying that $q$ is unstable on $\mathcal S_r$.
	\end{assumption}
	
	The rightmost inequality in equation \eqref{eq:stability_bounds} implies that $C_\ast > C_F$, so that $q \in \mathcal S_r$. The leftmost inequality ensures that $D_c \dot C|_{C=C_\ast} > 0$, so that $q$ is unstable as an equilibrium on $\mathcal S$. %, which follows from the requirement that $D_c \dot C|_{C=C_\ast} > 0$. 
	\begin{comment}
	Explicitly, a little algebra yields the requirement for instability
	\begin{equation}
	\label{eq:stability_assumption}
	%\begin{split}
	D_c \dot C \big|_{C=C_\ast} 
	%= \frac{A_{ipr} C}{2 A_{serca} \tau_h(C) (C^2 - 2 K \gamma^2 c_t^2)} \left(D_c\zeta(C_\ast) - D_c h_\infty(C_\ast) \right) \\
	= \frac{1}{\tau_h(C_\ast)} \left(\frac{2 A_{\text{IPR}} K_h^4 C_\ast^8 \gamma c_t}{ A_{\rm SERCA} (K_h^4 + C_\ast^4) (C_\ast^2 - C_F^2)} - 1 \right) > 0 .   
	%\end{split}
	\end{equation}
	\end{comment}
	%
	%Under Assumption \ref{ass:1}, the equilbrium $q \in \mathcal S_r$ as depicted in Fig. \ref{}. Hence, it is either saddle type, or node type. We are nterested in the latter case, which obtains in the case that
	%
	%From this we obtain the following.
	%
	%\begin{assumption}
	%The parameters $p, c_t$ are chosen so that  the equilibrium $q$ is an unstable node. Explicitly, $p, c_t$ are such that the following condition holds:
	%\begin{equation}
	%\label{eq:stability_bounds}
	%\frac{2 A_{ipr} K_h^4 \gamma c_t}{A_{serca}} > \frac{1}{C_\ast^8}\left(K_h^4 + C_\ast^4 \right) \left(C_\ast^2 - C_F^2 \right) > 0.
	%\end{equation}
	%\end{assumption}
	%
	%In fact the rightmost inequality in \eqref{eq:stability_bounds} follows from Assumption \ref{ass:1}, since $q \in \mathcal S_r$ implies $C_\ast > C_F$.
	Note that Assumption \ref{ass:1} is satisfied for the parameter values in Table \ref{tab:params_post_scaling}, for which we have a unique equilibrium $q \approx (0.074,1.48) \in \mathcal S_r$ satisfying $C_\ast \approx 1.48 > C_F \approx 0.68$ and 
	\[
	\frac{2 A_{\text{IPR}} K_h^4 \gamma c_t}{A_{\text{SERCA}}} \approx 17.01 > \frac{1}{C_\ast^8}\left(K_h^4 + C_\ast^4 \right) \left(C_\ast^2 - C_F^2 \right) \approx 0.391 > 0 .
	\]
	Assumption \ref{ass:1} also implies 
	%also implies satisfaction of 
	\wm{that the regularity conditions
		\begin{equation}\label{eq:fold_regularity_R2}
			\tilde g(\zeta(C_F),C_F)\neq0, \qquad D_C \tilde g(\zeta(C_F),C_F)\neq0 ,
		\end{equation}
		at the fold point $F$ are fulfilled. 
		
		Note that} the reduced problem \eqref{eq:reduced_R2} is not defined at $F$, since the flow undergoes a finite time blow-up (solutions reach $F$ in finite time). The conditions in \eqref{eq:fold_regularity_R2} imply that the reduced flow is oriented locally toward $F$, classifying the regular fold as a \emph{regular jump point} \cite{Kuehn2015}, where there is a transition from slow to fast.
	
	\
	
	The dynamics for $0 < \delta \ll 1$ in compact subdomains \rev{in} regime (R2) \rev{are} well-described by existing theory, and summarised in the following result.
	%The first covers the dynamics bounded away from the regular jump point $F$, where Fenichel theory implies that compact submanifolds of $S_{a/r}$ perturb to nearby slow manifolds.
	
	%\WM{Sam: following statement is for \eqref{eq:main_R2} in R2, right? You use $C$, not $c$!}
	\begin{lemma}\label{lem:slow_manifolds_R2}
		Given system \wm{\eqref{eq:main_R2}}
		%	\eqref{eq:main2} respectively \eqref{eq:main1} 
		and Assumption \ref{ass:1}, there exists $\delta_0>0$ such that for all $\delta\in(0,\delta_0]$ compact submanifolds %$\{(\zeta(C),C):C\in[C_-,C_+]\} \subset \mathcal S_a$ and $\{(\zeta(C),C):C\in[C_l,C_r]\} \subset \mathcal S_r$
		of $\mathcal S_a$ and $\mathcal S_r$ perturb to $O(\delta)$-close locally invariant slow manifolds of the form
		\[
		%	\begin{split}
		\mathcal S_{a,\delta} = \left\{(\zeta(C)+O(\delta),C):C\in[C_-,C_+]\right\}, \ \ 
		\mathcal S_{r,\delta} = \left\{(\zeta(C)+O(\delta),C):C\in[C_l,C_r]\right\},
		%	\end{split}
		\]
		respectively, where $C_\pm$ and $C_{l,r}$ are any positive constants satisfying $0<C_-<C_+<C_F$ and $C_F<C_l<C_r<\infty$. The leading order flow on $\mathcal S_{a/r,\delta}$ is given by the reduced flow on $\mathcal S_{a/r}$, and the reduced flow equilibrium $q \in \mathcal S_r$ perturbs to a nearby equilibrium $q_\delta \in \mathcal S_{r,\delta}$ of unstable node type.
		
		The attracting slow manifolds $\mathcal S_{a,\delta}$ extend through a neighbourhood of the regular jump point, exiting along the fast flow which is $O(\delta^{2/3})$-close to the fast fiber $\{(\zeta(C_F),C):C>C_F\}$; see Figure \ref{fig:manifold_R2}.
	\end{lemma}
	
	\begin{proof}
		The dynamics bounded away from a neighbourhood of $F$, including the persistence of the unstable node $q_\delta$, is described by Fenichel theory \cite{fenichel1979geometric}. The dynamics near $F$ is described by \cite[Theorem 2.1]{Krupa2001a}.
	\end{proof}
	
	%The dynamics bounded away from a neighbourhood of $F$ are described by Fenichel theory. The dynamics near $F$ are described by the result in Krupa and Szmolyan \cite{Krupa2001a}.
	
	%\begin{theorem}\label{thm:jump}
	%	Given the system \eqref{eq:main1} respectively \eqref{eq:main2} and Assumption \ref{ass:1}. There exists $\tilde \delta_0>0$ such that for all $\delta\in(0, \tilde\delta_0]$ the attracting slow manifolds $\mathcal S_{a,\delta}$ extend through a neighbourhood of the regular jump point, exiting along the fast flow within a distance which is $O(\delta^{2/3})$-close to the fast fiber $\{(\zeta(C_F),C):C>C_F\}$.
	%\end{theorem}
	
	%\begin{proof}
	%	See Krupa and Szmolyan \cite{Krupa2001a}.
	%\end{proof}
	
	It follows that trajectories are either repelled to infinity or exponentially attracted to slow manifolds $\mathcal S_{a,\delta}$, after which they follow the slow flow up to the neighbourhood of the regular jump point $F$, before %where by Theorem \ref{lem:slow_manifolds_R2} they 
	leaving via the fast flow. Once on the fast flow, the global separation of slow and fast variables in regime (R2) prohibits the existence of a return mechanism on any compact domain. %; no cycles are possible in the regime (R2) alone. Hence, 
	Limit cycles for $0 < \delta \ll 1$ must traverse both regimes (R1) and (R2).
	%To observe a limit cycle as $\delta \to 0$, both regimes R1 and R2 are necessary.

	% section singular_limit_analysis_2 (end)

	\subsection{\revsj{Existence and uniqueness of the relaxation oscillations}}
	\label{ssec:statement_of_the_main_result}
	
	\revsj{We now use the results of the preceding sections in order to state an existence and uniqueness result for relaxation oscillations} 
%	We are now able to state our main result: existence of a limit cycle of relaxation type 
	in system \wm{\eqref{eq:general_form}}, in the $(p,c_t)-$parameter region specified by Assumption \ref{ass:1}. %{\color{red}{In light of the refocusing of the paper to highlight the procedure rather than the specific results for calcium, I think it might be helpful to think again about the title of this section and the sentence of introduction here. How about ``Existence and uniqueness of the relaxation oscillation" for a title, and introductory sentence along the lines of ``We now use the GSPT and the results of the previous subsections to prove the existence and uniqueness of a relaxation oscillation in system (...) in the parameter regime ....}}
	
	%A result pertaining to a relationship between the period of oscillations and the model parameter $\tau_{\rm max}$ is also given.
	
	%\WM{FYI, I have changed all references to (4.*) to (5.*) equations. Btw, I wonder if we should move this section to the beginning of section 5 (for discussion).}
	
	We consider system \wm{\eqref{eq:general_form}} in regime (R1), for which the fold point $F$ identified in regime (R2) `collapses' onto the point $Q_f = (h_F,0) = (\zeta(C_F),0)\in S_h$.
	%Heuristically, one can view the non-hyperbolicity of $S_h$ as a consequence of the contraction of $S_a$ onto $S_r$ as $\epsilon\to0$ \textcolor{red}{\textbf{[do side-by-side figures to illustrate this point?]}}.
	Having identified this point, we can construct the singular orbit % $\Gamma$ shown in Figure  by a concatenation 
	\[
	\Gamma=  \Gamma_c \cup \Gamma_h \cup \Gamma_l ,
	\]
	where
	\[
	\Gamma_h=\left\{(h,0):h\in[0,h_F] \right\}, \qquad \Gamma_c=\left\{(0,c):c\in[0,c_d] \right\},
	\]
	and $\Gamma_l$ is the (unique) heteroclinic orbit of the layer problem \eqref{eq:layer_R1} connecting $Q_f$ to the `drop point' $Q_d = (0, c_d) \in S_c$; see \wm{Figure~\ref{SingLimFig} (Right).} 
	%\ref{fig:singular_orbit_R1}. 
	It should be noted that because of curvature in the layer flow, existence of a drop point $Q_d \in S_c$ must be shown explicitly (recall Remark \ref{rem:return}). %; we defer the details to Lemma \ref{lem:map12} in the Appendix \ref{sec:poincare_map}.
	%\[
	%Q_f=(\zeta(C_F),0), \qquad Q_d = (h_d, 0), \qquad Q_0 = (0,0),
	%\]
	%see Figure \ref{SingLimFig} \textcolor{red}{\textbf{[Note that $\Gamma_l$ can be defined as the stable `pseudo-separatrix' emanting from $Q_f$, and that the drop point $Q_d$ should be defined more explicitly -- we might have to include a lemma to gurentee it's existence?..]}}.
	\begin{comment}
	\begin{figure}[t!]
	\centering
	%	\begin{subfigure}[b]{0.48\linewidth}
	\includegraphics[scale=0.25]{Images/fig_calcium_autocatalator_sing_lim.png}
	%		\caption{\textbf{[Old one]}}
	%		\end{subfigure}
	%		\begin{subfigure}[b]{0.48\linewidth}
	%		\includegraphics[width=\linewidth]{Images/fig_calcium_autocatalator_sing_lim_new.png}
	%		\caption{\textbf{[New one]}}
	%		\end{subfigure}
	\caption{The singular orbit $\Gamma=\Gamma_h \cup \Gamma_c \cup \Gamma_l$. Note that the flow indicated on $S_h$ is not formally defined, and included for illustrative purposes only; it represents the rightward motion along $\mathcal S_{a,\delta}$ in Figure \ref{fig:manifold_R2}, and helps to illustrate the three-time-scale structure of the oscillations. \wm{[show true relaxation oscillation, e.g., from Figure 3. Maybe merge with Figure 5?] }}
	\label{fig:singular_orbit_R1}
	\end{figure}
	\end{comment}
	
	\begin{theorem}\label{thm:main}
		Consider system \wm{\eqref{eq:general_form}} with $(p,c_t)$ fixed within the region specified by Assumption \ref{ass:1}. % and the parameter values in Table \ref{tab:params_post_scaling}, 
		There exists an $\epsilon_0>0$ such that for all $\epsilon\in(0,\epsilon_0)$, there exists a % locally unique  %and exponentially attracting 
		relaxation cycle $\Gamma_\epsilon$ which is $O(\epsilon^{1/3})-$close to the singular orbit $\Gamma$ in the Hausdorff distance as $\epsilon \to 0$. The relaxation cycle $\Gamma_\epsilon$ is exponentially attracting, with Floquet exponent bounded above by $-\kappa/\epsilon^2$ for some constant $\kappa>0$. \rev{Moreover, for any fixed $M > 0$ such that $\Gamma_\epsilon$ is contained within the ball $B(0,M)$, choosing $\epsilon_0>0$ sufficiently small guarantees that $\Gamma_\epsilon$ is the only limit cycle in $B(0,M)$.}
	\end{theorem}
	%
	%\WM{The Floquet multiplier estimate is correlated with the next Proposition 5.11, correct? We could/should combine these statements.}
	%
	Theorem \ref{thm:main} states existence \rev{and uniqueness} for the relaxation cycles observed in Figure \ref{fig:non_dim}. Note that the \revsj{$O(\delta)$ time-scale separation in regime (R2) (as opposed to the $O(\epsilon)$ time-scale separation in regime (R1))} %additional time-scale {\color{red}{which one? There were four timescales in the analysis.}}
	leads to a relaxation cycle $\Gamma_\epsilon$ which is $O(\epsilon^{1/3})$ from $\Gamma$, in contrast to the usual $O(\epsilon^{2/3})$ separation associated with two time-scale relaxation oscillations; see, e.g., \cite{Krupa2001b,Kuehn2015}. This is a consequence of the $O(\delta^{2/3}) = O (\epsilon^{1/3})$ separation near the fold in the regime (R2), as described in Lemma \ref{lem:slow_manifolds_R2}. The bound $-\kappa/\epsilon^2$ on the Floquet exponent also differs from the usual $-\kappa/\epsilon$ bound. This is a consequence of the total time spent in the vicinity of the attracting slow manifold $\mathcal S_{a,\delta}$, which is $O(1)$ on the \wm{infra-slow time-scale $\tau_1 = \epsilon^2 t$}. %\rev{Theorem \ref{thm:main} and the preceding GSPT analysis also reveals the role of particular flux terms in .} {\color{red}{This looks like an unfinished thought. I'm not sure this is the right place for this comment though. I don't think the role of the flux terms is clear from the Theorem and I think we want to make a bigger deal of this in a place where more applied readers are likely to find it. For instance, make a comment at the beginning of Section 5, as I have indicated above, then pick it up again in the Remark below and the Conclusions section.}}
	
	The proof of Theorem \ref{thm:main} utilises the so-called \textit{blow-up method} \cite{Dumortier1996} in the formulation of \cite{Krupa2001a,Krupa2001b}. \rev{Due to the length of the proof and the conceptual similarities with the analysis undertaken for the autocatalator problem in \cite{Gucwa2009}, this is deferred to Appendix \ref{app:proof_theorem}}. The main task is to resolve the degeneracy associated with the non-hyperbolic line $S_h$ and, in particular, the point $Q_o = (0,0)$. This can be achieved in a two-step process, by means of a cylindrical blow-up along $S_h$, and a second (successive) spherical blow-up necessary to resolve a persistent degeneracy stemming from the point $Q_o$. We refer the interested reader to Figures \ref{fig:cyl_bu} and \ref{fig:sphere_bu} in particular, which illustrate the main dynamical features after blow-up and further illustrate the similarity with the autocatalator model presented in \cite{Gucwa2009}.
	
	\begin{remark}
		\label{rem:sing_analysis}
		\rev{\revsj{As noted in the introductory discussion of this section, o}ne can also prove the existence of a limit cycle in system \eqref{eq:general_form} by an application of the Poincare-Bendixson theorem, i.e.,~without recourse to blow-up techniques. However, very little dynamical insight relating to %stability, uniqueness and 
		the geometric structure of the oscillations is obtained via such an argument. Importantly, the role of different flux terms in producing different phases of the oscillations is revealed via GSPT analysis, insofar as each segment of the singular relaxation cycle in Figure \ref{SingLimFig} (or, more precisely, Figure \ref{fig:cyl_bu}) perturbs to a phase of the oscillation that is dominated by one or more %(potentially competing) 
		flux terms. Specifically, the fast transition bounded away from the $h$ and $c$ axes is governed by the layer problem \eqref{eq:layer_R1}. Here, the flux terms $h / \tau_h(c)$ and $J_{\rm IPR}(h,c)$ are active, and dominate the dynamics. The phase of the oscillations corresponding to vertical flow down the $c-$axis is governed by the reduced problem \eqref{eq:reduced_R1} where $J_{\rm SERCA}^+(c)$ is active, and the phase of the oscillations corresponding to flow along the $h-$axis is governed by the reduced problem \eqref{eq:reduced_R2_explicit}, where all five flux terms play a role in determining the dynamics. The characteristic time-scale associated with each phase of the oscillation is also revealed via the GSPT approach; see again Remark \ref{rem:time-scales}. Finally we note that approaches based on GSPT and blow-up techniques can be lifted to higher dimensions, and therefore provide a natural geometric approach for the analysis of ODE models characterised by multiple time-scales and/or switching more generally.}
	\end{remark}
	
	Our analysis also allows for a leading order approximation of the oscillation period $\mathcal T$ as a function of the model parameter $\tau_{\rm max}$, as in the following Proposition. \rev{Our results are consistent with numerical findings in \cite{sneyd2017dynamical}, where $\tau_{\rm max}$ is shown to play an important role in determining the period of oscillations in the three-dimensional open-cell model \eqref{eq:Full_model_open}.} %, $p$ and $c_t$. 	
	\begin{proposition}\label{prop:period}
		Consider system \wm{\eqref{eq:general_form}} with $(p,c_t)$ fixed within the region specified by Assumption \ref{ass:1}. Then there exists an $\epsilon_0 > 0$ such that for all $\epsilon \in (0,\epsilon_0)$, the leading order approximation for the period $\mathcal T$ of the relaxation cycle $\Gamma_\epsilon$ is linear in $\tau_{\rm max}$. More precisely, written 
		% and % and  with $0 < \epsilon \ll 1$, and the parameter values in Table \ref{tab:params_post_scaling} except for 
		%	$\tau_{\rm max}$ allowed to vary within the regime
		%	\[
		%	\tau_{\rm max} = O(\epsilon^{-2}) .
		%	\]
		%	Then the period $T$ of the relaxation oscillation $\Gamma_\epsilon$, written
		in terms of the \wm{fast time-scale $t$ of system \eqref{eq:general_form}} we have %grows linearly with $\tau_{\rm max}$ in accordance with
		\[
		\mathcal T \sim \wm{\epsilon^{-2}} \tau_{\rm max} v(p, c_t) \qquad \text{as} \qquad \epsilon \to 0,
		\]
		where $v(p, c_t)$ is smooth, positive and bounded on the relevant domain.
	\end{proposition}
	
	\begin{proof}
		
		Assuming Theorem \ref{thm:main}, % and recalling the discussion in Section \ref{sec:time-scales}, 
		the leading order approximation for the period $\mathcal T$ as $\epsilon \to 0$ is determined by the total time spent in the vicinity of the attracting \wm{infra-slow} manifold $\mathcal S_a$. 
		This can be approximated using the expression for the corresponding reduced flow given in \eqref{eq:reduced_R2} \wm{which evolves on the infra-slow time-scale $\tau_1 = \epsilon^2 t$}. 
		{Let $C(\tau_1)$ denote a solution for \eqref{eq:reduced_R2} such that $\zeta(C(0)) = 0$ and $C(\tau_1^0) = C_F$ for some $\tau_1^0 = \epsilon^2 t_0 > 0$. }
		%{\color{red}{Should $\tilde \tau_0$ in the previous and next sentence be replaced with $\tau_1^0$?}}
		\SJ{This yields $C(0) = \sqrt{K} \gamma c_t$ and $C(\tau_1^0) = \sqrt{2 K} \gamma c_t$, and hence}
		\begin{equation}
			\label{eq:period}
			\epsilon^{2} \mathcal T \sim \int_{0}^{\wm{\tau_1^0}} \wm{d\tau_1} 
			= \int_{\sqrt{K} \gamma c_t}^{\sqrt{2 K} \gamma c_t} \frac{1}{\dot C} dC ,
		\end{equation}
		where $\epsilon^{2}$ appears in the left-hand-side since we aim to approximate the period $\mathcal T$ on the \wm{fast time-scale $t = \epsilon^{-2} \tau_1$ of system \eqref{eq:general_form}}. 
		%{\color{red}{VK: Even with the reminder that $\tilde \tau_0 = \epsilon^2 t_0 > 0$, this step is not clear. By the discussion after \eqref{eq:reduced_R2}, we have that $\dot{C}=\frac{dC}{d \tilde{\tau}}$ but the expression here seems to be in terms of the original time, $t$.}}
		
		Using \eqref{eq:reduced_R2}, the integrand can be written as $1/\dot C = \tau_{\rm max} V(C; p , c_t)$, where the function $V(C ; p , c_t)$ is smooth with respect to $C, c_t, p,$ and independent of $\tau_{\rm max}$. By Assumption \ref{ass:1} we have that $\dot C > 0$, and hence $V(C ; p , c_t) > 0$ on $\mathcal S_a$. Thus, {\color{black}{as $\epsilon \to 0$,}} %VK: moved this from below the equation, for ease of reading.
		\[
		\mathcal T \sim \wm{\epsilon^{-2}} \tau_{\rm max} \int_{\sqrt{K} \gamma c_t}^{\sqrt{2 K} \gamma c_t} V(C ; p ; c_t) dC = \wm{\epsilon^{-2}} \tau_{\rm max} v(p, c_t),
		\]
		 where $v(p,c_t)$ is some smooth, positive and bounded function.
		%\textcolor{red}{\textbf{[The expression for $v(p,c_t)$ is disgusting, but may be digestible for an asymptotic expression $p \sim 0$. The linear relationship with $\tau_{\rm max}$, as well the $p$ and $c_t$ dependence, is discussed in the \cite{sneyd2017dynamical}: see Figure S3 in the supplementary material. Note also that the higher order correction may need to take the reduced flow along $\mathcal S_c$ into account]}}.
	\end{proof}
	
	%\wm{
	\begin{remark}\label{rem:num-period}
		Figure~\ref{fig:non_dim} shows relaxation oscillations with an approximate period $\mathcal T \approx 2\times 10^{4}$, which compares well to the order estimate $\epsilon^{-2} \tau_{\rm max} \approx 5\times 10^{4}$ derived above and evaluated using the parameter values given in Table~\ref{tab:params_post_scaling}.
	\end{remark}
	%}
	
	\rev{Finally, we refer the interested reader to Appendix \ref{sec:bifurcations} for additional results pertaining to the onset of the relaxation oscillations under parameter variation. Here we provide numerical evidence for the explosive onset of relaxation oscillations under variation of the total calcium $c_t$. Such findings are particularly relevant for future analyses of the three-dimensional open cell model \eqref{eq:Full_model_open_ct}, where $c_t$ is a slow variable. We also provide an analytical result describing the gradual (i.e.,~non-explosive) onset of oscillations under variation of the parameter $\tau_{\rm max}$, confirming previous observations in \cite{sneyd2017dynamical} on the role of $\tau_{\rm max}$ as an important parameter for controlling the existence and period of oscillations in models for intracellular calcium dynamics.}

	\section{Discussion and Conclusion}
	\label{sec:discussion_conclusion}
	
	%\wm{Non-standard} 
	Time-scale separation \rev{and switching are} ubiquitous in {\color{black}{models of}} biological and physiological phenomena, but literature on the %mathematical modelling 
	{\color{black}{analysis}} of such systems, {\color{black}{particularly via methods of geometric singular perturbation theory,}} is relatively sparse. 
	\rev{There are a number of reasons for this. The first significant obstacle to analysis of such models concerns the identification of suitable perturbation parameters, since these are frequently not explicit in a model. In the case that one or more perturbation parameters can be identified, a second obstacle arises if there is more than one, namely, the question of how to relate or order perturbation parameters in such a way that the resulting singular perturbation problem is both tractable and reliable as an approximation of the original model. Finally, there are obstacles relating to the mathematical analysis of the resulting singular perturbation problem; singular perturbation problems derived from ODE models for chemical, biological and physiological phenomena are frequently characterised by non-standard time-scale separations, more than two time-scales, and loss of smoothness in the singular limit due to switching, all of which can complicate the analysis.} %This is perhaps due to the significant mathematical complications inherent in such systems, and the perceived absence of a well adapted mathematical framework for their analysis. At present, however, 
	%there is a growing awareness that classical methods based on GSPT can be successfully adapted and applied in such contexts. This has been demonstrated in the context of chemical and biological models in the work of Kosiuk and Szmolyan \cite{Gucwa2009,kosiuk2011scaling,kosiuk2016geometric}. Formal developments in subsequent work by Wechselberger \cite{wechselberger2018geometric}, Goeke \& Walcher \cite{goeke2014constructive}, Lizarraga, Marangell \& Wechselberger \cite{Lizarraga2020,Lizarraga2020c} and Kruff \& Walcher \cite{Kruff2019} have aided in providing a systematic framework for \wm{a coordinate-independent} analysis of such problems via GSPT.
	
	In this work, \rev{we have outlined a heuristic procedure for the formulation as singular perturbation problems of ODE models characterised by multi-scale dynamics and/or switching. The procedure was sketched in generality in Section \ref{sec:the_procedure}, and consists of the following steps:
	(I) non-dimensionalise, \we{(IIa) associate small parameters to maximal process/flux ratess, (IIb) associate small parameters to steep switches, (III)} relate small parameters, and (IV) analyse the system via perturbation theory.}
	
	The procedure was applied in detail {in Sections \ref{sec:steps_I_IV}-\ref{sec:slow-fast_analysis_and_statement_of_the_main_result} to} a closed-cell model for intracellular calcium dynamics. % obtained via a reduction from a canonical open-cell model developed in \cite{sneyd2017dynamical}. %By adopting a process-oriented philosophy and methodology, 
	\rev{Specifically, we identified} a total of five small parameters. \rev{Three of these small parameters, $\epsilon_1, \epsilon_2$ and $\epsilon_3$, were identified in step \we{(IIa)}, and correspond to constant pre-factors of particular flux terms} %The identified small parameters were either (i) 
	%uniform perturbations setting the scale of a particular process 
	(e.g.~$V_s$)\rev{. The remaining two small parameters, $\epsilon_4$ and $\epsilon_5$, were identified in step \we{(IIb)}, and derive from the presence of switching behaviour due to the presence of Hill-type functions with a \we{steep} gradient.} % or (ii) non-uniform perturbations deriving from the switching behaviour associated with particular biological processes. % (the half-values $K_i$). 
	\rev{%For the sake of tractability in a singular perturbation analysis, a
	All five small parameters were related to a single small parameter $\epsilon \ll 1$ via the polynomial scaling \eqref{eq:common_scaling} in step \we{(III)}, the form of which was based on order of magnitude comparisons. This yielded the singular perturbation problem \eqref{eq:main2} (or equivalently \eqref{eq:main1}) %in Proposition \ref{prop:sing_pert_form}, 
	which featured a non-standard time-scale separation}; {\color{black}{the non-standard time-scale separation stemmed from the presence of small parameters associated with}} \rev{both steps (IIa)-(IIb).} %switching behaviour,}}  %which %leads} Small parameters of the latter kind % in particular 
	%are responsible for \wm{non-standard} time-scale separations 
%	since the corresponding singular limit is non-uniform in state space.\we{[comment: I believe it is non-standard not only because of switching terms, but simply because it is non-standard (in the smooth sense as well), i.e. it has both features described, non-standard due to (IIa) and (IIb); please correct me if I am wrong]}
	
	\rev{In Section \ref{sec:slow-fast_analysis_and_statement_of_the_main_result} we analysed system \eqref{eq:main1} using the coordinate-independent formulation of GSPT developed in \cite{wechselberger2018geometric}. }{ \color{black}{The non-standard time-scale separation in the model}} led to distinct scaling regimes (R1) and (R2), with distinct time-scale separations. In the context of biological switching in general, an increasing number of studies show that distinct scaling regimes of this kind are to be expected \cite{Gucwa2009,kosiuk2016geometric,Kristiansen2019d}. These references, along with the current manuscript, also demonstrate the suitability of blow-up methods for combining dynamical information obtained in distinct scaling regimes. %Following the identification of small parameters, a common scaling \wm{with a single small parameter $\epsilon \ll 1$} was proposed in \eqref{eq:common_scaling}. This led to the singular perturbation problem \eqref{eq:main2}, which is given in a general form suited to the GSPT framework in \cite{wechselberger2018geometric}. 
	%and with a single small parameter $\epsilon \ll 1$. 
	We proved existence\rev{, uniqueness and stability for} %and uniqueness 
	the observed three-time-scale relaxation cycles\rev{, as described} in Theorem \ref{thm:main}\rev{, and provided an estimate for the period of the oscillations in Proposition \ref{prop:period}.  Finally, further results pertaining to the onset of the relaxation oscillations under variation of parameters of interest ($p, c_t$ and $\tau_{\rm max}$) are described in Appendix \ref{sec:bifurcations}.} 
	
	{\color{black}{We emphasise that there are other, seemingly simpler, methods for proving the existence of relaxation oscillations in the specific calcium dynamics model investigated in this work. In particular, using standard phase plane arguments in combination with the Poincar\'e-Bendixson theorem might be a natural way to proceed for this two-dimensional model. However, }}
	\revsj{%the specified parameter regime, 
	our methods provided additional information on the geometric and time-scale structure of the oscillations which cannot be obtained directly via these other methods.} 
	{\color{black}{Furthermore, our methods are transferable to higher-dimensional problems (although we acknowledge that their implementation in higher-dimensional models will have challenges.)}}
	%, in a setting which may readily be transferred to similar and higher dimensional problems.}
	
%	\wm{We then identified a singular Andronov-Hopf bifurcation together with a complete/incomplete canard explosion as an onset mechanism for oscillations under the variation of the total calcium concentration or IP$_3$ concentration.} We also described a relationship between the system parameter $\tilde \tau_{\rm max}$, which sets the speed of evolution of $h$, and the period of oscillations. In particular, Proposition \ref{prop:period} describes a linear relationship for values $\tilde \tau_{\rm max} = O(\epsilon^{-2})$, and Theorem \ref{thm:hopf} shows that a supercritical Andronov-Hopf provides \SJ{one possible} mechanism for the onset of oscillations with $\tilde \tau_{\rm max} = O(\epsilon^{-3/2})$. 
	
	%\wm{We also identified a singular Andronov-Hopf bifurcation together with a complete/incomplete canard explosion as an alternative onset mechanism for oscillations under the variation of the total calcium concentration or IP$_3$ concentration.}
	%\WM{check/compare/update after section 6 is done.}
	
%	{\color{red}{}I think this would be a good place to return to the idea that it is possible to get some of the information presented here for the 2D calcium model (e.g., existence of relaxation oscillations) via other more straightforward methods, but there are advantages in understanding to doing it our way - and this paper was intended to demonstrate the method rather than prove something momentus about the 2D calcium model. This could be a brief comment.}
	
	There are many directions in which one can take this work. For example, {it could be interesting to give a detailed description of the canard and incomplete canard explosions described in Section \ref{sub:can_explosion}.}
	%one is frequently interested in qualitative changes to dynamics under variation of either the IP$_3$ concentration $p$, the total calcium concentration $c_t$, or both. Hence, it is natural to consider a \wm{two}-parameter bifurcation analysis in $(p,c_t)$. \SJ{Such an analysis would involve a detailed description of the canard and incomplete canard explosions described in Section \ref{sub:can_explosion}.} %Preliminary bifurcation analysis suggests that there are multiple possible mechanisms for the onset of relaxation oscillations in system \eqref{eq:main2}. One possibility is a canard explosion which is similar to that described in \cite{gucwa2009geometric}. This is expected in parameter regimes prescribed by Assumption \ref{ass:1}, in which the system has a unique equilibrium. In parameter regimes corresponding to up to three equilibria, incomplete canard explosions with similarities to the system described in \cite{de2015neural} can be expected. A third possibility is suggested by the identification of the Andronov-Hopf bifurcation under $\tilde \tau_{max}$ variation in Theorem \ref{thm:hopf}, and does not involve canard phenomena. 
	%\WM{cases one and three are not the same!? Emphasize!?}
	\SJ{The} current manuscript also paves the way for rigorous analytical treatment of the higher-dimensional open-cell model from which it is derived. In fact, one could view system \eqref{eq:Full_model} as a layer problem for the open-cell but `almost-closed' model \eqref{eq:Full_model_open_ct} with $|J_{pm}| \ll 1$. In this context, a \wm{one}-parameter bifurcation analysis of system \eqref{eq:main2} in $c_t$ amounts to an analysis of the layer problem to system \eqref{eq:Full_model_open_ct}; a three-dimensional system with (at least) four time-scales, due to an additional slow evolution in $c_t$. The case $|J_{pm}| = O(1)$ corresponding to `more open' cells is also of physiological significance. In both cases, this work provides a solid foundation upon which these significant and highly non-trivial analyses become feasible.

	\rev{
	\section*{Acknowledgements}
	
	The authors would like to thank the associate editor, editor and reviewers of the first submitted version of the manuscript, whose feedback and observations led to significant improvements in the current version.

%	S.~Jelbart also acknowledges partial funding from the SFB/TRR 109 Discretization and Geometry in Dynamics grant, and together with M.~Wechselberger, partial funding	from the ARC Discovery Project grant DP180103022. {\color{red}{VK: Seems like we should either acknowledge funding here or in the footnote on the first page but not in both places. I don't care where it happens}}
	}

	% section discussion_conclusion (end)

	\bibliographystyle{siamplain}
	\bibliography{./biblio_new}

	% section references (end)

	\appendix

	%\section{Equilibria in the system \eqref{eq:reduced_R2}}
	%\label{app:equilibria}
	
	%For each positive, real root of the equation \eqref{eqs:reduced_equilibria_R2b}, there exists an equilibrium in the reduced problem \eqref{eq:reduced_R2}. The coefficients $\alpha_i$, $i=0,1,2,3$ are given by
	%\begin{equation}
	%\begin{aligned}
	%\alpha_3 &= \frac{A_{\text{SERCA}}}{\gamma c_t A_{\text{\text{IPR}}}}, && \alpha_2 = - K_h^4 - \frac{A_{\text{SERCA}}}{ A_{\text{\text{IPR}}}} K \gamma c_t , \\
	%\alpha_1 &= \frac{K_h^4 A_{\text{SERCA}}}{\gamma c_t A_{\text{\text{IPR}}}} , && \alpha_0 = - \frac{K K_h^4 \gamma c_t A_{\text{SERCA}}}{A_{\text{\text{IPR}}}} .
	%\end{aligned}
	%\end{equation}

	\section{Proof of Theorem \ref{thm:main}}
	\label{app:proof_theorem}
	
	In this Appendix we prove Theorem \ref{thm:main}. The main technique is the blow-up method developed in \cite{Dumortier1996} in the formalism of \cite{Krupa2001a,Krupa2001b}. In Section \ref{sec:blow_up} we introduce a cylindrical blow-up of the degenerate line $S_h$, allowing us to resolve much of the degeneracy in \eqref{eq:general_form}, and extend the singular cycle $\Gamma$ onto the blow-up cylinder.
	%In Section \ref{sec:spherical_blowup}, a successive (spherical) blow-up is introduced in order to resolve remaining non-hyperbolicity asociatated with a nilpotent point identified following the cylinderical blow-up.
	This allows for the presentation of a singular cycle \revsj{with improved hyperbolicity properties}, about which a Poincaré map is defined and analysed in Section \ref{sec:poincare_map}. 
%	{\color{red}{Not sure what ``improved" means here; perhaps use a different word or explain a bit? I think using improved is fine later in the section, since it is explained there.}}
	We present a sequence of lemmas which will allow for a proof of Theorem \ref{thm:main}, which is given in Section \ref{ssec:proof_of_theorem}. As part of this proof we rely on additional results obtained after the application of a successive (spherical) blow-up, which must be introduced in order to resolve additional degeneracy associated with a degenerate point that is identified following the cylindrical blow-up. This (successive) blow-up analysis is deferred to Section \ref{sec:spherical_blowup} for expository reasons.
	
	\begin{remark}
		The geometry and analysis of the relaxation cycles is similar in many respects to the autocatalytic relaxation oscillations studied in \rev{\cite{milik2001} and subsequently} in \cite{Gucwa2009}. %We refer to this reference throughout in situations where technical details are `similar enough' to be omitted for the sake of brevity.
	\end{remark}

	\subsection{Blow-up of the non-hyperbolic line $S_h$} 
	\label{sec:blow_up}

	We consider the extended system
	\begin{equation}\label{eq:main_bu}
		\begin{split}
			\begin{pmatrix}
				h' \\
				c'
			\end{pmatrix}
			&= g(h,c)+\epsilon G(h,c,\epsilon),  \\
			\epsilon' \ \  &= 0 
		\end{split}
	\end{equation}
	obtained by adding the trivial equation $\epsilon'=0$ to system \eqref{eq:general_form}, and define a weighted cylindrical blow-up by the transformation
	\begin{equation}
		\label{eq:cylindrical_bu}
		r \geq 0, \ \left(\bar c, \bar \epsilon \right) \in S^1 \mapsto
		\begin{cases}
			c = r \bar c , \\
			\epsilon = r^2 \bar \epsilon ,
		\end{cases}
	\end{equation}
	which `blows up' $S_h$ to the cylinder $\{r = 0\} \times S^1 \times \mathbb R_+$. We \SJ{work in two coordinate charts, defined via $K_1: \bar c=1$ and $K_2: \bar\epsilon=1$,} for which we \SJ{introduce} chart-specific coordinates
	\begin{equation}\label{eq:chart_coordinates}
		K_1: \ c = r_1,  \epsilon = r_1^2 \epsilon_1,  \qquad K_2: \ c = r_2 c_2,  \epsilon= r_2^2.
	\end{equation}
	\SJ{Chart} $K_1$ is referred to as an entry/exit chart, \SJ{and} chart $K_2$ is referred to as the `family rescaling' chart. The transition maps between $K_1$ and $K_2$ are given by
	\begin{equation}
		\begin{aligned}
			\kappa_{12}: \ r_1 &= r_2c_2, && \epsilon_1=c_2^{-2}, && c_1 > 0 , \\
			\kappa_{21}: \ c_2 &= \epsilon_1^{-1/2}, && r_2= r_1 \epsilon_1^{1/2}, && \epsilon_1 > 0 .
		\end{aligned}
	\end{equation}
	%In what follows, for objects with direct analogues in the analysis of the autocatalator in \cite{gucwa2009geometric}, we adopt an equivalent notation \textbf{[necessary?]}.
	We will adopt some common notational conventions throughout the analysis. In particular, given a set $\mathcal A$, 
	% expressed in terms of coordinates $(h,c,\epsilon)$
	we denote its image %under the blow-up transformation \eqref{eq:cylindrical_bu} by $\bar \gamma$, and its image 
	in coordinate chart $K_i$ by $\mathcal A_i$.

	\subsubsection{Dynamics in charts}
	\label{sec:dynamics_in_charts}
	
	In the family rescaling chart $K_2$ we have \SJ{$r_2 = \sqrt{\epsilon}=\delta$}. Following the time desingularisation $dt = \delta^{-3}dt_2$, which amounts to division of the vector field by $\delta^3$, one obtains system \eqref{eq:main_R2} in chart $K_2$, except with the notation $C = c_2$, $\delta = r_2$. %i.e.~the dynamics in chart $K_2$ is obtained in the analysis of the dynamics in regime (R2), except with the notation $C = c_2$, $\delta = r_2$. 
	Since we have already studied this system in Section \ref{ssec:singular_limit_analysis_R2}, we need not consider it further here.
	
	\
	
	Now consider the dynamics in chart $K_1$. Differentiating the relevant expressions in \eqref{eq:chart_coordinates} and applying a time desingularisation $dt = r_1^{-3}dt_1$ which amounts to division by $r_1^3$, we obtain
	\begin{equation}\label{eq:K1_equations}
		\begin{split}
			h ' &= r_1 \tau_{\rm max}^{-1} \left(\hat h_\infty(\epsilon_1) - h \right) \left(1 + \epsilon_1^2 \right) ,   \\
			r_1' &= r_1 \left[\hat J_{\text{IPR}}(h, r_1, \epsilon_1) - \epsilon_1 \hat{\mathfrak J}^+_{\text{SERCA}}(r_1) + \epsilon_1^2 \hat{\mathfrak J}^-_{\text{SERCA}}(r_1) \right], \\
			\epsilon_1' &= - 2 \epsilon_1 \left[\hat J_{\text{IPR}}(h, r_1, \epsilon_1) - \epsilon_1 \hat{\mathfrak J}^+_{\text{SERCA}}(r_1) + \epsilon_1^2 \hat{\mathfrak J}^-_{\text{SERCA}}(r_1) \right],
		\end{split}
	\end{equation}
	where
	\begin{equation}
		\begin{split}
			\hat h_\infty(\epsilon_1) &:= \mathfrak h_\infty(r_1, r_1^2 \epsilon_1), \qquad 
			\hat J_{\text{IPR}}(h, r_1, \epsilon_1) := \frac{1}{r_1^4} \jipr(h, r_1, r_1^2 \epsilon_1), \\
			\hat{\mathfrak J}^+_{\text{SERCA}}(r_1) &:= \frac{1}{r_1^2} \mathfrak J^+_{\text{SERCA}}(r_1), \qquad  \hat{\mathfrak J}^-_{\text{SERCA}}(r_1) := \mathfrak J^-_{\text{SERCA}}(r_1) ,
		\end{split}
	\end{equation}
	all extend smoothly to $r_1=0$ due to common factors of $r_1$ in their respective numerators. \SJ{Note that the} $(')$ notation \SJ{now denotes} differentiation with respect to $t_1$.
	
	%\begin{proof}
	%	Equations for $h'$ and $r_1'$ are obtained simply by rewriting the right hand sides of the equations for $h'$ and $c'$ respectively in the $K_1$ coordinates given in \eqref{eq:chart_coordinates}, followed by a time desingularisation $dt = r_1^{-3}dt_1$. To obtain the equation for $\epsilon_1'$, notice that
	%	\[
	%	\epsilon_1' = \left(\frac{\epsilon}{r_1^2}\right) \epsilon' - \left(\frac{2 \epsilon_1}{r_1}\right) r_1' = - \left(\frac{2 \epsilon_1}{r_1}\right) r_1' .
	%	\] 
	%\end{proof}
	We identify two lines of equilibria
	\begin{equation}
		\label{eq:K1_eq_lines}
		l_{h,1} = \{(h, 0, 0) : h \geq 0 \} , \qquad l_{c,1} = \{(0, r_1, 0) : r_1 \geq 0 \} ,
	\end{equation}
	as well as a smooth curve of equilibria
	\begin{equation}
		\mathcal S_1 = \left\{(\Psi(\epsilon_1), 0, \epsilon_1) : \epsilon_1 \geq 0 \right\} , \qquad 
		\Psi(\epsilon_1) = \frac{A_{\text{SERCA}}}{A_{\text{\text{IPR}}} \gamma c_t} \epsilon_1 \left(1 - K \gamma^2 c_t^2 \epsilon_1\right) .
	\end{equation}
	%where
	%\[
	%\Psi(\epsilon_1) = \frac{A_{\text{SERCA}}}{A_{\text{\text{IPR}}} \gamma c_t} \epsilon_1 \left(1 - K \gamma^2 c_t^2 \epsilon_1\right) .
	%\]
	The line $l_{c,1}$ corresponds to the image of the attracting critical manifold $S_c$ within chart $K_1$. 
	For $\epsilon_1>0$, the curve $\mathcal S_1$ coincides with the critical manifold identified in \eqref{eq:S_R2}, which we denote here by $\mathcal S_2$. The origin, which we denote by $Q_3 = (0,0,0)$, lies at the intersection of all three curves, % \SJ{$Q_3 = l_{h,1} \cap \mathcal S_1 \cap l_{c,1}$}, 
	and the regular fold point $F$ in equation \eqref{eq:Fold} is identified as a fold point $Q_5$ with chart $K_1$ coordinates
	\[
	Q_5 = \left(h_f, r_{f,1}, \epsilon_{f,1} \right) = \left(\frac{A_{\text{SERCA}}}{4 K A_{\text{\text{IPR}}} \gamma^3 c_t^3} , 0, \frac{1}{2 K \gamma ^2 c_t^2} \right).
	\]
	The point $Q_5$ divides $\mathcal S_1$ into normally hyperbolic branches as $\mathcal S_1 = \mathcal S_{a,1} \cup \{Q_5\} \cup \mathcal S_{r,1}$,
	where
	\[
	\mathcal S_{a,1} = \left\{(\Psi(\epsilon_1), 0, \epsilon_1) : \epsilon_1 > \epsilon_{f,1}) \right\} , \qquad
	\mathcal S_{r,1} = \left\{(\Psi(\epsilon_1), 0, \epsilon_1) : \epsilon_1 < \epsilon_{f,1}) \right\} ;
	\]
	see Figure \ref{fig:K1}. The following result describes stability properties of $l_{h,1}, l_{c,1},\mathcal S_{a,1}, \mathcal S_{r,1}$ and $Q_3$.
	
	\begin{center}
		\captionsetup{format=plain}
		\begin{figure}[t!]
			\hspace{0em}\centerline{\includegraphics[scale=0.65]{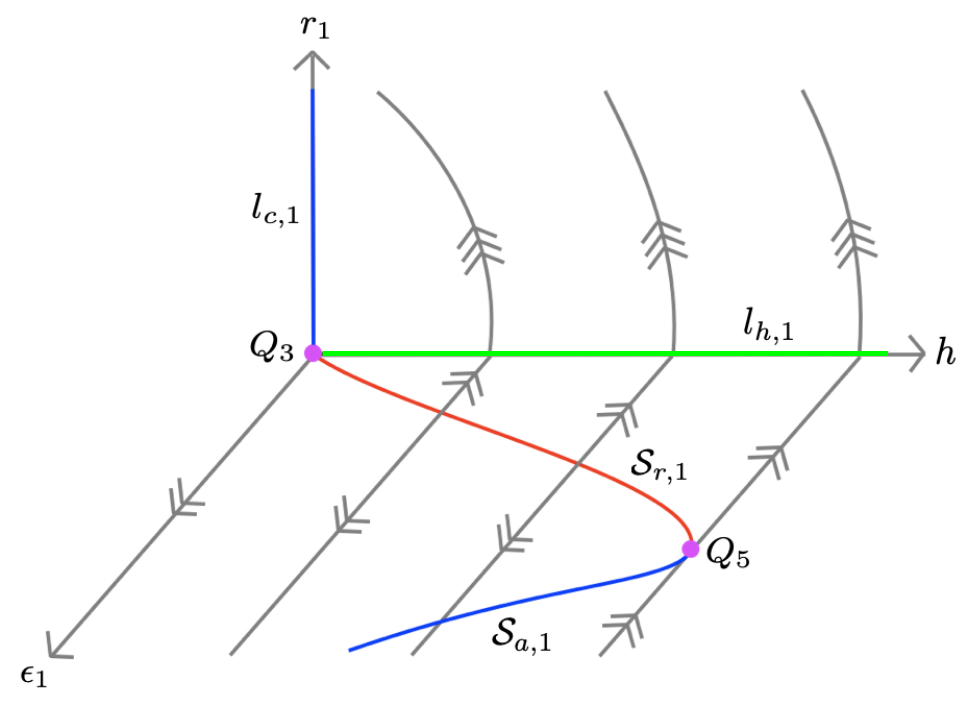}}
			\caption{Main dynamical features identified in the analysis in chart $K_1$. Normally hyperbolic and attracting (repelling) manifolds of equilibria $l_{c,1}$ and $\mathcal S_{a,1}$ ($\mathcal S_{r,1}$) are shown in blue (red), and the line of saddle-type steady states $l_{h,1}$ is shown in green. The point $Q_5$ is the image of the regular jump point $F$ in chart $K_1$ coordinates. The equilibrium $Q_3$ is nonhyperbolic and will require further blow-up; see Section \ref{sec:spherical_blowup}.}\label{fig:K1}
		\end{figure}
	\end{center}
	
	\begin{lemma}\label{lem:K1_dynamics} The following hold for system \eqref{eq:main_bu}:
		\begin{enumerate}
			\item[(i)] The line $l_{h,1} \setminus Q_3$ is partially hyperbolic and attracting, with eigenvalues \SJ{$\lambda = - r_1 \tau_{\rm max}^{-1} , \ 0, \ 0$.}
			%		Stable and center eigenspaces respectively are given by
			%		\[
			%		E^s(l_{h,1} \setminus Q_3) = span\left\{\left(- \frac{1}{K_\tau^4 \tau_{max}} , D_{h} \hat J_{\text{IPR}}(0,r_1,0), 0\right)^T\right\},
			%		\]
			%		and
			%		\[
			%		E^c(l_{h,1} \setminus Q_3)  = span\left\{(0,1,0)^T,(0,0,1)^T\right\}.
			%		\]
			\item[(ii)] The line $l_{h,1} \setminus Q_3$ is normally hyperbolic and saddle-type, with eigenvalues $\lambda = - 2 A_{\text{\text{IPR}}} \gamma c_t h$, $A_{\text{\text{IPR}}} \gamma c_t h$,  $0$. For each $b = \{(0, \alpha, 0)\} \in l_{h,1} \setminus Q_3$, the corresponding one-dimensional stable manifold $W^s(b)$ is contained within the line $h = b, r_1 = 0, \epsilon_1 \geq 0$.
			\item[(iii)] The critical manifolds $\mathcal S_{a,1}$, $\mathcal S_{r,1}$ are normally hyperbolic and attracting, respectively repelling.
			\item[(iv)] The point $Q_3$ is a nilpotent singularity, with eigenvalues $\lambda = 0,0,0$.
		\end{enumerate}
	\end{lemma}
	
	\begin{proof}
		Statements (i), (iii), (iv) and the first assertion in (ii) can be shown by direct calculations following linearisation of the system \eqref{eq:K1_equations}. To see that the stable manifold $W^s(b)$ for a given base point $b = \{(0, \alpha, 0)\} \in l_{h,1} \setminus Q_3$ lies within the line $h = \alpha, r_1 = 0, \epsilon_1 \geq 0$, consider the system obtained from \eqref{eq:K1_equations} by restricting to the invariant plane $\{r_1 = 0\}$ (i.e.,~to the cylinder):
		\begin{equation}\label{eq:K1_inv2}
			\begin{split}
				h' &= 0, \\
				\epsilon_1' &= - 2 \epsilon_1 A_{\text{\text{IPR}}} \gamma c_t \left(h - \frac{A_{\text{SERCA}}}{A_{\text{\text{IPR}}} \gamma c_t} \epsilon_1 \left(1 - K \gamma^2 c_t^2 \epsilon_1\right)\right).
			\end{split}
		\end{equation}
		The result follows from the observation that each line $h = \alpha, \epsilon_1 \geq 0$ is invariant under the flow induced by \eqref{eq:K1_inv2}, with dynamics governed by
		\begin{equation}
			\epsilon_1' = - 2 \epsilon_1 A_{\text{\text{IPR}}} \gamma c_t \alpha + O(\epsilon_1^2) .
		\end{equation}
	\end{proof}
	The results in Lemma \ref{lem:K1_dynamics} lead to the following result.
	
	\begin{lemma}\label{lem:K1_extended_manifold} Fix $\nu_+ > \nu_- > 0$ and $\eta > 0$. For $\eta$ sufficiently small, system \eqref{eq:main_bu} has a two-dimensional center manifold
		\[
		M_1 = \left\{(\hat h_\infty(\epsilon_1), r_1, \epsilon_1) : r_1 \in [\nu_-, \nu_+], \epsilon_1 \in [0,\eta]  \right\} ,
		\]
		containing the line of equilibria $l_{c,1}$ as the restriction $M_1|_{\epsilon_1 = 0}$. Moreover, there exists a stable foliation $\mathcal F_1$ with base $M_1$ and one-dimensional fibers. The contraction along $\mathcal F_1$ in a time interval of length $t$ is stronger than $e^{ct_1}$ for any $0 < c < \nu_- / \tau_{\rm max}$, and the flow on $M_1$ is strictly decreasing in the variable $r_1$.
	\end{lemma}
	
	\begin{proof}
		Existence of a two-dimensional strongly attracting center manifold $M_1$ containing $l_{c,1}$ follows from Lemma \ref{lem:K1_dynamics} and center manifold theory, and the existence of a stable foliation with exponential contraction follows from Fenichel theory.
		
		To show that the flow on $M_1$ is strictly decreasing in the variable $r_1$, we restrict equations \eqref{eq:K1_equations} to $M_1$, obtaining the system
		\begin{equation}
			\begin{split}
				r_1' &= - \left(\frac{V_s}{K_s^2 + r_1^2}\right) r_1 \epsilon_1 + O(r_1 \epsilon_1^2) , \\
				\epsilon_1' &= \left(\frac{2 V_s}{K_s^2 + r_1^2}\right) \epsilon_1^2 + O(\epsilon_1^3) .
			\end{split}
		\end{equation}
		For $\eta > 0$ sufficiently small, $r_1' < 0$ as claimed.
	\end{proof}

	% section dynamics_in_charts (end)
	
	\subsubsection{An improved singular cycle $\Gamma$}
	\label{sec:dynamics_in_the_blowup_combined}

	Our analysis thus far leads to the dynamics sketched in Figure \ref{fig:cyl_bu}. We have two lines of steady states  ($l_h$ and $l_c$), a folded critical manifold $\mathcal S$ on the cylinder, and distinguished points $Q_i$, $i = 1, \ldots , 5$ given in chart $K_1$ coordinates by $Q_1 =(h_f, 0, 0), \ Q_2 = (0, r_{d,1} , \ 0 ), \  Q_3 = (0, 0, 0 ), \ Q_4 = (0, 0, \epsilon_{l,1} )$ and $Q_5 = (h_f, 0, \epsilon_{f,1} )$. 
	%\begin{equation}
	%\begin{split}
	%Q_1 = & \left(h_f, 0, 0 \right), \qquad Q_2 = \left(0, r_{d,1} , 0 \right), \qquad Q_3 = \left(0, 0, 0 \right), \\
	%& Q_4 = \left(0, 0, \epsilon_{l,1} \right), \qquad Q_5 = \left(h_f, 0, \epsilon_{f,1} \right) .
	%\end{split}
	%\end{equation}
	The location of point $Q_4$ is given explicitly by $\epsilon_{l,1}= 1/K \gamma^2 c_t^2$. We do not give an explicit location for point $Q_2$, noting simply that $r_{d,1}>0$ by Lemma \eqref{lem:map12} below.
	%Each point $Q_i$ corresponds to a jump, landing, or transition onto/off of the cylinder.
	Taking the points $Q_i$ as concatenation points allows for the construction of a singular relaxation cycle
	\[
	\Gamma = \Gamma_1 \cup \Gamma_2 \cup \Gamma_3 \cup \Gamma_4 \cup \Gamma_5,
	\]
	where, again in $K_1$ coordinates, we define
	\begin{equation}
		\begin{split}
			\Gamma_1 &= \Gamma_l , \\
			\Gamma_2 &= \left\{(0, r_1, 0) : r_1 \in (0, r_{d,1}) \right\}, \\
			\Gamma_3&= \left\{(0, 0, \epsilon_1) : \epsilon_1 \in (0, \epsilon_{l,1}) \right\}, \\
			\Gamma_4 &= \left\{(\Psi(\epsilon_1), 0, \epsilon_1) : \epsilon_1 \in (\epsilon_{f,1}, \epsilon_{l,1}) \right\}, \\
			\Gamma_5 &= \left\{(h_f, 0, \epsilon_1) : \epsilon_1 \in (0, \epsilon_{f,1}) \right\},
		\end{split}
	\end{equation}
	%and the orbit segment $\Gamma_1$ can be defined here by $\Gamma_1 = \Gamma_l$ (
	see Figure \ref{fig:cyl_bu}. In order to prove Theorem \eqref{thm:main}, we need to show that $\Gamma$ persists for $\epsilon \ll 1$ sufficiently small.
	
	\begin{center}
		\captionsetup{format=plain}
		\begin{figure}[t!]
			\hspace{0em}\centerline{\includegraphics[scale=0.6]{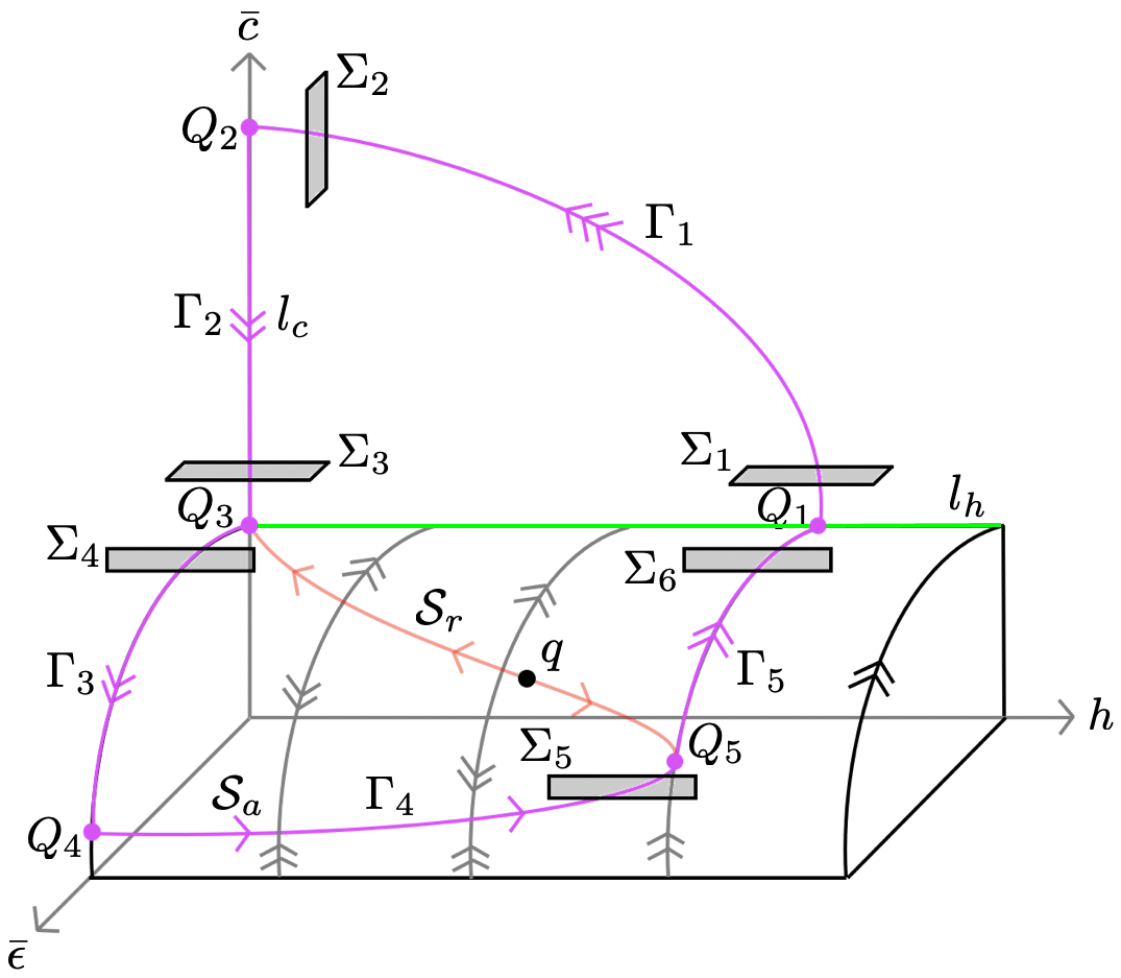}}
			\caption{Main dynamical features after cylindrical blow-up of the degenerate line $S_h$. The dynamics observed in the analysis in scaling regime (R1) correspond to the dynamics bounded `above' the cylinder, while the dynamics in scaling regime (R2) appear on the cylinder itself. The singular relaxation cycle $\Gamma = \Gamma_1 \cup \Gamma_2 \cup \Gamma_3 \cup \Gamma_4 \cup \Gamma_5$ is shown in magenta, and the segments $\Sigma_i$, $i = 1 , \ldots 6$ used to define the Poincaré map in Section \ref{sec:poincare_map} are also shown. Although both are indicated with double arrows, we note that the speed along $\Gamma_2$ is not the same as the speed along $\Gamma_5$ (see Remark \ref{rem:time-scales}).}\label{fig:cyl_bu}
		\end{figure}
	\end{center}
	
	%Our aim is to prove the following result for the blown-up system, which immediately implies Theorem \ref{thm:main} upon application of the blow-down map.
	
	%\begin{theorem}\label{thm:main_bu} Given Assumptions 1, there exsits $\epsilon_0 > 0$ such that the blown-up vector field $\bar X$ has a family of attracting periodic orbits $\bar \Gamma_\epsilon$ parameterised by $\epsilon\in(0,\epsilon_0]$, which converges to the singular cycle $\Gamma$ in the Hausdorff distance as $\epsilon\to0$.
	%\end{theorem}
	
	%To prove \ref{thm:main_bu} we define a section $\Sigma_1$ transversal to $\Gamma$, and show that the Poincaré map $\pi : \Sigma_1 \to \Sigma_1$ induced by the flow is a strong contraction. Existence of the family of limit cycles $\bar \Gamma_\epsilon$ then follows by the contraction mapping principle.

	\subsection{Poincaré map}
	\label{sec:poincare_map}

	We define a number of sections transversal to $\Gamma$, expressed in chart $K_1$ coordinates as follows:
	\begin{equation}\label{eq:sections}
		\begin{split}
			\Sigma_1 &= \left\{(h, \rho_1, \epsilon_1) : |h - h_f| \leq \alpha_1, \epsilon_1 \in [0, \beta_1] \right\}, \\
			\Sigma_2 &= \left\{(\nu_1, r_1, \epsilon_1) : |r_1 - r_{d,1}| \leq \rho_2, \epsilon_1 \in [0, \beta_2] \right\}, \\
			\Sigma_3 &= \left\{(h, \rho_1, \epsilon_1) : |h| \leq \alpha_2, \epsilon_1 \in [0, \beta_3] \right\}, \\
			\Sigma_4 &= \left\{(h, r_1, \beta_2) : |h| \leq \alpha_2, r_1 \in [0, \rho_3] \right\}, \\
			\Sigma_5 &= \left\{(h, r_1, \epsilon_{f,1} + \beta_3) : |h - h_f| \leq \alpha_3, r_1 \in [0, \rho_4] \right\}, \\
			\Sigma_6 &= \left\{(h, r_1, \beta_2) : |h - h_f| \leq \alpha_1, r_1 \in [0, \rho_5] \right\} ,
		\end{split}
	\end{equation}
	see Figure \ref{fig:cyl_bu}. We define the Poincaré map $\pi:\Sigma_1\to\Sigma_1$ induced by the flow of \eqref{eq:main_bu} by a composition of component maps
	\begin{equation}\label{eq:poincare_map}
		\pi = \pi_{6}\circ\pi_{5}\circ \pi_4\circ\pi_2 \circ\pi_2\circ\pi_1,
	\end{equation}
	where $\pi_i:\Sigma_i\to\Sigma_{i+1}$ for $i\in\{1,\ldots,5\}$, and $\pi_6:\Sigma_6\to\Sigma_1$. We consider each of the transition maps $\pi_i$ in turn. Except in the case of the map $\pi_1$, the arguments presented are similar to those in \cite{Gucwa2009}.

	\subsubsection*{The map $\pi_1:\Sigma_1\to\Sigma_2$. Boundedness argument}
	\label{ssec:map12}
	
	%\emph{Flowbox/boundedness argument.}
	We consider the map $\pi_1:\Sigma_1\to\Sigma_2$ in the original coordinates $(h,c,\epsilon)$ of the extended system \eqref{eq:main_bu}. Note that in these coordinates,
	\begin{equation}
		\begin{split}
			\Sigma_1 &= \left\{(h, \rho_1, \epsilon) : |h - h_f| \leq \alpha_1, \epsilon \in [0, \rho_1^2 \beta_1] \right\} , \\
			\Sigma_2 &= \left\{(\nu_1, c, \epsilon) : |c -  c_d| \leq \rho_2, \epsilon \in [0, (c_d + \rho_2)^2\beta_2] \right\} ,
		\end{split}
	\end{equation}
	where $c_d = r_{d,1}$. The transition map $\pi_1$ is characterised by the following result.
	
	\begin{lemma}\label{lem:map12}
		%There exists an $\epsilon_0>0$ such that for all $\epsilon \in (0,\epsilon_0)$, 
		For $\epsilon$ sufficiently small, the map $\pi_1$ is a well-defined diffeomorphism. %, and at most algebraically expanding in the variable $h$.
	\end{lemma}
	
	\begin{proof}
		We consider the auxiliary problem to the layer problem for the extended system \eqref{eq:main_bu},
		\begin{equation}\label{eq:auxilliary_layer}
			\begin{pmatrix}
				{h'} \\
				{c'} \\
				{\epsilon'}
			\end{pmatrix}
			=
			\begin{pmatrix}
				- \tau_{\rm max}^{-1} \\
				\mathfrak{J}^{(0)}_{\text{IPR}}(h,c) \\
				0
			\end{pmatrix},
		\end{equation}
		obtained after making the time desingularisation $dt = c^{-4}h^{-1} dT$. Note that the dash notation now refers to differentiation by the new time $T$. System \eqref{eq:auxilliary_layer} is equivalent to the layer problem \eqref{eq:main_bu}$|_{\epsilon = 0}$ on the relevant domain $h,c>0$. The transition time $T_{h_0}$ for a trajectory with initial condition $(h_0,\rho_1,0) \in \Sigma_1$ to reach a point $(\nu_1, \tilde c, 0) \in \Sigma_2$ can be determined by integrating the equation for $h'$. We obtain
		\[
		h(T)=- \tau_{\rm max}^{-1} T+h_0 \qquad \implies \qquad T_{h_0} = \tau_{\rm max} (h_0 - \nu_1) , 
		\]
		where $\nu_1$ is chosen so that $T_{h_0} >0$. Now fix a compact box
		\[
		R=\{(h,c) : h\in[h_-,h_+], c\in[c_-,c_+]\},
		\]
		such that $h_-,c_->0$, and $\Sigma_{1,2}\subset R$. We obtain an upper bound for the corresponding transition time $t_{h_0}$ for the layer problem \eqref{eq:main_bu}$|_{\epsilon = 0}$ via
		\[
		t_{h_0} = \int_0^{t_{h_0}} dt = \int_0^{T_{h_0}} \frac{1}{c^4(T)h(T)} dT \leq \frac{{T_{h_0}}}{c_-^4h_-} < \infty.
		\]
		It follows that $\pi_1|_{\epsilon = 0}$ is a well-defined diffeomorphism. Since the flow in $R$ is regular, it follows by regular perturbation theory and the flowbox theorem that for $\epsilon\ll1$ sufficiently small, $\pi_1$ is also a well-defined diffeomorphism.% for $\epsilon \geq 0$ sufficiently small, which is `at worst' algebraically expanding.
	\end{proof}

	\subsubsection*{The map $\pi_2:\Sigma_2\to\Sigma_3$. Fenichel theory in $(h,c,\epsilon)-$coordinates}
	\label{ssec:map23}
	%\emph{Fenichel theory.}
	The map $\pi_2$ is characterised by the following result, stated in terms of the coordinates $(h, c, \epsilon)$.
	
	\begin{lemma}\label{lemma:map45}
		For $\epsilon>0$ sufficiently small, the map $\pi_2$ is well-defined and the $c$-component of the map is contracting with rate $e^{-a_2/\epsilon}$, for a constant $a_2>0$.
	\end{lemma}
	
	\begin{proof}
		Given $\nu_1>0$ sufficiently small, Fenichel theory implies that solutions with initial conditions in $\Sigma_2$ are exponentially attracted to their base points on the slow manifold $S_{c,\epsilon}$ described in Lemma \ref{lem:slow_manifolds_R1}. Trajectories follow the slow flow on $S_{c,\epsilon}$ until they reach $\Sigma_3$. Hence the map $\pi_2:\Sigma_2 \to \Sigma_3$ is well-defined. Exponential contraction in the $c-$component follows by Fenichel theory \cite{fenichel1979geometric}.
	\end{proof}

	\subsubsection*{The map $\pi_3:\Sigma_3\to\Sigma_4$. Spherical blow-up of $Q_3$}
	\label{ssec:map34}
	
	%\emph{Spherical blow-up}.
	Let $R_3\subset\Sigma_3$ be an arbitrarily small rectangle centered at the point $(0, \rho_1, 0)$ in chart $K_1$ coordinates, i.e., at the intersection $\Gamma \cap \Sigma_3$. The transition map $\pi_3: \Sigma_3 \to \Sigma_4$ is characterised by the following result.
	
	\begin{lemma}\label{lem:map34} Given $\rho_1 > 0$ and $R_3 \subset \Sigma_3$ sufficiently small, the restricted map $\pi_3|_{R_3} = \pi_{3,R_3}$ is $C^1$-smooth with the following properties:
		\begin{itemize}
			\item[(i)] The extension of the manifold $M$ described in Lemma \ref{lem:K1_extended_manifold} intersects $\Sigma_4$ in a $C^1$-curve $\sigma_4$, which is tangent to $r_1=0$.
			\item[(ii)] Each restricted map $\pi_{3,R_3}|_{\epsilon_1 = const}$ is a strong contraction with rate $e^{-a_3/\epsilon_1}$ for a constant $a_3>0$
			%as the initial coordinate $\epsilon_{1,in} \to 0$.
			\item[(iii)] The image $\pi_3(R_3)$ is an exponentially thin wedge about $\sigma_4$ in $\Sigma_4$.
		\end{itemize}
	\end{lemma}
	
	\begin{proof}
		The proof relies on a spherical blow-up of the nilpotent point $Q_3$. For expository reasons, we defer the proof to Appendix \ref{sec:spherical_blowup}.
	\end{proof}

	\subsubsection*{The map $\pi_4:\Sigma_4\to\Sigma_5$. Fenichel theory in chart $K_2$}
	\label{ssec:map45}
	
	%\emph{Fenichel in chart $K_2$.}
	The transition map $\pi_4:\Sigma_4\to\Sigma_5$ is analysed in chart $K_2$. We may use the fact that $r_2 = \delta$, however, in order to describe the result in terms of the regime (R2) analysis given in Section \ref{ssec:singular_limit_analysis_R2}.
	% In the following result, the subscript notation in $\mathcal S_{a,\delta,2}$ and $\mathcal S_{r,\delta,2}$ indicates that the manifolds in Lemma \ref{lem:slow_manifolds_R2} are being considered in chart $K_2$.
	
	\begin{lemma}\label{lemma:61map}
		Given fixed $\rho_1 > 0$, there exists $\beta_2 > 0$ such that the map $\pi_4$ is well-defined with the following properties:
		\begin{itemize}
			\item[(i)] Each restricted map $\pi_4|_{\delta= const}$ is a strong contraction with rate $O(e^{-a_4/\epsilon^2})$ for a constant $a_4>0$.
			\item[(ii)] The image $\pi_4(\Sigma_4)$ is an exponentially thin wedge in $\Sigma_5$, which is exponentially close to the smooth curve formed by the intersection $\mathcal S_{a,\delta,2} \cap \Sigma_5$.
		\end{itemize}
	\end{lemma}
	
	\begin{proof}
		For sufficiently small $\beta_2 > 0$ our analysis in regime (R2) in Section \ref{ssec:singular_limit_analysis_R2}, in particular Lemma \ref{lem:slow_manifolds_R2}, implies that solutions with initial conditions in $\Sigma_4$ are exponentially attracted to their base points on $\mathcal S_{a,\delta,2}$, after which they follow the slow flow and intersect $\Sigma_5$ transversally. Hence $\pi_4:\Sigma_4\to\Sigma_5$ is well-defined. Statements (i)-(ii) follow by the exponential contractiveness of the slow manifolds $\mathcal S_{a,\delta,2}$. In particular, $O(e^{-a_4/\epsilon^2})$ contraction follows since solutions track the slow flow on $\mathcal S_{a,\delta,2}$ for $O(1)$ times on the infra-slow time-scale $\tilde \tau = \epsilon^2 t$; see again Remark \ref{rem:time-scales}.
	\end{proof}

	\subsubsection*{The map $\pi_5:\Sigma_5\to\Sigma_6$. Flow past the fold in chart $K_2$}
	\label{ssec:map56}

	%\emph{Regular jump point in chart $K_2$.}
	The transition map $\pi_5:\Sigma_5\to\Sigma_6$ is also considered in chart $K_2$. In these coordinates we have
	\begin{equation}
		\begin{split}
			\Sigma_5 &= \left\{\left(h, \frac{1}{(\epsilon_{f,1} + \beta_3)^2}, r_2 \right) : |h - h_{f,2}| \leq \alpha_3, r_2 \in \left[0, (\epsilon_{f,1} + \beta_3)^2 \rho_4 \right] \right\}, \\
			\Sigma_6 &= \left\{\left(h, \frac{1}{\sqrt{\beta_2}}, r_2\right) : |h - h_{f,2} | \leq \alpha_1, r_2 \in \left[0, \sqrt{\beta_2} \rho_5 \right] \right\}.
		\end{split}
	\end{equation}
	We obtain the following result.
	
	\begin{lemma}\label{lem:map12_2} %For $\Sigma_5$ fixed sufficiently small and for all sufficiently small $r_2 > 0$,
		\revsj{Fix the section $\Sigma_5$ sufficiently small. Then for all $r_2>0$ sufficiently small,~}
	%	Given fixed $\alpha_3 > 0$, there exists $\beta_1 > 0$ such that for all $r_2>0$ sufficiently small,
		the transition map $\pi_5$ is well-defined. \revsj{In particular, the } % the
		image $\pi_5(\Sigma_5) \subset \Sigma_6$ \revsj{is exponentially narrow in the $h-$coordinate ($O(e^{-a_5/r_2})$ for some $a_5>0$), and centered about the value $h = h_{f,2} + O(r_2^{2/3})$.} % is wedge-shaped with width $O(e^{-a_5/r_2})$ for fixed $a_5>0$ in the $h-$coordinate about $(h_f+O(r_2^{2/3}) , \beta^{-1/2}, r_2) \in \Sigma_6$.
	\end{lemma}
%	{\color{red}{This Lemma needs to be punctuated differently. It is too hard to parse with this format. Perhaps split the sentence at ``such that"  or replace ``such that" with ``and" and then think again about the second half of the lemma? I think you are saying that the image is wedge-shaped in the h-coordinate, with the wedge centred on the bracketed expression and with width in the h coordinate given by ... where $\alpha_5$ is a positive constant. This is hard to extract from the current sentence.}}
	
	\begin{proof}
		For $r_2 > 0$ sufficiently small, solutions with initial conditions in \revsj{$\Sigma_5$} are exponentially attracted to their base points on the slow manifolds $\mathcal S_{a,\delta,2}$ described in Lemma \ref{lem:slow_manifolds_R2}. Lemma \ref{lem:slow_manifolds_R2} also implies the extension of $\mathcal S_{a,\delta,2}$ \revsj{(and nearby trajectories)} through the neighbourhood of the fold point $Q_5$\revsj{, see again Figure \ref{fig:manifold_R2}}. \revsj{After leaving a neighbourhood of the fold, solutions follow the fast flow for finite time before intersecting $\Sigma_6$. This shows that $\pi_5$ is well-defined.} %The extended manifolds $\mathcal S_{a,\delta,2}$, and hence nearby trajectories, intersect $\Sigma_6$ transversally as shown in Figure \ref{fig:manifold_R2}, after which trajectories tracking the slow flow along $\mathcal S_{a,\delta,2}$ follow a trajectory which converges to the orbit segment $\Gamma_5$ in the limit $r_2 \to 0$. 
		%Hence for $r_2 > 0$ sufficiently small, the map $\pi_5$ is well-defined. 
		Exponential contraction follows from Fenichel theory \cite{fenichel1979geometric} and the exponential contractiveness undergone in the neighbourhood of $Q_f$. \revsj{Finally, t}he estimate %$h_f + O(r_2^{2/3})$ estimate VK: Was this included in error?
		for the $h-$component follows from Lemma \ref{lem:slow_manifolds_R2} and regular perturbation theory. %{\color{red}{Note I removed a few words from this sentence since there was clearly something wrong - Sam, please check what you intended to say.}}
	\end{proof}

	\subsubsection*{The map $\pi_6:\Sigma_6\to\Sigma_1$. Hyperbolic transition near $l_h$}
	\label{ssec:map61}

	%\emph{Resonance near the line of saddle points $l_h$.}
	%\textcolor{red}{\textbf{[This section should be redone; the resonance did not survive the new choice of rescalings etc. Resonance is expected in the $K_1$ chart only in the case that $c$ and $\epsilon$ recieve equal wieghts in the blow-up map]}.}
	The analysis of the map $\pi_6:\Sigma_6\to\Sigma_1$ is carried out in chart $K_1$. Let $p_0 = (h_{f}, 0 , \beta_2)$ denote the point of intersection at $\Gamma \cap \Sigma_1$, and let $R_6 \subset \Sigma_6$ be an arbitrarily small (but fixed) rectangle centered at $p_0$. For notational simplicity, we rewrite system \eqref{eq:K1_equations} as
	\begin{equation}\label{eq:K1_sys}
		\begin{split}
			h' &= r_1 \phi(h,\epsilon_1),             \\
			r_1' &= r_1 \psi (h,r_1,\epsilon_1), \\
			\epsilon_1' &= - 2 \epsilon_1 \psi(h,r_1,\epsilon_1),
		\end{split}
	\end{equation}
	where
	\begin{equation}
		\label{eq:K1_fns}
		\begin{split}
			\phi(h,\epsilon_1) &= \tau_{\rm max}^{-1} \left(\hat h_\infty(\epsilon_1) - h \right) \left(1 + \epsilon_1^2 \right) ,   \\
			\psi(h, r_1, \epsilon_1) &= \hat J_{\text{IPR}}(h, r_1, \epsilon_1) - \epsilon_1 \hat{\mathfrak J}^+_{\text{SERCA}}(r_1, \epsilon_1) + \epsilon_1^2 \hat{\mathfrak J}^-_{\text{SERCA}}(r_1, \epsilon_1).
		\end{split}
	\end{equation}
	For ease of computations, we translate the point $Q_1 = (h_f, 0, 0)$ to the origin via the translation $\tilde h = h - h_{f}$, obtaining the system
	\begin{equation}\label{eq:K1_new}
		\begin{split}
			\tilde h' &= r_1 \tilde \phi(\tilde h,\epsilon_1),             \\
			r_1' &= r_1 \tilde \psi (\tilde h,r_1,\epsilon_1),  \\
			\epsilon_1' &= - 2 \epsilon_1 \tilde \psi(\tilde h,r_1,\epsilon_1),
		\end{split}
	\end{equation}
	where $\tilde \phi(\tilde h,\epsilon_1)=\phi(\tilde h+h_{f},\epsilon_1)$ and $\tilde \psi (\tilde h,r_1, \epsilon_1)= \psi (\tilde h+h_{f},r_1,\epsilon_1)$. Since \SJ{$\tilde \psi (0,0,0) =  p^2 \kipr \gamma c_t h_f  / k_\beta K_c^4 K_p^2 > 0$,}  
	%and so for the purposes of a local computation, 
	%$\tilde\psi(\tilde h,\epsilon_1,r_1)$ is strictly positive in a neighbourhood of the origin, so 
	we may consider the equivalent system obtained after division by this term:
	\begin{equation}\label{eq:map_61_sys}
		\begin{split}
			\tilde h' &= r_1 \left(\frac{\tilde \phi(\tilde h,\epsilon_1)}{\tilde \psi(\tilde h,r_1,\epsilon_1)} \right), \\
			r_1' &= r_1, \\
			\epsilon_1' &= - 2 \epsilon_1.                                  
		\end{split}
	\end{equation}
	System \eqref{eq:map_61_sys} has a non-hyperbolic equilibrium at $(0,0,0)$ with eigenvalues $0,1,-2$. %Resonant terms give rise to logarithmic terms which cannot be eliminated in a normal form transformation. We are able to prove however, that
	The transition map $\pi_6:\Sigma_2\to\Sigma_1$ is characterised by the following result.
	
	\begin{lemma}\label{lem:map23} For $\beta_2 > 0$ sufficiently small and a sufficiently small rectangle $R_6 \subset \Sigma_6$ centered at $p_0$, the restricted transition map $\pi_6|_{R_6} = \pi_{6,R_6}$ is well-defined and given by
		\begin{equation}
			\label{eq:map61}
			\pi_6|_{R_6} : (h, r_1, \beta_2) \mapsto \left(h_f + O (r_1,h-h_f), \rho_1, \beta_2 \left(\frac{r_1}{\rho_1} \right)^2 \right) .
		\end{equation}
		%is well-defined, and satisfies
		%\begin{equation}\label{eq:map23_bounds}
		%\begin{split}
		%h_{1,in} - \left(\frac{\beta_2 - r_{1,in}}{\tau_{max} K_\tau^4 \hat \jipr(h_f, 0, 0)}\right)  - & c_6 \ln \left(\frac{\beta_2}{r_{1,in}}\right) \leq h_{1,out} \\
		%& \leq h_{1,in} - \left(\frac{\beta_2 - r_{1,in}}{\tau_{max} K_\tau^4 \hat \jipr(h_f, 0, 0)}\right),
		%\end{split}
		%\end{equation}
		%where the constant $c_6$ can be chosen arbitrarily small for $\beta_2$ small \textcolor{red}{\textbf{[dependent on the claim in the proof below]}}. Moreover, restrictions of $\pi_6$ to lines $r_{1,in}=const.$ are at most algebraically expanding.
		%	where $a = (\tau_{max} A_{\text{\text{IPR}}} \gamma c_t )^{-1} > 0$.
		%	\begin{equation}
		%	\label{eq:expansion_const}
		%	a = \frac{1}{K_\tau^4 \tau_{max} A_{\text{\text{IPR}}} \gamma c_t} .
		%	\end{equation}
	\end{lemma}
	
	\begin{proof}
		We consider a solution $(\tilde h,\epsilon_1,r_1)(t)$ of \eqref{eq:map_61_sys} which satisfies
		\[
		(\tilde h,r_1,\epsilon_1 ) (0) = ( \tilde h_{in}, r_{1,in}, \beta_2 ) , \qquad
		(\tilde h,r_1,\epsilon_1 ) (T) = ( \tilde h_{out}, \rho_1, \epsilon_{1,out} ) .
		\]
		\begin{comment}
		\[
		\begin{cases}
		\tilde h(0) &= \tilde h_{1,in}, \\
		\epsilon_1(0) &= \beta_2, \\
		r_1(0) &= r_{1,in},
		\end{cases}
		\qquad
		\begin{cases}
		\tilde h(T) = \tilde h_{1,out}, \\
		\epsilon_1(T) = \epsilon_{1,out}, \\
		r_1(T) = \rho_1.
		\end{cases}
		\]
		\end{comment}
		Direct integration yields $\epsilon_1(t) = \beta_2 e^{- 2 t}$ and $r_1(t) = r_{1,in} e^t$, which leads to an expression for the transition time
		\[
		T = \ln \left(\frac{\rho_1}{r_{1,in}}\right) ,
		\]
		proving that
		\[
		\pi_6 : (h_{1,in},r_{1,in},\beta_2) \mapsto \left(h_{1,out}, \rho_1, \beta_2 \left(\frac{r_{1,in}}{\rho_1} \right)^2 \right).
		\]
		Expanding the equation for $\tilde h'$ in \eqref{eq:map_61_sys} about $(0,0,0)$ gives
		\[
		\tilde h'=-ar_1+O\left(r_1\tilde h,r_1 \epsilon_1^2\right),
		\]
		where $a = (\tau_{\rm max} A_{\text{\text{IPR}}} \gamma c_t )^{-1} > 0$. 
		%CLAIM: there exists some constant $c_6$ such that $dr_1^2+\mathcal{O}(r_1\tilde h^2,r_1\tilde h \epsilon_1,r_1,\epsilon_1^2,r_1^3)<c_6$ on the relevant domain, with $\beta_2$ chosen sufficiently small
		%If the preceding claim is true, we have that
		%\[
		%- a r_1 - c_6 \leq \tilde h' \leq - a r_1,
		%\]
		%from which one obtains the inequality \eqref{Map23Bounds} after integration.
		Since the expression for $\tilde h'$ in \eqref{eq:map_61_sys} is $C^1$, the order is well-behaved with respect to integration. The expression in \eqref{eq:map61} follows after direct integration and a coordinate translation $h = \tilde h + h_f$ which undoes the earlier transformation.
		%	
		%	\textcolor{red}{\textbf{[Needs to be completed: still need to prove the claim (or something similar to it) and provide a Lipschitz argument to prove $\pi_2$ is at most algebraically expanding on lines $r_{1,in}=const.$.]}}
	\end{proof}

	\subsection{Proof of Theorem \ref{thm:main}}
	\label{ssec:proof_of_theorem}
	
	\begin{proof}
		By the analysis presented in Section \ref{ssec:map12}, the Poincaré map $\pi:\Sigma_1\to\Sigma_1$ defined by the composition in \eqref{eq:poincare_map} is well-defined. Note that one must also include coordinate changes between charts in expression \eqref{eq:poincare_map}.
		
		Because $\epsilon$ is a constant of the motion in \eqref{eq:main_bu}, the lines $\epsilon={\rm const}$ are invariant under $\pi$. Since the relevant components of the restricted maps $\pi_i|_{\{\epsilon={\rm const}\}}$ are exponentially contracting for $i \in \{2,3,4,5\}$, % and at most algebraically expanding for $i \in \{1,2\}$, 
		it follows that the $h-$component of the restricted map $\pi|_{\{\epsilon={\rm const}\}}$ is also exponentially contracting. By the contraction mapping theorem, each $\pi|_{\{\epsilon={\rm const}\}}$ has a unique fixed point corresponding to an exponentially attracting periodic orbit $\Gamma_\epsilon$, and the family $\Gamma_\epsilon$ converges in the Hausdorff distance to the singular cycle $\Gamma$ as $\epsilon \to 0$. Theorem \ref{thm:main} follows after applying the blow-down transformation associated with the map \eqref{eq:cylindrical_bu}. 
		The $O(\epsilon^{1/3})$ separation follows from Lemma \eqref{lem:map12_2} after applying $r_2 = \sqrt{\epsilon}$; see again the discussion immediately following the statement of Theorem \ref{thm:main}. The $-\kappa/\epsilon^2$ bound on the Floquet exponent follows from Lemma \ref{ssec:map45}.
	\end{proof}

	\subsection{Proof of Lemma \ref{lem:map34}}
	\label{sec:spherical_blowup}

	In this section we prove Lemma \ref{lem:map34}. We start in $K_1$ coordinates with system \eqref{eq:K1_sys}, and drop the subscripts for notational convenience, i.e., we consider
	\begin{equation}
		\begin{split}
			h' &= r \phi(h,\epsilon),                   \\
			r' &= r \psi (h,r,\epsilon), \\
			\epsilon' &= - 2 \epsilon \psi(h,r,\epsilon), 
		\end{split}
	\end{equation}
	with $\phi$ and $\psi$ as defined in \eqref{eq:K1_fns}. We are interested in the dynamics near the nilpotent singularity $Q_3 = (0,0,0)$, where the Jacobian has an eigenvalue $\lambda = 0$ of multiplicity three; recall Lemma \ref{lem:K1_dynamics}. In order to resolve this, we define a spherical blow-up by the transformation
	\begin{equation}
		\label{eq:spherical_bu}
		s \geq 0, \ \left(\bar h, \bar r, \bar \epsilon \right) \in S^2 \mapsto
		\begin{cases}
			h = s \bar h , \\
			r = s \bar r , \\
			\epsilon = s \bar \epsilon ,
		\end{cases}
	\end{equation}
	which maps the point $Q_3$ to the sphere $\{s = 0\} \times S^2$. \SJ{We work in coordinate charts defined via $\mathcal K_1: \bar r=1$ and $\mathcal K_2: \bar\epsilon=1$, with} chart-specific coordinates
	\begin{equation}\label{eq:chart_coordinates_sphere}
		\begin{aligned}
			\mathcal K_1: \ h &= s_1 h_1, && r = s_1,  && \epsilon = s_1 \epsilon_1 , \\
			\mathcal K_2: \ h &= s_2 h_2, && r = s_2 r_2,  && \epsilon = s_2.
		\end{aligned}
	\end{equation}
	The transition maps between charts $\mathcal K_1$ and $\mathcal K_2$ are given by
	\begin{equation}
		\begin{aligned}
			\label{eq:sphere_trans_maps}
			\kappa_{12}: \ h_1 &= r_2^{-1}h_2, && s_1=r_2 s_2, && \epsilon_1=r_2^{-1}, \\
			\kappa_{21}: \ h_2 &= \epsilon_1^{-1}h_1, &&   r_2=\epsilon_1^{-1}, && s_2=s_1\epsilon_1.
		\end{aligned}
	\end{equation}
	Note that in chart $\mathcal K_1$ and $\mathcal K_2$ coordinates, $\Sigma_3 \subset \left\{s_1 = \rho_1 \right\}$ and $\Sigma_4 \subset \{s_2 = \beta_2 \}$.

	\subsubsection*{Chart $\mathcal K_1$ dynamics}

	After a suitable desingularisation (division by $s_1$), we obtain the following equations in chart $\mathcal K_1$:
	\begin{align}\label{eq:Ks1_equations}
		\begin{array}{lcl}
			h_1' = \bar{\phi}_1(h_1,\epsilon_1,s_1) - h_1 \bar{\psi}(h_1,\epsilon_1,s_1), \\
			\epsilon_1' = - 3 \epsilon_1\bar{\psi}_1(h_1,\epsilon_1,s_1),         \\
			s_1'= s_1 \bar{\psi}_1(h_1,\epsilon_1,s_1),
		\end{array}
	\end{align}
	where $\bar{\phi}_1(h_1,\epsilon_1,s_1)=s_1^{-1}\phi(s_1h_1,s_1\epsilon_1)$ and $\bar{\psi}_1(h_1,\epsilon_1,s_1)=s_1^{-1}\psi(s_1h_1,s_1,s_1\epsilon_1)$ are well-defined for $s_1=0$ due to \SJ{a} common factor of $s_1$ in the respective numerators. System \eqref{eq:Ks1_equations} has a line of steady states
	\[
	L_{c,1} = \{(0,0,s_1) : s_1\geq0\} ,
	\]
	and we denote the endpoint of $K_{c,1}$ by $p_c = (0,0,0) \in L_{c,1}$. We also identify the following invariant subspaces:
	\begin{enumerate}
		\item[(i)] the plane $\epsilon_1 = 0$;
		\item[(ii)] the plane $r_1 = 0$;
		\item[(iii)] the $h_1-$axis $r_1 = \epsilon_1 = 0, h_1 \geq 0$;
		\item[(iv)] the $\epsilon_1-$axis, which we denote by
		\begin{equation}
			\label{eq:inv_line_K1}
			\gamma_1 = \left\{(0, \epsilon_1, 0) : \epsilon_1 \geq 0 \right\} .
		\end{equation}
	\end{enumerate}
	For the purpose of stating the following result, we also write the rectangle $R_3 \subset \Sigma_3$ in chart $\mathcal K_1$ coordinates:
	\[
	R_3 = \left\{(h,\epsilon_1,\rho_1) : |h| \leq \tilde \alpha, \epsilon_1 \in [0,\tilde\alpha] \right\} .
	\]

	\begin{lemma}\label{lem:extended_manifolds} The following holds for system \eqref{eq:Ks1_equations}:
		\begin{enumerate}
			\item[(i)] The line $L_{c,1}$ is normally hyperbolic and attracting, with eigenvalues \SJ{$\lambda = - \tau_{\rm max}^{-1}, \ 0, \ 0$}.
			%with associated stable and center eigenspaces given by
			%		\begin{equation}
			%		\begin{split}
			%		E^s &= span\{(-1/(\tau_{max} K_\tau^4), 0, s_1 a(s_1))^T\}, \\
			%		E^c &= span\{(0,0,1)^T,(b(s_1),a(s_1),0)^T\},
			%		\end{split}
			%		\end{equation}
			%		respectively, where
			%		\[
			%		a(s_1) = \frac{\kipr p^2 (\gamma  c_t- (1 + \gamma )s_1)}{k_\beta K_p^2
			%			\left(K_c^4+s_1^4\right)}, \qquad b(s_1) = -\frac{V_s}{K_s^2+s_1^2}.
			%		\]
			%		$E^s$ and $E^c$ are linearly independent for $s_1 < \gamma c_t/(1 + \gamma)$.
			%	In particular, this is true at the point $p_c \in L_{c,1}$.
			\item[(ii)] There exists an attracting two-dimensional center manifold $\mathcal{M}_1$ with graph representation
			\begin{equation}
				\label{eq:M1}
				h_1 = K_h^4 s_1 \epsilon_1^2 + O(s_1^2 \epsilon_1^2) ,
			\end{equation}
			containing $L_{c,1}$ and the invariant $\epsilon_1$-axis as restrictions $\mathcal M_1|_{\epsilon_1 = 0}$ and $\mathcal M_1|_{s_1 = 0}$ respectively. The manifold $\mathcal M_1$ can be chosen to be the continuation of the manifold $M_1$ in Lemma \ref{lem:K1_extended_manifold} under the flow, and the variable $\epsilon_1$ is strictly increasing on $\mathcal M_1 \setminus L_{c,1}$.
		\end{enumerate}
	\end{lemma}

	\begin{proof}
		%The Jacobian restricted to $L_{c,1}$ is
		%	\[
		%	J|_{L_c}=
		%	\begin{pmatrix}
		%	- 1/(\tau_{max} K_\tau^4) & 0 & 0 \\
		%	0  & 0 & 0 \\
		%	a(s_1)  & b(s_1) & 0
		%	\end{pmatrix}.
		%	\]
		Statement (i) follows after linearisation of system \eqref{eq:Ks1_equations}, and existence of an attracting two-dimensional center manifold $\mathcal M$ at $p_c$ follows from center manifold theory. The graph representation \eqref{eq:M1} can be determined by standard matching arguments. Restricting system \eqref{eq:Ks1_equations} to $\mathcal M_1 \setminus L_{c,1}$, we obtain
		\begin{equation}
			\begin{split}
				\epsilon_1' &= 3 A_{\text{SERCA}} \epsilon_1^2 + O(\epsilon_1^3) , \\
				s_1' &= A_{\text{SERCA}} s_1 \epsilon_1 + O(s_1 \epsilon_1^2) .
			\end{split}
		\end{equation}
		and hence $\epsilon_1' > 0$, since $\epsilon_1 > 0$ on $\mathcal M_1 \setminus L_{c,1}$.
	\end{proof}

	\subsubsection*{Chart $\mathcal K_2$ dynamics}

	After a suitable desingularisation (division by $s_2$), the dynamics in chart $\mathcal K_2$ are governed by
	\begin{align}\label{eq:Ks2_equations}
		\begin{array}{lcl}
			h_2' = r_2 \bar{\phi}_2(h_2,r_2,s_2) + 2 h_2 \bar{\psi}_2(h_2,r_2,s_2), \\
			r_2' = 3 r_2 \bar{\psi}_2(h_2,r_2,s_2),  \\
			s_2' = - 2 s_2\bar{\psi}_2(h_2,r_2,s_2),
		\end{array}
	\end{align}
	where \SJ{$\bar{\phi}_2(h_2,r_2,s_2)=s_2^{-1}\phi(s_2h_2,s_2)$ and $\bar{\psi}_2(h_2,r_2,s_2)=s_2^{-1}\psi(s_2h_2,s_2r_2,s_2)$,} which are well-defined for $s_2=0$ due to a common factor of $s_2$ in the respective numerators. System \eqref{eq:Ks2_equations} has an equilibrium $p_s = (0,0,0)$, and the following invariant subspaces:
	\begin{enumerate}
		\item[(i)] the plane $r_2 = 0$;
		\item[(ii)] the plane $s_2 = 0$;
		\item[(iii)] the $h_2-$axis $r_2 = s_2 = 0, h_2 \geq 0$;
		\item[(iv)] the $r_2-$axis, which we denote by
		\begin{equation}
			\label{eq:inv_line_K2}
			\gamma_2 = \left\{(0, r_2, 0) : r_2 \geq 0 \right\} ;
		\end{equation}
		\item[(v)] the $s_2-$axis $h_2 = r_2 = 0, s_2 \geq 0$.
	\end{enumerate}
	We obtain the following result.
	
	\begin{center}
		\captionsetup{format=plain}
		\begin{figure}[t!]
			\hspace{0em}\centerline{\includegraphics[scale=0.55]{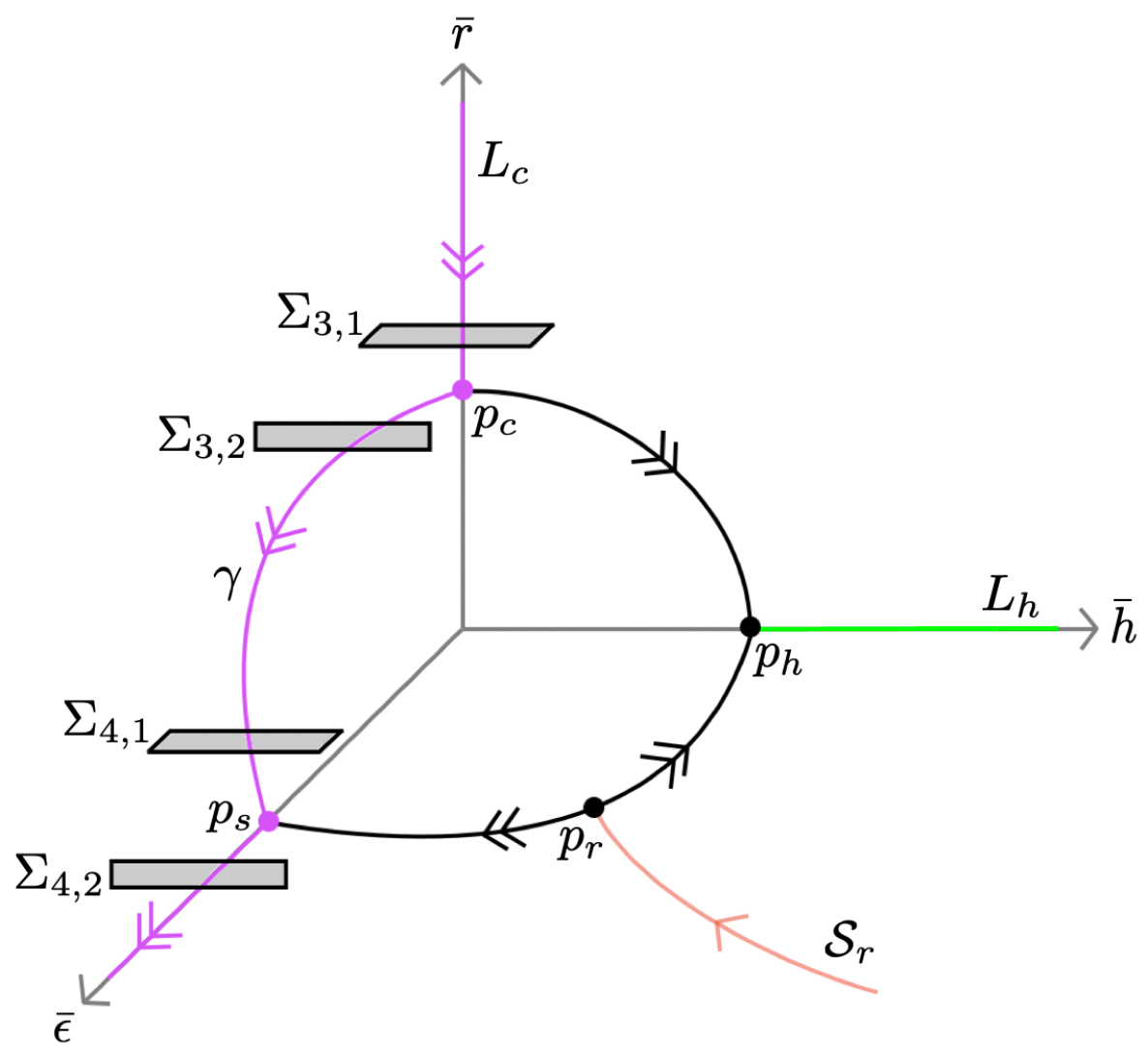}}
			\caption{Dynamics following spherical blow-up of $Q_3$. The relevant part of the singular relaxation cycle is shown in purple. Also shown are the transversal segments $\Sigma_{3,i}$ and $\Sigma_{4,i}$, $i = 1, 2$ used in the construction of the map $\pi_3: \Sigma_3 \to \Sigma_3$.}\label{fig:sphere_bu}
		\end{figure}
	\end{center}
	
	\begin{lemma}\label{lem:sphK1_dynamics} The following hold for system \eqref{eq:Ks2_equations}:
		\begin{enumerate}
			\item[(i)] The equilibrium $p_s$ is hyperbolic with eigenvalues \SJ{given by $\lambda = - 2 A_{\text{SERCA}}$, $- 3 A_{\text{SERCA}}$, $2 A_{\text{SERCA}}$, with} corresponding eigenvectors $(1,0,0)^T$, $(0,1,0)^T$, and $(0,0,1)^T$. % respectively. 
			The unstable manifold $W^u(p_s)$ lies within the $s_2-$axis $h_2 = r_2 = 0, s_2 \geq 0$.
			\item[(ii)] The invariant line $\gamma_1$ in \eqref{eq:inv_line_K1} coincides with the invariant line $\gamma_2$ in \eqref{eq:inv_line_K2} where domains overlap.
		\end{enumerate}
	\end{lemma}
	
	\begin{proof}
		The first part of statement (i) follows after linearisation of system \eqref{eq:Ks2_equations}. The fact that $W^u(p_s)$ lies within the $s_2-$axis follows from the form of the equations restricted to $h_2 = r_2 = 0$, namely
		\[
		s_2' = - 2 s_2 \bar \psi_2(0, 0, s_2) = \frac{2 V_s s_2}{K_s^2 + s_2^2} + O(s_2^2) \geq 0.
		\]
		Statement 2 follows by an application of the transition map $\kappa_{12}$ in \eqref{eq:sphere_trans_maps}.
	\end{proof}
	
	In particular, Lemma \ref{lem:sphK1_dynamics} implies the existence of a heteroclinic orbit $\gamma$ on the blown-up locus connecting $p_c$ and $p_s$. The resulting (global) singular limit analysis is shown in Figure \ref{fig:sphere_bu}.
	
	\begin{rem}
		Figure \ref{fig:sphere_bu} shows objects that are not explicitly identified in our analysis, since they do not play an important role in the relevant dynamics. We simply note here that one can prove the existence of an additional (partially hyperbolic) equilibrium $p_r$ corresponding to the endpoint of the (extension of the) repelling critical manifold $\mathcal S_r$ in chart $\mathcal K_2$. By looking in an additional chart $\bar h = 1$ one also identifies the image of the line of saddle-type steady states $l_{h}$, denoted $L_h$ in Figure \ref{fig:sphere_bu}, which terminates at a point $p_h$ on the blow-up sphere.
	\end{rem}

	\subsubsection{The map $\pi_3$}
	
	In order to prove Lemma \ref{lem:map34}, we consider the map $\pi_3 : \Sigma_3 \to \Sigma_4$ as a composition 
	\[
	\pi_3 = \pi_{4,1} \circ \pi_{3,2} \circ \pi_{3,1},
	\]
	where $\pi_{3,1} : \Sigma_{3,1} \to \Sigma_{3,2}$, $\pi_{3,2} : \Sigma_{3,2} \to \Sigma_{4,1}$ and $\pi_{4,1} : \Sigma_{4,1} \to \Sigma_{4,2}$ 
	%\[
	%\pi_{3,1} : \Sigma_{3,1} \to \Sigma_{3,2}, \qquad \pi_{3,2} : \Sigma_{3,2} \to \Sigma_{4,1}, \qquad \pi_{4,1} : \Sigma_{4,1} \to \Sigma_{4,2},
	%\]
	denote transition maps induced by the flow. We consider each map in turn. The arguments presented are similar to those in \cite[Appendix A]{Gucwa2009}.

	\subsubsection*{The map $\pi_{3,1}:\Sigma_{3,1} \to \Sigma_{3,2}$. Extension of $\mathcal M$ onto the blow-up sphere}
	
	%\emph{Extension of $\mathcal M$ onto the blow-up sphere}
	Here we are interested in the dynamics near the point $p_c$. The analysis is carried out in chart $\mathcal K_1$, and we define
	\[
	\Sigma_3 = \Sigma_{3,1}=\left\{(h_1,\epsilon_1,\rho_1) : |h_1| \leq \tilde \alpha_3 ,\epsilon_1 \in \left[0,\tilde \beta_3 \right] \right\},
	\]
	and
	\[
	\Sigma_{3,2}=\left\{(h_1,\rho_1,s_1) : |h_1| \leq \tilde \alpha_3, s_1 \in \left[0, \tilde \beta_3 \right] \right\}, 
	\]
	where $\tilde \alpha_3 := \alpha_2 / \rho_1$ and $\tilde \beta_3 := \beta_3 / \rho_1$. The following result describes the map $\pi_{3,1} : \Sigma_{3,1} \to \Sigma_{3,2}$.
	
	\begin{lemma}\label{lem:map3132} Given $\tilde \alpha_3, \tilde \beta_3 > 0$ sufficiently small, the restricted transition map $\pi_{3,1}|_{R_3} = \pi_{3,1,R_3}$ is well-defined and $C^1$ with the following properties:
		\begin{itemize}
			\item[(i)] The intersection $\mathcal M_1 \cap \Sigma_{3,2}$ is a smooth curve given by the graph
			\[
			h_1 = -K_h^4 \rho_1^2 s_1 + O(\rho_1^2 s_1^2) .
			\]
			\item[(ii)] Restricted to lines $\epsilon_1={\rm const}$ the map $\pi_{3,1,R_3}$ is exponentially contracting with rate $e^{-\tilde a_1/\epsilon_1}$ for a constant $\tilde a_1 > 0$.
		\end{itemize}
	\end{lemma}
	
	\begin{proof}
		Statements (i) and (ii) follow immediately from Lemma \ref{lem:K1_extended_manifold}.
	\end{proof}

	\subsubsection*{The map $\pi_{3,2}: \Sigma_{3,2} \to \Sigma_{4,1}$. Tracking $\mathcal M$ over the blow-up sphere}
	
	%\emph{Tracking $\mathcal M$ over the blow-up sphere.}
	We are interested here in the extension of the manifold $\mathcal M$ in chart $\mathcal K_1$. We define
	\[
	\Sigma_{4,1} = \left\{\left(h_1, \frac{1}{\beta_2}, s_1 \right) : |h_1| \leq \alpha_2, s_1 \in [0, \rho_3] \right\}.
	\]
	The dynamics are summarised in the following result.
	
	\begin{lemma}\label{lem:map3241} Given sufficiently small $\rho_1, \alpha_2, \beta_2 > 0$, the map $\pi_{3,2} : \Sigma_{3,2} \to \Sigma_{4,1}$ is a  well-defined diffeomorphism. The intersection of the extension of $\mathcal M_1$ under the flow with $\Sigma_{4,1}$ is a smooth curve with tangent vector
		\begin{equation}
			\label{eq:tangent_vector}
			t_Q \approx \left(- K_h^4 \rho_1^2 \left( \frac{1 - 3A_{\rm{SERCA}} \rho_1^2}{\rho_1 \beta_2} \right)^{1/3} , 0 , \left(\frac{A_{\rm{SERCA}} + \rho_1 \beta_2 - 1}{A_{\rm{SERCA}} (1 - 3 \rho_1)^2)}\right)^{1/3}\right)^T
		\end{equation}
		at $Q$.
	\end{lemma}

	\begin{proof}
		The intersection $\gamma_1 \cap \Sigma_{3,2}$ occurs at $P = (0, \rho_1, 0)$, and the intersection $\gamma_1 \cap \Sigma_{4,1}$ occurs at $Q = (0, 1 / \beta_2 ,0)$. Since the flow from $P$ to $Q$ along $\gamma_1$ is regular, all solutions with initial conditions in $\Sigma_{3,2}$ reach $\Sigma_{4,1}$ in finite time if $\rho_1, \alpha_2, \beta_2 > 0$ are sufficiently small. It follows that $\pi_{3,2}$ is a well-defined diffeomorphism.
		
		In order to understand the continuation of the manifold $\mathcal M_1$, we compute the evolution of its tangent space along the line $\gamma_1$. Note that in the parameterisation given in \eqref{eq:inv_line_K1}, $\epsilon = \rho_1$ corresponds to the point $P\in\Sigma_{3,2}$, and $\epsilon_1 = 1 / \beta_2$ corresponds to the point $Q \in \Sigma_{4,1}$. The variational equations along $\gamma_1$ are
		\begin{equation}\label{eq:variational_equations_a}
			\begin{pmatrix}
				\delta h_1'         \\
				\delta \epsilon_1 ' \\
				\delta s_1'
			\end{pmatrix}
			=
			\left(
			\begin{array}{ccc}
				A_{\text{SERCA}} \epsilon_1 - \tau_{\rm max}^{-1} & 0 & \tau_{\rm max}^{-1} K_h^4 \epsilon_1^2 \\
				- 3 A_{\text{\text{IPR}}} \gamma c_t \epsilon_1 & 6 A_{\text{SERCA}} \epsilon_1 & - 3
				c_t^2 K A_{\text{SERCA}} \gamma ^2 \epsilon_1^3 \\
				0 & 0 & - A_{\text{SERCA}} \epsilon_1 \\
			\end{array}
			\right)
			\begin{pmatrix}
				\delta h_1        \\
				\delta \epsilon_1 \\
				\delta s_1
			\end{pmatrix},
		\end{equation}
		coupled to the equation
		\begin{equation}\label{eq:variational_equations_b}
			\epsilon_1' = 3 A_{\text{SERCA}} \epsilon_1^2.
		\end{equation}
		Invariance of the $\epsilon_1-$axis guarantees that the vector $(0,1,0)^T$ is tangent to $\mathcal M_1$ at both $P$ and $Q$. By Lemma \ref{lem:map3132}, a second (linearly independent) tangent vector at $P$ is given by $(-K_h^4 \rho_1^2 , 0 , 1)^T$. This gives an initial value problem for the variational equations \eqref{eq:variational_equations_a} coupled to \eqref{eq:variational_equations_b} with
		\[
		\delta h_1(\rho_1) = -K_h^4 \rho_1^2 , \qquad \delta \epsilon_1(\rho_1) = 0, \qquad \delta s_1(\rho_1) = 1.
		\]
		Integrating equation \eqref{eq:variational_equations_b}, we obtain
		\[
		\epsilon_1(t_1) = \frac{\rho_1}{1 - 3 \rho_1 A_{\text{SERCA}} t_1}, 
		\]
		from which we obtain the following expression for the time $T$ taken for solutions to reach point $Q$:
		\[
		T = \frac{1 - \rho_1 \beta_2}{3 \rho_1 A_{\text{SERCA}}} .
		\]
		Plugging the expression for $\epsilon_1(t_1)$ into \eqref{eq:variational_equations_a}, solving the initial value problem and evaluating it at $T$ yields the desired result.
	\end{proof}

	\subsubsection*{The map $\pi_{4,1} : \Sigma_{4,1} \to \Sigma_{4,2}$. Hyperbolic transition near $p_s$}
	
	%\emph{Resonance. Explicit in $\mathcal K_2$.}
	We consider the dynamics near the hyperbolic equilibrium $p_s$ in chart $\mathcal K_2$, for which
	\[
	\Sigma_{4,1} = \left\{(h_2, \beta_2, s_2) : |h_2| \leq \beta_2 \alpha_2, s_2 \in \left[0, \tilde \rho_4  \right] \right\} ,
	\]
	and
	\[
	\Sigma_4 = \Sigma_{4,2} = \left\{(h_2, r_2, \beta_2) : |h_2| \leq \tilde \alpha_4, r_2 \in \left[0, \tilde \rho_4 \right] \right\}, 
	\]
	where $\tilde \rho_4: = \rho_3 / \beta_2$ and $\tilde \alpha_4 := \alpha_2 /\beta_2$. Noting that %near $p_s = (0,0,0)$ we have
	%\bar \psi_2(h_2,r_2,s_2) \sim \bar 
	\SJ{$\psi_2(0,0,0) = - A_{\text{SERCA}} < 0$,} we may consider the system obtained from \eqref{eq:Ks2_equations} after dividing the right hand side by a locally positive factor of $- \bar \psi_2(h_2,r_2,s_2)$:
	\begin{align}\label{eq:Ks2_desingularised_equations}
		\begin{array}{lcl}
			h_2' = - 2h_2 - r_2 \left(\frac{\bar{\phi}_2(h_2,r_2,s_2)}{\bar\psi_2(h_2,r_2,s_2)}\right), \\
			r_2' = - 3 r_2,                                                                             \\
			s_2' = 2 s_2.
		\end{array}
	\end{align}
	Like \eqref{eq:Ks2_equations}, system \eqref{eq:Ks2_desingularised_equations} has a hyperbolic saddle at $p_s = (0,0,0)$.
	%Recall that resonace is an issue only in the invariant plane $\{ r_2 = 0 \}$. \textcolor{red}{\textbf{[I have commented the rest of this section out, since previous calculations dealt with different blow-up weights in which we had a genuine resonance. This is no longer an issue]}}.

	\begin{lemma}\label{lem:map4142}
		Given $\alpha_2, \beta_2, \tilde \alpha_4, \tilde \rho_4 > 0$ sufficiently small, the transition map $\pi_{4,1}$ is $C^1-$smooth with form
		\begin{equation}
			\label{eq:map41}
			\pi_{4,1} : (h_2, \beta_2, s_2) \mapsto 
			\left( O\left(\frac{h_2 s_2}{\beta_2} \right) , \beta_2 \left(\frac{s_2}{\beta_2} \right)^2, \beta_2 \right) .
		\end{equation}
		%	where
		%	\[
		%	\tilde \pi_{5,3}(h_{in},s_{in})=\left(\frac{s_{in}}{\delta}\right)^2h_{in}+...
		%	\]
	\end{lemma}
	
	\begin{proof}
		Consider a solution $(h_2,r_2,s_2)(t)$ for \eqref{eq:Ks2_desingularised_equations} which satisfies
		\[
		(h_2,r_2,s_2)(0) = (h_{1,in}, \beta_2, s_{2,in}) , \qquad 
		(h_2,r_2,s_2)(T) = (h_{2,out}, r_{2,out}, s_{2,in}, \beta_2) .
		\]
		\begin{comment}
		\[
		\begin{cases}
		h_2(0) &= h_{1,in}, \\
		r_2(0) &= \beta_2, \\
		s_2(0) &= s_{2,in},
		\end{cases}
		\qquad
		\begin{cases}
		h_2(T) = h_{2,out}, \\
		r_2(T) = r_{2,out}, \\
		s_2(T) = \beta_2.
		\end{cases}
		\]
		\end{comment}
		Direct integration yields $r_2(t) = \beta_2 e^{- 3 t}$ and $s_2(t) = s_{2,in} e^{2 t}$, which leads to an expression for the transition time
		\[
		T = \frac{1}{2} \ln \left(\frac{\beta_2}{s_{2,in}}\right) ,
		\]
		proving that the transition map $\pi_{4,1}$ is of the form
		\[
		\pi_{4,1}: (h_{2,in}, \beta_2, s_{2,in} ) \mapsto \left(h_{2,out}, \beta_2 \left(\frac{s_{2,in}}{\beta^2} \right)^2, \beta_2 \right).
		\]
		The estimate for $h_{2,out}$ in \eqref{eq:map41} follows by an application of Belitskii's theorem \cite{Belitskii1973}, see also \cite[Theorem 3.1]{Homburg2010}, which guarantees a $C^1$ transition of the desired form.
		%, it suffices to notice that the right-hand-side of the equation for $h_2'$ in \eqref{eq:Ks2_desingularised_equations} is $C^1$ so that the order does not change with integration. 
	\end{proof}

	\subsubsection{Proof of Lemma \ref{lem:map34}}
	
	By Lemmas \ref{lem:map3132}, \ref{lem:map3241}, \ref{lem:map4142}, the restricted map $\pi_{3,R_3}$ is $C^1$ since it is a restriction of a composition of the $C^1$ maps $\pi_{3,i}$ and $\pi_{4,i}$, $i=1,2$. Statements (ii) and (iii) in Lemma \ref{lem:map34} follow from Lemmas \ref{lem:map3132}, \ref{lem:map3241}, \ref{lem:map4142}, with strong contraction due to the map $\pi_{3,1}$.
	
	Smoothness properties of the curve $\sigma_4$ in statement (i) follow from the fact that $\pi_{3,R_3}$ is $C^1$, and the fact that $\sigma_4$ is tangent to $r_1 = 0$ follows if we consider the expression for the tangent vector $t_Q$ in \eqref{eq:tangent_vector} as a first order approximation of the curve $\mathcal M\cap \Sigma_{4,1}$ and apply the map $\pi_{4,1}$ in Lemma \ref{lem:map4142}.

%	\begin{comment}
	
%	\section{Proof of Theorem \ref{thm:hopf}}
%	\label{app:hopf_proof}
	
%	\com{\textbf{Combine with the next appendix}} 
	
%	\

%	\end{comment}

	\section{Onset of oscillations}
	\label{sec:bifurcations}

	Here, we briefly address the basic mechanisms leading to the onset of oscillations under parameter variation. We focus on three important model parameters: total calcium concentration $c_t$, the IP$_3$ concentration $p$ and the time-scale parameter $\tau_{\rm max}$ of the $h$-dynamics.
	%\WM{Sam: please see my comments in this section. I'd like to focus on complete canard explosion and mention incomplete only in a remark. Then the regular AH-bif. The aim here si to compare frequency/period at onset and how it compares to the relax osc regime. }
	%We discuss \wm{two distinct} mechanisms for the onset of oscillations in system \eqref{eq:general_form} under parameter variation. 
	
	%\WM{[NZ] the next section is completely new}
	
	%\subsection{Canard induced oscillations}
	\subsection{Singular Andronov-Hopf bifurcation and canard explosion under variation of $c_t$ or $p$}
	\label{sub:can_explosion}
	
	%{\color{red}{VK: Some rewording below, to remove repetition}}
	%\wm{
	%A first important observation in system \eqref{eq:general_form} is that the variation of $\tau_{\rm max}$
	%\WM{[Sam] please check captions of Figures 8 and 9.}
	%\WM{update accordingly after Figs 8-9 haver been clarified.}
	The first mechanism involves a singular Andronov-Hopf (AH) bifurcation and a corresponding \textit{canard explosion} \cite{Dumortier1996,Krupa2001b,Kuehn2015} %, i.e.,~solutions which lie within the intersection of the (extended) slow manifolds $S_{a,\delta}$ and $S_{r,\delta}$ of Lemma \ref{lem:K1_dynamics}, VK: sentence didn't make sense to me with this i.e., in it. Presumably this is a definition of what you mean by a canard?
	and may arise in system  \eqref{eq:general_form} under variation in either $p$ or $c_t$ (but not under variation of $\tau_{\rm max}$). 
	%For these cases we simply provide a brief description and refer to the literature, leaving the detailed analysis for future work.
	%The singular nature of the AH bifurcation and the corresponding well-known {\em canard phenomenon} 
	The occurrence of this mechanism is correlated with the passage of the equilibrium $q$ through the fold point $F$ in Figure \ref{fig:manifold_R2} under parameter variation. Note, that this implies a violation of the regularity condition \eqref{eq:fold_regularity_R2}. %Specifically in system \eqref{eq:main_R2}, this can occur under variation of either $p$ or $c_t$ (but not under $\tau_{\rm max}$). VK: removed since it repeats a sentence above.
	Results in \cite{Krupa2001b,Dumortier1996} then imply the existence of a nearby singular AH bifurcation for a locally unique parameter value $p = p_{ah} + O(\delta)$ or $c_t = c_{t,ah} + O(\delta)$. We highlight that periodic orbits arising from the singular AH bifurcation %resulting from variation of  $p$ or $c_t$ 
	are expected to have an oscillation period $\mathcal T = O(\epsilon^{-7/4})$, which is intermediate between the intermediate-slow and infra-slow time-scales $t_1 = \epsilon^{3/2} t$ and $\tau_1 = \epsilon^2t$ respectively.

	\begin{figure}[t!]
		\centering
		\begin{subfigure}[b]{0.48\linewidth}
			\includegraphics[width=\linewidth]{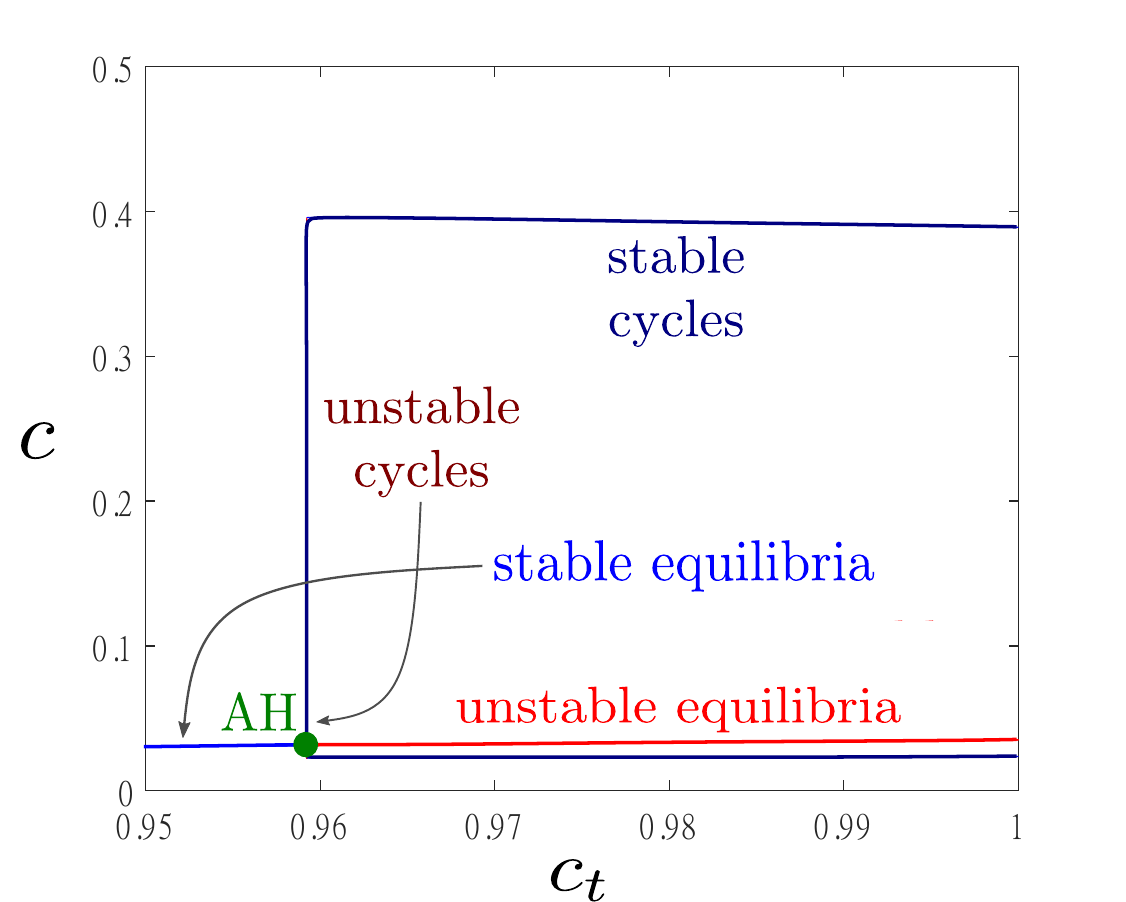}
		\end{subfigure}
		\quad
		\begin{subfigure}[b]{0.485\linewidth}
			\includegraphics[width=\linewidth]{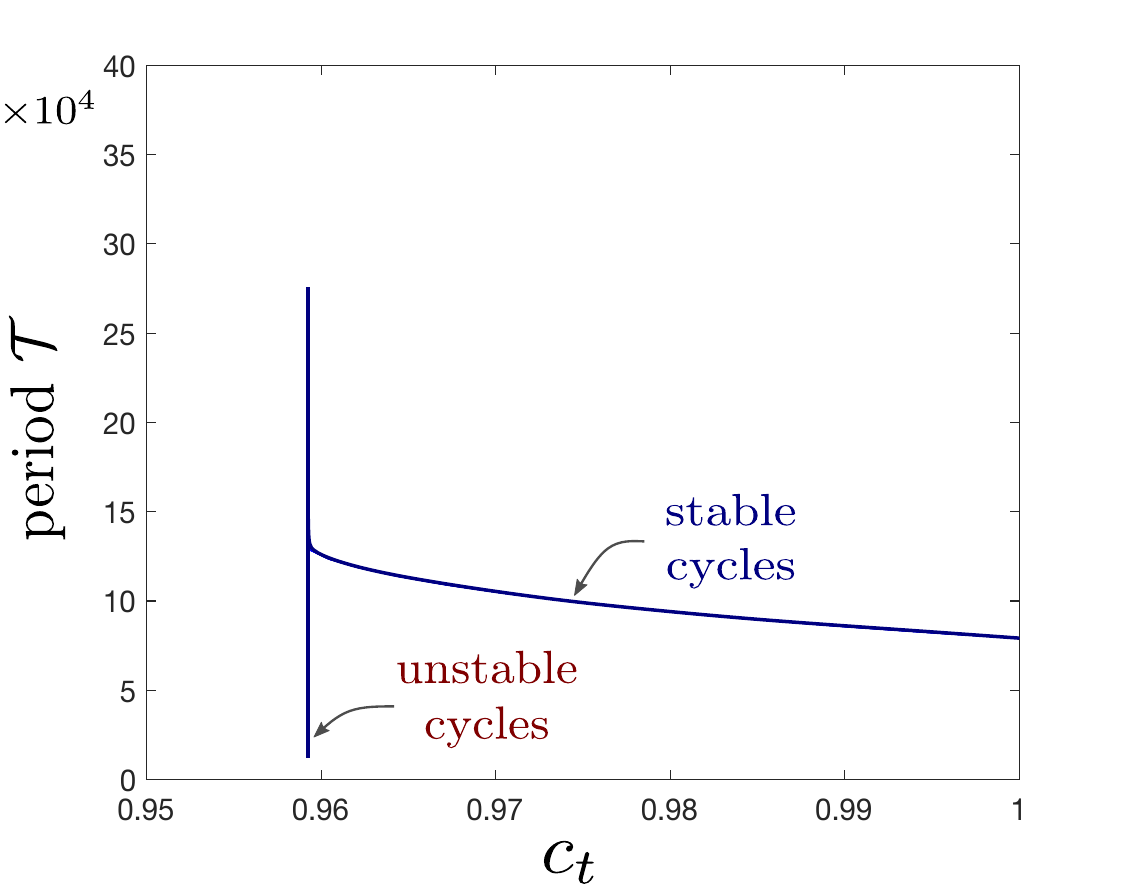}
		\end{subfigure}
		\caption{{Bifurcation diagram for system \eqref{eq:general_form} under variation of $c_t$, computed numerically in MatCont \cite{MATCONT}, for $p=0.015$  and other parameters as in Table \ref{tab:params_post_scaling}: %(Left) stable and unstable/saddle equilibria are represented in black {\color{red}{looks like light blue, not black, to me}} and red respectively. The blue branches show the maximum and minumum values of $c$ for the stable limit cycles of relaxation type. % which by Theorem \ref{thm:main} exists. 
				(Left) The onset of oscillations happens via a subcritical singular Andronov-Hopf (AH) bifurcation at $c_t\approx 0.96$. %which is subcritical (i.e.~first Lyapunov coefficient  is positive). 
				The transition to stable relaxation oscillations happens via a canard explosion in an exponentially small parameter regime.  %similar that described for autocatalytic systems in \cite{gucwa2009geometric}. 
				(Right) Period of the branch of relaxation oscillations. %{\color{red}{If I was being very picky, I would say that the labels in the left panel should say ``equilibria" since we use a plural for the periodic orbits. I'm not sure it really matters though....}}
				%The oscillation period at the onset near the singular AH bifurcation is observed to be $3  \times 10^{-4}$, which matches the theoretical prediction of  $\mathcal T \propto O(\epsilon^{-7/4}) \sim 3.6 \times 10^{-4}$  very well. 
				%The relaxation cycle have an increased period of $O(\tau_{max}\epsilon^{-2})$; see Remark~\ref{rem:num-period}. 
				%[remark: the spike in period probably indicates the position of the SNPO of canard cycles.] VK: removed this since I couldn't see it was well-justified here.
		}}
		\label{fig:can_exp}
	\end{figure}
	
	Figure~\ref{fig:can_exp} shows a bifurcation diagram for system \eqref{eq:general_form} under variation of $c_t$, with $p=0.015$ and other parameters as in Table \ref{tab:params_post_scaling}. As expected, the onset of oscillations happens via a singular AH bifurcation, at $c_t\approx 0.96$. The AH bifurcation is observed to be subcritical; the criticality has been confirmed numerically in MatCont \cite{MATCONT} by showing that the corresponding first Lyapunov coefficient of the AH bifurcation is positive. 
	
	The rapid onset of the (non-standard) relaxation oscillations described in Theorem \ref{thm:main} is expected to occur over an exponentially small interval in $c_t$ (or $p$); this is referred to as a \textit{canard explosion} \cite{Dumortier1996,Krupa2001b,Gucwa2009} because of the dramatic (or explosive) growth of amplitude. The rapid onset of oscillations is observed as an almost vertical segment in Figure~\ref{fig:can_exp}. 
	
	\begin{rem}
		Due to the subcritical nature of the singular AH bifurcation observed in Figure~\ref{fig:can_exp}, there exists a small parameter regime in which there is bistability between a stable relaxation/canard cycle and  a stable equilibrium; on the scale of Figure~\ref{fig:can_exp}, this regime is too small to be seen. The existence of this regime of bistability implies the occurrence of a saddle-node bifurcation of canard cycles (not indicated in Figure~\ref{fig:can_exp}). We refer to, e.g.,~\cite{Krupa2001b,deMaesschalck2015} for details.
		%It is expected to have a similar geometric structure to the one described in \cite{gucwa2009geometric}.} 
		%{\color{red}{We could omit this Remark without losing anything except some length.}}
	\end{rem}
	
	\begin{figure}[t!]
		\centering
		\includegraphics[scale=0.62]{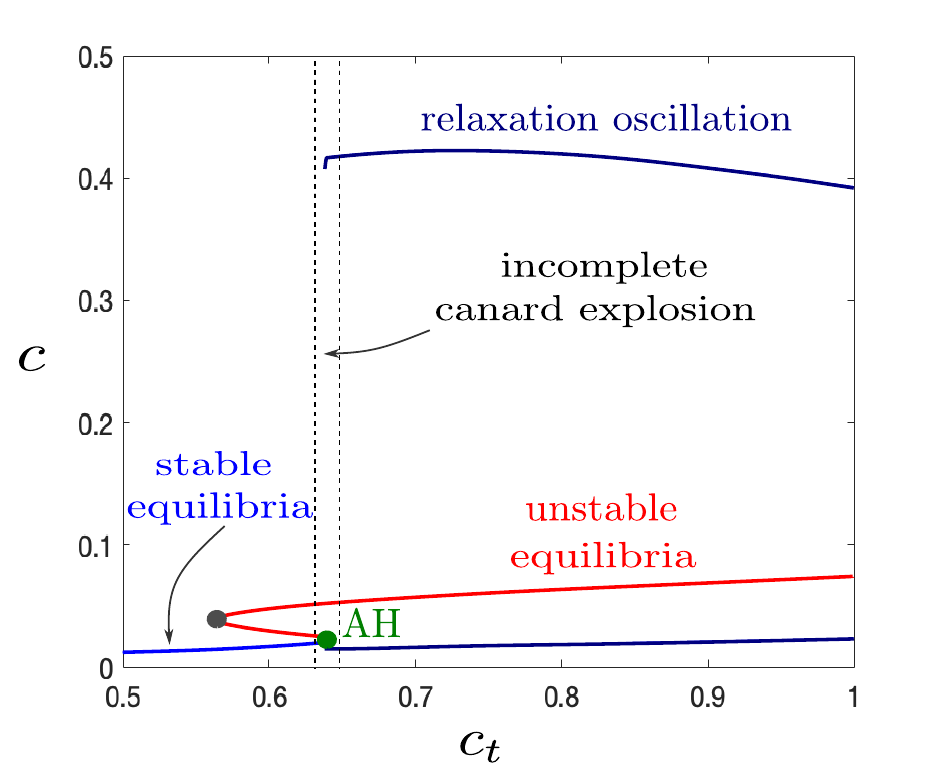}
		\caption{{Bifurcation diagram for system \eqref{eq:general_form}, computed numerically in MatCont \cite{MATCONT} with parameters as in Table \ref{tab:params_post_scaling}; compare with Figure~\ref{fig:can_exp}. %The colors indicate branch types and stabilities as in Fig.~\ref{fig:can_exp}.
				%stable and unstable/saddle equilibria are represented in black and red respectively (the insider branch is saddle-type). The blue branches correspond to max/min values associated with the stable limit cycle of relaxation type, which by Theorem \ref{thm:main} exists. 
			%	The onset of oscillations happens via a subcritical singular AH bifurcation. \rev{We were unable to connect the relaxation oscillations with the Andronov-Hopf bifurcation, where the system becomes extremely stiff. This is expected from theory: t}he transition to stable relaxation oscillations happens via an `incomplete' canard explosion involving \rev{multiple singular bifurcations and} canard homoclinic cycles; see, e.g., \cite{deMaesschalck2015} for details. 
				{\color{black}{%I find this confusing. I realise this has been modified to accommodate comments of Referee 1 but I don't think it works yet. I suggest 
				A subcritical singular AH bifurcation occurs on the middle branch of the S-shaped curve of equilibria. Theoretical considerations \cite{deMaesschalck2015} lead us to expect this AH bifurcation to produce a short unstable branch of periodic orbits but we were unable to find this numerically because of the extreme stiffness of the system. The transition to stable relaxation oscillations is similarly expected to happen via an `incomplete' canard explosion involving multiple singular bifurcations and canard homoclinic cycles; see, e.g., \cite{deMaesschalck2015} for details.}}
				}}
		%[Sam: $p$-value  must be different than in Figure 8, correct? Secondly, there must be another AH-bif sitting on the middle branch between the two SN-bif (the incomplete canard explosion regime), correct? The corresponding branch of `small' periodics terminates in a `small' homoclinic. Btw, I suspect that the `real' canard explosion in Fig 8 happens near the lower AH-bif (based on the observation seen here). Please check.]}
		%Note that the lower branch corresponding to minimal $c-$values, is close to the attracting slow manifold $\mathcal S_{a,\delta}$ in (R2).  
		%	In system \eqref{eq:general_form} we expect an explosion of type (ii) in the $c_t$ region between dashed vertical lines at the left-hand side of the relaxation oscillatory regime, and an explosion of type (i) in the $c_t$ region on the right-hand side of the relaxation oscillatory regime. We were unable to complete the diagram at the left-hand end of the relaxation oscillatory regime, were the system becomes extremely stiff. We note that the transition from stable-unstable transition occurs (on both sides) via subcritical (singular) Andronov-Hopf bifurcation.}}
		\label{fig:can_explosion}
	\end{figure}

	Figure~\ref{fig:can_explosion} shows another bifurcation diagram for system \eqref{eq:general_form} under variation of $c_t$, this time with all parameters as in Table \ref{tab:params_post_scaling}. Note the S-shaped branch of equilibria which indicates that the choice of $p=0.025$ puts the system into the regime in which there can be three equilibria; %the `three equilibria' regime. Recall that the reduced problem \eqref{eq:reduced_R2} can have up to three equilibria; 
	see Figure~\ref{fig:cusp}. There is a subcritical singular AH bifurcation at $c_t\approx 0.63$.  As before, the criticality has been confirmed numerically in MatCont \cite{MATCONT}. %by calculating the corresponding first Lyapunov coefficient of the AH bifurcation which is positive (and, hence, subcritical). 
	Importantly, the singular AH bifurcation resides on the middle branch (close to the lower fold) of the S-shaped equilibrium branch.

	The details of the canard explosion in system \eqref{eq:general_form} can differ quite significantly, depending on the number of equilibria. 
	%depending on the number of real positive roots to the cubic polynomial \eqref{eqs:reduced_equilibria_R2b}. In the simplest case, equation \eqref{eqs:reduced_equilibria_R2b} has only one real and positive root, so that for $\epsilon \ll 1$ system \eqref{eq:main_R2} has a unique equilibrium. In this case, the associated canard explosion is expected to have a similar geometric structure to the one described in \cite{gucwa2009geometric}. In the alternative (generic) scenario, \eqref{eqs:reduced_equilibria_R2b} has three real positive roots, and system \eqref{eq:general_form} has three equilibria.
	% The canard explosion mechanism in system \eqref{eq:main_R2} can differ quite significantly, however, %differs in kind to that associated with autocatalator models described in \cite{gucwa2009geometric}: 
	%since varying $p$ (or $c_t$) also varies the shape of the manifold $\mathcal S$ through parameter dependence of $\zeta(C)$. This can lead to additional (up to three) equilibria on $\mathcal S$, and an
	In the case that there are three equilibria, there is an \textit{incomplete canard explosion} in which canard cycles may terminate prematurely in a homoclinic bifurcation. Figure~\ref{fig:can_explosion} indicates the small zone where such an incomplete canard explosion happens. In this zone, the large relaxation oscillations turn into large canard cycles which terminate in a  large amplitude homoclinic bifurcation. Similarly, the small unstable oscillation cycles born out of the singular AH bifurcation turn into small canard cycles that terminate in a small amplitude homoclinic bifurcation.
	
	%This phenomenon has been described in the context of neural excitation in % in a manner similar to the canard explosion phenomenon described in \cite{Wechselberger2015}. 
	\begin{rem}
		A rigorous treatment of this incomplete canard explosion goes beyond the scope of this article and is deferred for future work. We refer to, e.g., \cite{deMaesschalck2015}, which studies that phenomenon in the context of (neural) excitation.
	\end{rem}
	%\SJJ{, however in Figure \ref{fig:can_explosion} we present numerical evidence for the presence of both explosive mechanisms under $c_t-$variation}.
	%}

	%\WM{So that canard explosion would produce $O(\epsilon^{-4/7})$ period at onset, correct? So that is different to the story in section 6. Should that be discussed with the other AH-bifurcation in section 6?}

	\subsection{\SJ{Regular Andronov-Hopf bifurcation under $\tau_{\rm max}$ variation}}
	
	%\SJJ{A second, alternative mechanism has been studied numerically in \cite{sneyd2017dynamical}, and concerns the variation of $\tau_{max}$ over several orders of magnitude. We show that after a significant (order of magnitude) decrease in $\tau_{max}$, system \eqref{eq:main_R2} becomes a \textit{regular perturbation problem} which undergoes a regular Andronov-Hopf bifurcation in regime (R2).}
	
	%\WM{So, here is my understanding: for $\tilde \tau_{max} = O(\epsilon^{-2})$ there is no AH-bif under $\hat \tau_{max}$ variation. Using the same overall model scaling but $\tilde \tau_{max} = O(\epsilon^{-3/2})$, system (5.1) turns into a standard slow-fast system with $h$ being fast. The branch $h=0$ is normally hyperbolic, but $c=0$ is degenerate. I think that explains the shape in phase space (squeezed in $c$-direction). Could be compared with Figure 7.\\
	%Zooming into $c=O(\epsilon^{1/2})$ gives a regular system described to leading order by (6.3). Here one can identify AH-bif under variation of $\hat \tau_{max}$ as you describe it in Thm 6.2.}
	
	From a physiological point of view, $\tau_{\rm max}$ is an important parameter, as it is one of the major determinants of oscillation period. For example, although we do not have a detailed understanding of oscillation frequency as a function of $\tau_{\rm max}$, we do know that an increase in $\tau_{\rm max}$ leads to a decrease in oscillation frequency \cite{sneyd2017dynamical}. \SJ{Numerical studies of the corresponding open-cell model \eqref{eq:Full_model_open_ct} in \cite{sneyd2017dynamical} showed the {\color{black}{onset of 
relaxation oscillations via a supercritical Hopf bifurcation as $\tau_{\rm max}$ is increased.}}
	%termination of relaxation oscillations %oscillations similar to those described for the closed-cell model in Theorem \ref{thm:main} 
		%in a supercritical AH bifurcation. %which occurs after a significant decrease in $\tau_{max}$; 
		%see \cite[Figure S3]{sneyd2017dynamical}. 
		These findings have been reproduced for system \eqref{eq:general_form} in Figure \ref{fig:period}, which shows the period of oscillations in system \eqref{eq:general_form} as a function of the original dimensionless quantity $\tilde \tau_{\rm max}$ from Section \ref{sub:non_dimensionalisation}. \rev{We revert to the original $\tilde \tau_{\rm max}$ in these figures and the following in order to consider parameter variations over several orders in $\epsilon$.}}
	%In considering variation over several orders of magnitude in Figure \ref{fig:period} and the following, we recall the original dimensionless notation $\tilde \tau_{max}$ from Section \ref{sub:non_dimensionalisation}.}

	\begin{figure}[t!]
		\centering
		%	\begin{subfigure}[b]{0.45\linewidth}
		%		\includegraphics[width=\linewidth]{Images/period.png}
		%	\end{subfigure}
		\begin{subfigure}[b]{0.49\linewidth}
			\includegraphics[width=\linewidth]{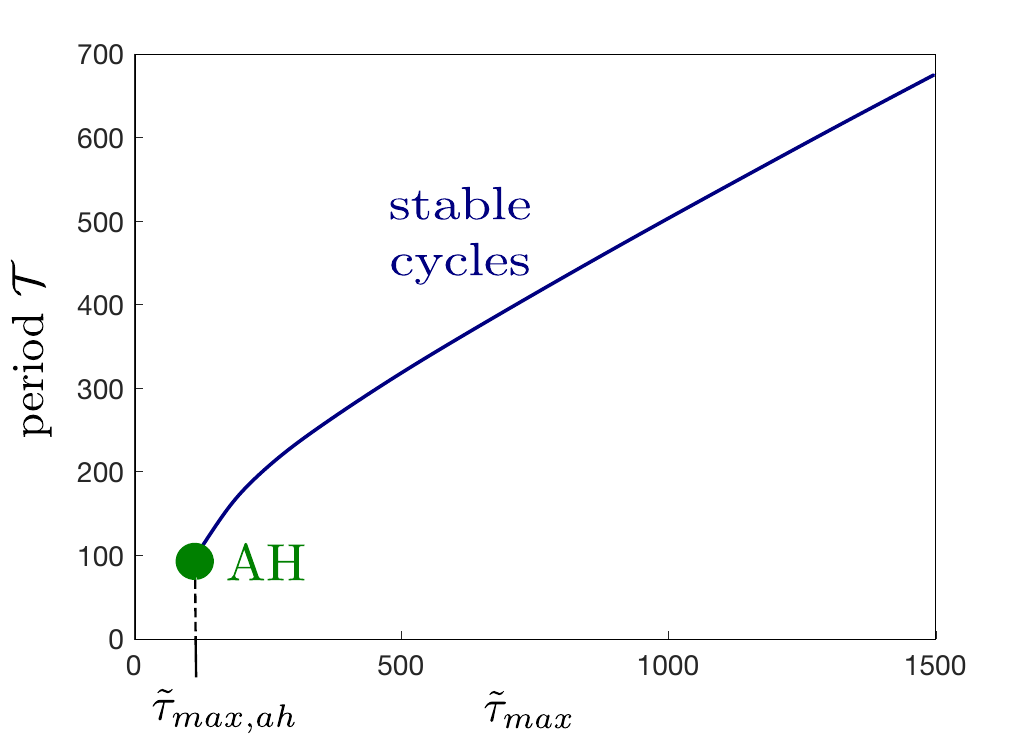}
		\end{subfigure}
		\begin{subfigure}[b]{0.49\linewidth}
			\includegraphics[width=\linewidth]{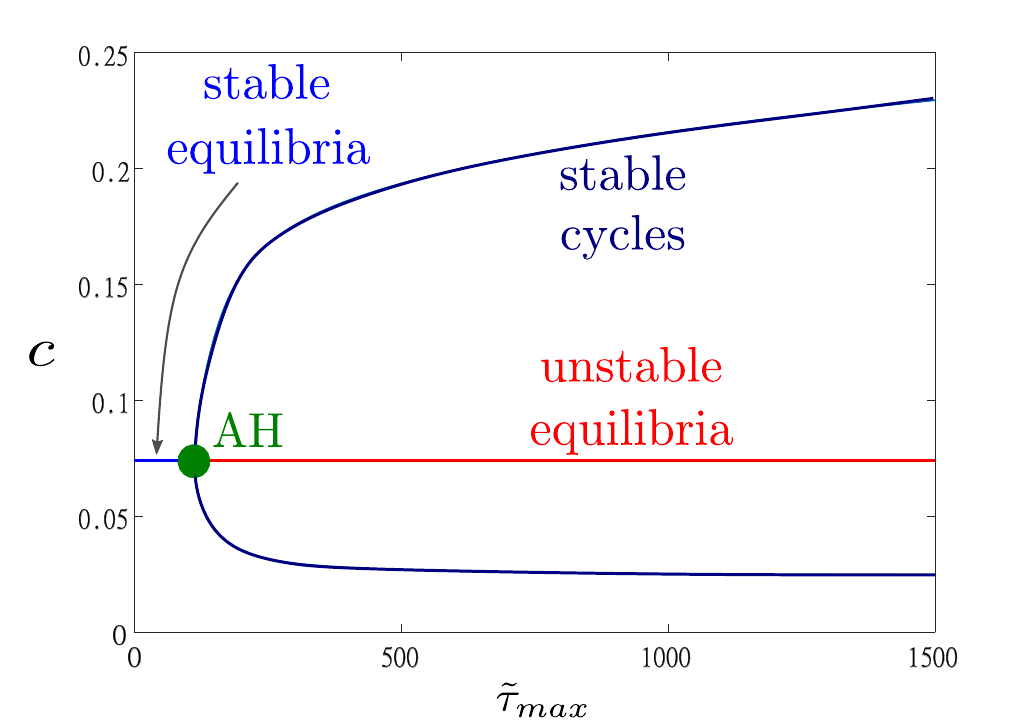}
			%	\caption{$\hat{\tau}_{max}$ = $5.4~\times~10^{4}$}
			%	\label{fig:phase_tau_max_large}
		\end{subfigure}
		\caption{Bifurcation diagrams for \eqref{eq:general_form} as a function of $\tilde \tau_{\rm max}$, calculated with MatCont \cite{MATCONT} for parameter values as in Table \ref{tab:params_post_scaling}: (Left) Period $\mathcal T$ as a function of $\tilde \tau_{\rm max}$; (Right) Maximum and minimum values of $c$ as a function of $\tilde \tau_{\rm max}$. Note that we plot with respect to $c = \sqrt \epsilon C$ here.
			The \SJ{green} disk indicates the onset of oscillations at a supercritical AH bifurcation at $\tilde \tau_{\rm max} = \tilde \tau_{\rm max,ah} \approx 112.5$ or equivalently, $\nu_{\rm max,ah} \approx 5.83\times10^{-2}$. %{\color{red}{Font sizes for tick marks etc are different in the two panels.}}
			}
		%By Theorem \ref{thm:hopf}, \SJ{$\mathcal T \propto \epsilon^{-3/2} \sqrt{\nu_{max}} =  \epsilon^{-1/2}\sqrt{\tilde \tau_{max}}$} near the onset of oscillations in the scaling regime $\tilde \tau_{max} = O(\epsilon^{-3/2})$. \SJ{Following the onset of oscillations, i.e.~for larger values of $\tilde \tau_{max}$, we observe an approximately linear dependence on $\tilde \tau_{max}$. This is consistent with Proposition \ref{prop:period}, which implies an asymptotically linear relationship} $\mathcal T \propto \tilde \tau_{max}$ in the relaxation oscillatory regime $\tilde \tau_{max} = O(\epsilon^{-2})$. VK: removed here since it is quite detailed and sits better in the body text (where it already is).
		%\SJ{(Right) corresponding bifurcation diagram in the $(\tilde \tau_{max}, c)-$plane. Note that we plot with respect to $c = \sqrt \epsilon C$ here. The equilibrium along $C = C_\ast$ is stable (unstable) for $\tilde \tau_{max}$ less than (greater than) $\tilde \tau_{max,ah}(\epsilon)$, and depicted in black (red). The green disk represents the Andronov-Hopf bifurcation, and curves of min/max $c-$values corresponding to the stable limit cycle are shown in blue.}}
		%Variation of the period of oscillations of system \eqref{eq:main1} with changes in $\hat \tau_{max}$ Left panel: behaviour for $\hat \tau_{max} = O(1)$. Right panel: Enlargement for $\hat \tau_{max}$ small. 
		\label{fig:period}
	\end{figure}

	\SJ{In order to identify the AH bifurcation analytically, we consider system \eqref{eq:main_R2} except with
	%the parameter rescaling \eqref{eq:common_scaling} except with
	\begin{equation}
	\label{eq:scale_tau_new}
	%\epsilon = \epsilon_{\tau_{\rm max}}^{2/3} \nu_{\rm max} , %\qquad \text{or} \qquad 
	%\tau_{\rm max} = \nu_{\rm max} \epsilon^{-3/2} ,
	\tau_{\rm max} = \delta \nu_{\rm max} = \epsilon^{1/2} \nu_{\rm max} ,
	\end{equation}
	}and vary $\nu_{\rm max} \in [0,\xi]$, where $\xi>0$ is fixed, as a bifurcation parameter. %This is equivalent to setting $\tau_{\rm max} = 
	%{\color{red}{This isn't easy to follow. With the rescaling, we are not really looking at the AH bifurcation in (4.11). Could this be rewritten as "In order to identify the AH bifurcation analytically, we consider the parameter rescaling  \eqref{eq:common_scaling} except with ... relaxation oscillations. In this case, equations (4.11) are replaced with a rescaled version and, within regime (R2), the dynamics is governed by ....}} VK: Okay, I don't think it is very clear the way it is written, but it is a matter of taste only, so happy to leave it unless the referees also have a problem.
	%
	The alternative scaling \eqref{eq:scale_tau_new} amounts to restriction to the scaling regime $\tilde \tau_{\rm max} = O(\epsilon^{-3/2})$, instead of the scaling regime $\tilde \tau_{\rm max} = O(\epsilon^{-2})$ corresponding to relaxation oscillations. In this case, the dynamics within regime (R2) are governed by
	\begin{equation}\label{sys:main_iii_rescaled}
	\begin{split}
	h' &= - \frac{1}{\tau_h(C)} \left(h - h_\infty(C)\right) , \\
	C' &= \tilde f_0(h,C) + \delta \tilde f_{\rm rem}(h, C, \delta) .
	\end{split}
	\end{equation}
	%System \eqref{sys:main_iii_rescaled} is obtained by the same methods as  from \eqref{eq:Non_dim} after a suitable time desingularisation which amounts to division of the right-hand-side by $\delta^4$. 
	System \eqref{sys:main_iii_rescaled} is similar to system \eqref{eq:main_R2}, except that there is no small parameter factoring the equation for $h'$. Setting $\delta = 0$ in \eqref{sys:main_iii_rescaled} yields
	\begin{equation}\label{sys:main_iii_rescaled_lo}
	\begin{split}
	h' &= - \frac{1}{\tau_h(C)} \left(h - h_\infty(C)\right) , \\
	C' &= \tilde f_0(h,C) ,
	\end{split}
	\end{equation}
	which has equilibria for $(h_\ast, C_\ast)$ satisfying $h_\ast = h_\infty(C_\ast) = \zeta(C_\ast)$,  where $\zeta(C)$ is given by \eqref{eq:S_R2}. In fact, it follows from our observations in Section \ref{ssec:singular_limit_analysis_R2} that system \eqref{sys:main_iii_rescaled_lo} can have one, two, or three equilibria, depending on the location in $(p,c_t)-$parameter space. We note that system \eqref{sys:main_iii_rescaled} is a \emph{regular perturbation problem}, with leading order dynamics determined by the limiting system \eqref{sys:main_iii_rescaled_lo}.
	
	\SJ{The observed AH bifurcation is described in the following result. Note that neither $h_\infty(C)$ nor $\zeta (C)$ depend on $\tilde \tau_{\rm max}$, so the number and location of equilibria is also independent of $\tilde \tau_{\rm max}$.}
	%As in Assumption [], we restrict attention regions of parameter space for which system \eqref{} has only a single equilbrium \textbf{[It follows from the discussion in Section \ref{} that the number of equilibria will not depend on $\tau_{\rm max}$:

	\begin{comment}
	\centering
	\begin{subfigure}[b]{0.48\linewidth}
	\includegraphics[width=\linewidth]{Images/timeseries_nondim.png}
	\end{subfigure}
	\begin{subfigure}[b]{0.48\linewidth}
	\includegraphics[width=\linewidth]{Images/timeseries_nondim_zoom.png}
	\end{subfigure}

	%We are interested in the possibility of an Andronov-Hopf bifurcation in system \eqref{sys:main_iii_rescaled_lo}. Since neither $h_\infty(C)$ nor $\zeta (C)$ depend on $\tilde \tau_{\rm max}$, the number and location of any equilibria are also independent of $\tilde \tau_{\rm max}$. %Hence, region of $(p,c_t)-$space specified in Assumption \ref{app:1} likewise corresponds to the region in which system \eqref{sys:main_iii_rescaled_lo} has a unique unstable node. 
	%We also note that the inequality $C_\ast > C_F$ is sufficient to restrict to the case in which there is a unique equilibrium, which perturbs to a nearby equilibrium $q_\delta$ in system \eqref{sys:main_iii_rescaled} for all $0 \leq \delta \ll 1$ sufficiently small. 
	
	%\
	
	%The following result describes the observed Andronov-Hopf bifurcation.}
	%Assumption 2 is no longer applicable here though, since the stability of the equilibrium is determined by a different expression which \textit{will} depend on $\tau_{\rm max}$.
	\end{comment}
	
	\begin{theorem}
	\label{thm:hopf}
	Consider system \eqref{sys:main_iii_rescaled_lo}, with a unique equilibrium $q = (h_\ast, C_\ast)$ such that
	\begin{equation}
	\label{eq:det_cond}
	C_\ast > C_F, \qquad  \big| D_C h_\infty(C_\ast) \big| > \frac{D_C \tilde f_0(q_\delta)}{D_h \tilde f_0(q_\delta)} ,
	\end{equation}
	where the coordinate $C_\ast$ can be written explicitly in terms of the model parameters. Then there exists $\delta_0 >0$ such that for each $\delta \in [0, \delta_0)$, system \eqref{sys:main_iii_rescaled} has a unique equilibrium $q_\delta$ such that $q_\delta \to q$ as $\delta \to 0$.  Furthermore, $q_{\delta}$ undergoes an AH bifurcation \SJ{for}
	\begin{equation}
	\label{eq:hopf_value}
	\nu_{\rm max,ah}(\delta) = \frac{\left(K_h^4 + C_\ast^4 \right) C_\ast }{2 A_{\text{SERCA}} K_h^4(C_\ast^2 - 2 K \gamma^2 c_t^2)} + O(\delta) ,
	%\delta w(C_\ast) + O(\delta^2) ,
	\end{equation}
	%	\SJ{[No $K_\tau$.]} where $w(C_\ast)$ is a smooth function specified in Appendix \ref{ssub:hopf_correction}, equation \eqref{eq:correction_term}. 
	or, equivalently,
	\[
	\tilde \tau_{\rm max,ah}(\epsilon) = \epsilon^{-3/2} \nu_{\rm max,ah}(\sqrt{\epsilon}) = \frac{\left(K_h^4 + C_\ast^4 \right) C_\ast }{2 A_{\text{SERCA}} K_h^4(C_\ast^2 - 2 K \gamma^2 c_t^2)} \epsilon^{-3/2} + O(\epsilon^{-1}) .
	\]
	\end{theorem}
	
	\begin{proof}
%	The proof is based on a fixed point analysis, and deferred to Appendix \ref{app:hopf_proof}. The conditions in \eqref{eq:det_cond} ensure satisfaction of the positive determinant condition associated with the AH point. The fact that $C_\ast$ can be written explicitly in terms of the model parameters follows from the fact that $C_\ast^2$ is a solution to the cubic polynomial \eqref{eqs:reduced_equilibria_R2b}.
	%
	The fact that $C_\ast$ can be written explicitly in terms of the model parameters follows from the fact that $C_\ast^2$ is a solution to the cubic polynomial \eqref{eqs:reduced_equilibria_R2b}.
	
	Equilibria of system \eqref{sys:main_iii_rescaled} satisfy $\tilde f_0(h, C) + \delta \tilde f_{\rm rem}(h, C, \delta) = 0 $. For $\delta = 0$ in particular we have $\tilde f_0(\zeta(C), C) = 0$ and $D_h \tilde f_0(\zeta(C), C) > 0$ for all $C>0$, recalling the expressions in \eqref{eq:K2_rhs}. It follows by the implicit function theorem that there exists an open interval $U$, a constant $\tilde \delta_0 >0$ and a smooth function $\varphi:U \times [0,\delta_0) \to \mathbb R$ such that for all $(C,\delta) \in U \times [0, \tilde \delta_0)$ we have
	\[
	\tilde f_0(\varphi(C,\delta) , C) + \delta \tilde f_{\rm rem}(\varphi(C,\delta) , C, \delta) = 0,
	\]
	where $\varphi(C,\delta) = \zeta(C) + \delta \psi(C,\delta)$ for a smooth function $\psi$. %Hence, the equilibrium $q_\delta = (h_\ast, C_\ast)$ satisfies
	%\[
	%h_\ast = \zeta(C_\ast) + \delta \psi(C_\ast,\delta) = h_\infty(C_\ast), \qquad \delta \in [0, \tilde \delta_0) .
	%\]
	Evaluating the Jacobian $J$ at $q_\delta$ yields trace
	\[
	\text{Tr} J(q_\delta) = - \frac{1}{\tau_h(C_\ast)} + D_C \tilde f_0(q_\delta) + \delta \tilde f_{\rm rem}(q_\delta, \delta) ,
	\]
	and determinant
	\begin{equation}
		\label{eq:det}
		\det J(q_\delta) = - \frac{1}{\tau_h(C_\ast)} \left( D_C \tilde f_0(q_\delta) + D_C h_\infty(C_\ast) D_h \tilde f_0(q_\delta) \right) + O(\delta) .
	\end{equation}
	Since we assume that $C_\ast > C_F$, the \SJ{inequalities $D_C \tilde f_0(q_\delta) > 0$ and $D_h \tilde f_0(q_\delta) > 0$ follow by equations \eqref{EV_R2} and \eqref{eq:fold_nondegeneracy}}. Since $D_C h_\infty(C_\ast)<0$ follows by direct calculations, 
	%\[
	%D_C h_\infty(C_\ast) = - \frac{4 K_h^4 C_\ast^3}{(K_h^4 + C_\ast^4)^2} < 0,
	%\]
	the determinant condition $\det J(q_\delta) > 0$ follows by the second condition in \eqref{eq:det_cond} for all $\delta$ sufficiently small. Finally, the Andronov-Hopf condition $\text{Tr} J(q_\delta) = 0$ can be solved for $\delta \in [0,\delta_0)$, $\delta_0 > 0$ sufficiently small, using the implicit function theorem. We obtain
	\begin{equation}
		\label{eq:hopf_approx}
		\nu_{\rm max,ah}(\delta) = \frac{K_h^4 + C_\ast^4}{K_h^4 D_C \tilde f_0(\zeta(C_\ast),C_\ast)} + \delta w(C_\ast) + O(\delta^2) ,
	\end{equation}
	where
	\begin{equation}
		\label{eq:correction}
		w(C_\ast) = 4 A_{\text{\text{IPR}}} \gamma c_t C_\ast^3 \psi(C_\ast,0) - 5 A_{\text{\text{IPR}}} C_\ast^4 \zeta(C_\ast) - 2 K A_{\text{SERCA}} \gamma^2 c_t .
	\end{equation}
	The function $\psi(C_\ast, 0)$ can be approximated by standard matching techniques in order to obtain an explicit higher order correction, but we omit this calculation for brevity. %; see Appendix \ref{ssub:hopf_correction}. 
	Expression \eqref{eq:hopf_value} follows from \eqref{eq:hopf_approx} after substituting the expression in \eqref{EV_R2} for $D_C \tilde f_0(\zeta(C_\ast),C_\ast)$.
	\end{proof}

	Numerical observations indicate a supercritical bifurcation and hence, the existence of nearby stable oscillations; see Remark \ref{rem:hopf_numerics} below. % for $\tau_{\rm max} \in (\tau_{\rm max,ah}, \tau_{\rm max,ah} + k_+)$ for some $k_+ > 0$. 
	%We do not consider whether the small-amplitude cycles connect with the large amplitude relaxation oscillations observed for $\tau_{\rm max} = O(\epsilon^{-2})$ in this work, however Figure \ref{fig:period} provides numerical support for this claim; see also Remark \ref{rem:connection} below. 
	The qualitative shape of the profile in Figure \ref{fig:period} can also be explained by our findings. \SJ{Let $\mathcal T$ denote the oscillation period on the fastest time-scale $t$, as in Proposition \ref{prop:period}. Then}
	
	\smallskip
	
	\begin{enumerate}
	\item[(i)] \SJ{The fixed point analysis used to prove Theorem \ref{thm:hopf}, % in Appendix \ref{app:hopf_proof}, 
		in particular the form of the determinant \eqref{eq:det},} \SJ{implies an oscillation period \SJ{$\mathcal T \propto \epsilon^{-3/2} \sqrt \nu_{\rm max} = \epsilon^{-1/2} \sqrt{\tilde \tau_{\rm max}}$} % and leading order dependence $\mathcal T \propto \sqrt{\tilde \tau_{\rm max}}$ 
	for parameter values $\tilde \tau_{\rm max} = O(\epsilon^{-3/2})$ close to \SJ{the AH bifurcation.} %\propto \sqrt{\tilde \tau_{\rm max}}$ 
	%	the Andronov-Hopf bifurcation;
	\item[(ii)] Proposition \ref{prop:period} implies an oscillation period $\mathcal T = O(\epsilon^{-2})$ and leading order dependence $\mathcal T \propto \tilde \tau_{\rm max}$ for parameter values $\tilde \tau_{\rm max} = O(\epsilon^{-2})$, in the relaxation oscillatory regime.}
	\end{enumerate}

	\begin{remark}
	\label{rem:hopf_numerics}
	The conditions for the applicability of Theorem \ref{thm:hopf} are satisfied for system \eqref{sys:main_iii_rescaled} with parameter values taken from Table \ref{tab:params_post_scaling} and $\delta = \sqrt{2.5 \times 10^{-3}} = 5 \times 10^{-2}$. The conditions in \eqref{eq:det_cond} are satisfied since
	\[
	\big| D_C h_\infty(C_\ast) \big| \approx 0.20 > \frac{D_C \tilde f_0(q_\delta)}{D_h \tilde f_0(q_\delta)} \approx 8.3 \times 10^{-2} ,
	\]
	and the AH bifurcation occurs at $q_\delta = (h_\ast, C_\ast) \approx (0.0792, 1.48 )$, so $C_\ast > C_F \approx 0.68$. We obtain a numerical estimate of $\nu_{\rm max,ah} \approx 5.83\times 10^{-2}$, %\approx 1.4 \cdot 10^{-2}$ 
	in close agreement with the leading order estimate \eqref{eq:hopf_value} as $\delta \to 0$:
	\[
	\nu_{\rm max,ah}(0) = \frac{\left(K_h^4 + C_\ast^4 \right) C_\ast }{2 A_{\text{SERCA}} K_h^4 (C_\ast^2 - 2 K \gamma^2 c_t^2)} %+ \delta w(C_\ast) + O(\delta^2) 
	\approx 5.42\times 10^{-2} ,
	\]
	where we have used the %expression for $w(C_\ast)$ in equation \eqref{eq:correction_term} and the
	numerical value for $C_\ast$. A first Lyapunov coefficient $l_1 \approx -3.13 < 0$ was calculated numerically in the software package MatCont \cite{MATCONT}, indicating a supercritical bifurcation.
	\end{remark}

\end{document}